\documentclass[11pt,a4paper]{article}
\usepackage[utf8]{inputenc}
\usepackage{amsfonts,amssymb,amsmath,amsthm,indentfirst}
\usepackage{epsfig}
\usepackage{enumerate}
\usepackage{amsfonts}
\usepackage{latexsym}
\usepackage{amssymb}
\usepackage{amsmath}
\usepackage{textcomp}
\usepackage{mathrsfs}
\usepackage{graphicx}
\usepackage{ulem}
\usepackage{appendix}
\usepackage{fancyhdr,ifthen}
\usepackage{enumitem}
\usepackage{changepage} 

\usepackage{xcolor}
\usepackage{geometry}
\usepackage{hyperref}
\hypersetup{colorlinks=true}
\hypersetup{
    colorlinks=true,
    linkcolor=blue,
    filecolor=blue,
    urlcolor=magenta,
    anchorcolor=blue,
    citecolor={red},
}
\makeatletter
\newcommand{\refcheckize}[1]{%
  \expandafter\let\csname @@\string#1\endcsname#1%
  \expandafter\DeclareRobustCommand\csname relax\string#1\endcsname[1]{%
    \csname @@\string#1\endcsname{##1}\wrtusdrf{##1}}%
  \expandafter\let\expandafter#1\csname relax\string#1\endcsname
}
\makeatother

\refcheckize{\cref}
\refcheckize{\Cref}

\numberwithin{equation}{section}

\newtheorem{theorem}{Theorem}[section]

\newtheorem{proposition}[theorem]{Proposition}
\newtheorem{remark}[theorem]{Remark}

\newcommand{\norm}[1]{\left\Vert#1\right\Vert}
\newcommand{\abs}[1]{\left\vert#1\right\vert}
\newcommand{\R}{\mathbb{R}}

\newcommand{\fint}{\,-\mspace{-19.4mu}\int} 

\title{\textbf{Existence of nested polygonal vortex patches for the generalized SQG equation}}

\begin{document}

\author{}  
\date{}
\maketitle
\vskip -1.0cm
\centerline{
         {\large Edison Cuba}\footnote{Department of Mathematics, State University of Campinas, Rua S\'{e}rgio Buarque de Holanda, 651, Cidade Universit\'{a}ria, 13083-859, Campinas-SP, Brazil;
{\url{ ecubah@ime.unicamp.br}} }
     {\large and Lucas C. F. Ferreira}\footnote{Department of Mathematics, State University of Campinas, Rua S\'{e}rgio Buarque de Holanda, 651, Cidade Universit\'{a}ria, 13083-859, Campinas-SP, Brazil;
{\url{ lcff@ime.unicamp.br}} (Corresponding Author)}}

\vskip 1.3cm

\begin{abstract}
\small{\noindent This paper investigates time-periodic solutions of both the surface quasi-geostrophic (SQG) equation and its generalized form (gSQG) within the more singular regime, focusing on the evolution of patch-type structures. Assuming the underlying point vortex equilibrium satisfies a natural nondegeneracy condition, we employ an implicit function argument to construct families of co-rotating nested polygonal vortex patch solutions. These configurations provide precise asymptotic descriptions of the geometry of the evolving patch boundaries. Our results contribute to the broader understanding of coherent rotating structures arising in active scalar equations with singular velocity coupling.

\vskip 6mm

}
\end{abstract}

\vspace{3mm}

\tableofcontents



\pagestyle{fancy} \fancyhf{} \renewcommand{\headrulewidth}{0pt}
\chead{\ifthenelse{\isodd{\value{page}}}{Existence of nested polygonal patches for gSQG}{E. Cuba and L. C. F. Ferreira }} \rhead{\thepage}

\newpage

\section{Introduction}\label{section1}
The generalized surface quasi-geostrophic (abbr. gSQG) equation is a two-dimensional partial differential equation that appears in atmospheric science, where
it models evolution of the temperature on the surface of a planet. This study focuses on the dynamics of $2m+1$-fold vortex patch solutions for the two-dimensional gSQG equation. The equation is formulated as follows
\begin{equation}\label{1-1}
	\begin{cases}
		\partial_t\theta + v \cdot \nabla \theta = 0, & \text{in } \mathbb{R}^2 \times (0,T), \\
		v = \nabla^\perp \psi, & \\
		\psi = -(-\Delta)^{-1+\frac{\alpha}{2}} \theta, & \\
		\theta|_{t=0} = \theta_0, & \text{in } \mathbb{R}^2,
	\end{cases}
\end{equation}
where \( 0 \leq \alpha < 2 \), and the operator \( \nabla^\perp \) denotes a counterclockwise rotation by an angle of \( \pi/2 \) in the plane. In this formulation, the scalar function \( \theta: \mathbb{R}^2 \times (0,T) \to \mathbb{R} \) represents the potential temperature and \( v: \mathbb{R}^2 \times (0,T) \to \mathbb{R}^2 \) denotes the velocity field.
\vskip2mm

\noindent
For \( 0 \leq \alpha < 2 \), the nonlocal operator \( (-\Delta)^{-1+\frac{\alpha}{2}} \) is defined by
\begin{equation*}
  (-\Delta)^{-1+\frac{\alpha}{2}} \theta(\boldsymbol{x}) = \int_{\mathbb{R}^2} K_\alpha(\boldsymbol{x} - \boldsymbol{y}) \, \theta(\boldsymbol{y}) \, d\boldsymbol{y},
\end{equation*}
where \( K_\alpha \) denotes the fundamental solution of the operator \( (-\Delta)^{-1 + \frac{\alpha}{2}} \) in the plane. The kernel \( K_\alpha \) takes the form
\begin{equation*}
  K_\alpha(\boldsymbol{x}) =
  \begin{cases}
    \displaystyle\frac{1}{2\pi} \log\left( \frac{1}{|\boldsymbol{x}|} \right), & \text{if } \alpha = 0, \\[8pt]
    \displaystyle\frac{C_\alpha}{2\pi} \cdot \frac{1}{|\boldsymbol{x}|^\alpha}, & \text{if } 0 < \alpha < 2,
  \end{cases}
\end{equation*}
with the constant \( C_\alpha \) given by
\begin{equation*}
  C_\alpha = \frac{\Gamma\left( \frac{\alpha}{2} \right)}{2^{1-\alpha} \Gamma\left( \frac{2 - \alpha}{2} \right)},
\end{equation*}
where \( \Gamma \) denotes the Euler Gamma function.

\vskip2mm

\noindent The system reduces to the classical two-dimensional incompressible Euler equations when \( \alpha = 0 \), and corresponds to the surface quasi-geostrophic (SQG) equation when \( \alpha = 1 \). The full range \( \alpha \in [0,2) \) has been progressively studied in the literature: the cases \( \alpha \in [0,1] \) were analyzed by C\'{o}rdoba et al.~\cite{cordova}, while the more singular regime \( \alpha \in [1,2) \) was explored by Chae et al.~\cite{chae}. The present work focuses on this latter regime.
\vskip2mm

\noindent It has been well established since the seminal work of Yudovich~\cite{Yud} that any bounded and integrable initial datum  gives rise to a unique global-in-time weak solution of~\eqref{1-1}, see also~\cite{Bur2005MR2186035}.  Thanks to Yudovich’s theory, one can rigorously analyze vortex patches, which correspond to vorticity uniformly distributed over a bounded region \( D \), that is, \( \theta_0 = 1_D \). Since the vorticity is transported along the flow, it follows that \( \theta(x, t) = 1_{D_t} \), where \( D_t \) evolves in time and can be described via the contour dynamics formulation.  The global-in-time persistence of boundary regularity is currently established only for the case \( \alpha = 0 \), as initially proven by Chemin~\cite{C}; see also the alternative proof by Bertozzi and Constantin~\cite{B-C}.  The analysis becomes significantly more delicate for the generalized surface quasi-geostrophic (gSQG) equation when \( \alpha \in (0,2) \). In this regime, only local-in-time persistence of regularity in Sobolev spaces has been rigorously established; see~\cite{chae,Gan08,Rod05}. Furthermore, numerical simulations suggest the potential for finite-time singularity formation in this range of \( \alpha \); see~\cite{C,SD19}. In addition, finite-time singularities have been rigorously demonstrated for multi-signed vortex patches in a half-plane configuration for various values of \( \alpha \); see~\cite{KRYZ,KYZ17,Gancedo-1}. Resnick established the existence of global weak solutions in \( L^2 \) for the SQG equation in \cite{Resnick}, exploiting a special cancellation property stemming from the oddness of the Riesz transform. This result was subsequently generalized to the gSQG setting by Chae et al.\ in \cite{chae}. Nevertheless, the question of nonuniqueness remains open (see \cite{vicol} and references therein). More recently, a number of ill-posedness results have been established in both H\"older and Sobolev spaces, addressing issues related either to the evolution of patch boundaries or to the regularity of the initial data; see~\cite{Cor-Zo21,Cor-Zo22,KL23a,KL23}.

\vskip2mm
\noindent For \( 0 < \alpha < 2 \), the velocity field can be recovered using the Biot–Savart law in conjunction with the Green–Stokes theorem, leading to the representation
\begin{equation*}
	v(\boldsymbol{x}, t) =
	\frac{C_\alpha}{2\pi} \int_{\partial D_t} \frac{1}{|x - y|^\alpha} \, d\boldsymbol{y}, \quad 0 < \alpha < 2.
\end{equation*}
Assuming that the patch boundary \( \partial D_t \) is described by a smooth, periodic parametrization \( z(\xi,t) \), with \( \xi \in [0, 2\pi) \), one obtains the contour dynamics formulation. In the range \( 0 < \alpha < 1 \), the evolution of the interface is governed by
\begin{equation*}
	\partial_t z(\xi,t) =
	\frac{C_\alpha}{2\pi} \int_0^{2\pi} \frac{\partial_\tau z(\tau,t)}{|z(\xi,t) - z(\tau, t)|^\alpha} \, d\tau, \quad 0 < \alpha < 1 .
\end{equation*}
For the more singular case \( 1 \leq \alpha < 2 \), the integral is divergent. To address this, one subtracts the tangential component of the velocity, yielding the  contour dynamics equation
\begin{equation*}
	\partial_t z(\xi,t) =
	\frac{C_\alpha}{2\pi} \int_0^{2\pi} \frac{\partial_\tau z(\tau,t) - \partial_\xi z(\xi,t)}{|z(\xi,t) - z(\tau, t)|^\alpha} \, d\tau, \quad 1 \leq \alpha < 2.
\end{equation*}

\vskip2mm
\noindent\noindent
This work investigates a special type of solution called \textit{rotating patches} (also known as relative equilibria or V-states). These patches maintain their shape over time, evolving through rigid body rotation. If we assume the rotation's center is at the origin, their evolution is described by
\begin{equation*}
  \theta(t,\mathbf{x}) = \mathbf{1}_{D_t}(\mathbf{x}),
  \quad \text{with} \quad D_t = Q_{\Omega t} D.
\end{equation*}
The notation \( D \subset \mathbb{R}^2 \)
  refers to a  bounded, smooth domain embedded in the plane. The symbol \( \Omega \in \mathbb{R} \) represents the angular velocity of this rotating domain, a crucial factor that will act as a bifurcation parameter in our upcoming analysis. Furthermore, the patch is said to be \textit{\( m \)-fold symmetric} if it satisfies the rotational invariance condition \( Q_{2\pi/m} D = D \) for some \( m \in \mathbb{N}^\ast \).


\vskip2mm
\noindent The existence and regularity of stationary or uniformly rotating vortex patches for two-dimensional active scalar equations have been the subject of extensive study over the past several decades. For the two-dimensional incompressible Euler equations, the first explicit nontrivial example of a rotating patch was discovered by Kirchhoff~\cite{Kirch}. The author established that an ellipse with semi-axes \( a \) and \( b \) rotates uniformly with an angular velocity of  \( ab/(a+b)^2 \). Numerical evidence for the existence of additional \( m \)-fold symmetric vortex patches (where  \( m \geq 2 \)) was provided by Deem and Zabusky~\cite{DZ78} in the 1970s.
Subsequently, Burbea~\cite{Bur} provided an analytical proof of the existence of \( V \)-states using  local bifurcation theory, constructing local branches of solutions that bifurcate from the Rankine vortex  (i.e., \( \mathbf{1}_{\mathbb{D}}(\mathbf{x}) \)) at the critical angular velocity \( \Omega = \frac{m-1}{2m} \). This method was later revisited by Hmidi, Mateu, and Verdera~\cite{hmidi2}. They not only offered deeper insights into the method but also rigorously demonstrated that the boundaries of these vortex patches are remarkably $C^\infty$-smooth and convex when observed near the bifurcation point. Further investigations into the analyticity of these boundaries were subsequently conducted by Castro, C\'ordoba, and G\'omez-Serrano~\cite{CCG16b}. Meanwhile, the broader global bifurcation structure of these solutions became the focus of work by Hassainia, Masmoudi, and Wheeler~\cite{Has2020MR4156612}.  Moreover, disconnected and nontrivial vortex patches have been constructed using various analytical and numerical methods; see, for instance,~\cite{Gar2021MR4284365, multipole, Hmidi-Mateu}. For additional results concerning simply connected vortex patches in the planar Euler equation, we refer the reader to \cite{del2016MR3507551,Serr,DHHM,HMV,DHHM,HH2} and references therein.


\vskip2mm
\noindent

\noindent
\noindent
    The existence of simply connected \( m \)-fold symmetric vortex patches for the generalized surface quasi-geostrophic (gSQG) equation with \( \alpha \in (0,1) \) was first established by Hassainia and Hmidi~\cite{HH15}, and further developed by De la Hoz, Hassainia, and Hmidi~\cite{de1}. These constructions provide global-in-time, time-periodic solutions in the patch setting, although the broader question of global regularity for general gSQG patches remains open. The approach is based on the Crandall–Rabinowitz local bifurcation theory, yet the required spectral analysis is notably more delicate than in the classical 2D Euler case. Later, Castro, C\'ordoba, and G\'omez-Serrano~\cite{Cas1} extended the existence results to the regime \( \alpha \in [1,2) \), demonstrating the existence of convex, smooth, simply connected vortex patches and proving that their boundaries are \( C^\infty \)-regular. In contrast to the Euler equation, elliptical patches do not serve as rotating solutions for \(\alpha > 0\), a finding detailed in \cite{serrano3}. Furthermore, through a rigorous computer-assisted proof, the existence of convex solutions that eventually lose their convexity in finite time has been confirmed. For results concerning the real analyticity of patch boundaries, see also~\cite{CCG16b}. We refer the reader to~\cite{cao2,Gar2021MR4284365,Gom,Hmidi-Mateu} for additional developments related to vortex patch solutions of the gSQG equation.


\vskip2mm
\noindent While most of the analytical studies focus on connected vortex patches, the existence of traveling and rotating vortex pairs has also attracted significant interest, particularly from a numerical perspective. The earliest numerical explorations of traveling symmetric pairs of simply connected patches in the Euler equations were carried out by Deem and Zabusky~\cite{DZ78}. This was followed by the work of Saffman and Szeto~\cite{SS}, who studied symmetric co-rotating vortex pairs, and by Dritschel~\cite{Dritschel}, who investigated asymmetric patch pairs. On the analytical side, Turkington~\cite{T} provided a variational proof of the existence of $N$-fold co-rotating vortex patches. This approach was subsequently extended by Keady~\cite{Keady}, who demonstrated the existence of translating symmetric vortex pairs, and by Wan~\cite{Wan}, who analyzed the existence and stability of desingularized solutions arising from systems of rotating point vortices. These results were further generalized to the setting of the generalized surface quasi-geostrophic (gSQG) equations with $\alpha \in (0,2)$ by Godard-Cadillac, Gravejat, and Smets~\cite{GGS}.

\vskip2mm

\noindent Although the variational approaches developed in~\cite{T,Keady,GGS} establish the existence of vortex pairs, they offer limited insight into the actual shape of the patches. For the range $\alpha \in [0,1)$, a more direct approach was obtained by Hmidi and Mateu~\cite{Hmidi-Mateu}, who constructed co-rotating and counter-rotating vortex pairs by desingularizing the contour dynamics equations and applying the implicit function theorem. This methodology was further developed and adapted to a wide range of vortex configurations, such as asymmetric vortex pairs~\cite{HH2}, Kármán vortex streets~\cite{G-Kar}, vortex patches associated with   doubly connected V-states for $\alpha \in [0,2)$ as detailed in \cite{cao4}, and V-states in bounded domains for the Euler equations~\cite{cao1}. These developments were further extended in~\cite{cao3} to the case of vortex–wave system. Other studies have considered concentrated multi-vortex structures arranged at the vertices of regular $n$-polygons or following periodic spatial patterns; see, for instance,~\cite{Gar2021MR4284365,Garcia,multipole}. In the regime $\alpha \in [1,2)$, new existence results for symmetric co-rotating and traveling vortex patches were obtained by Cao, Qin, Shan, and Zou~\cite{cao2}. More recently, the existence of asymmetric co-rotating and counter-rotating vortex pairs in this range has also been rigorously established in~\cite{edison2}.


\vskip2mm

\noindent The goal of this paper is to establish the existence of time-periodic, co-rotating solutions consisting of finitely many nested polygonal vortex patches for the generalized surface quasi-geostrophic (gSQG) equations \eqref{1-1}, in the range \(\alpha \in [1,2)\). We focus on arbitrary configurations built from nested regular polygons and exploit their symmetries by incorporating them into the formulation of the problem. Assuming a natural non-degeneracy condition on the underlying point vortex configuration, we apply a suitably adapted version of the implicit function theorem to construct such solutions. Our findings extend and complement the results of \cite{cao2,Cas1} by addressing highly symmetric nested structures beyond previously studied cases. For comparison, see \cite{serrano4} for the existence of stationary multi-layered vortex patch solutions in the Euler case (\(\alpha = 0\)), and \cite{multipole} for nested polygonal solutions to the gSQG equations in the regime \(\alpha \in [0,1)\).

\vskip2mm

\noindent The subsequent theorem, which fully encompasses both Theorem~\ref{thm:polygon} and Theorem~\ref{thm:polygon2}, provides the proof for the existence of co-rotating nested polygons when  $1\leq \alpha<2$.

\begin{theorem}\label{thm:informal-1polygon}
Let $\alpha \in [1,2)$, $\vartheta \in \{0,1\}$, $b_1, b_2 \in (0,1)$, $\gamma_0, \gamma_1 \in \mathbb{R} \setminus \{0\}$, and $d_1, d_2 \in (0,\infty)$. Suppose that  the pair $(\Omega^*, \gamma_2^*) \in (\mathbb{R} \setminus \{0\})^2$ solves system~\eqref{syst-pj2} and satisfies the non-degeneracy conditions ~\eqref{non-deg} and ~\eqref{det-j-poly}  in the case $1<\alpha<2$ and conditions \eqref{non-degb} and ~\eqref{det-j-polyb} when $\alpha=1$.  Then, for all sufficiently small  $\varepsilon > 0$, there exist three strictly convex domains $\mathcal{O}^\varepsilon_0$, $\mathcal{O}^\varepsilon_1$, and $\mathcal{O}^\varepsilon_2$, which are $C^{1+\beta}$-perturbations of the unit disc, along with a real parameter $\gamma_2 = \gamma_2(\varepsilon)$, such that the function
\begin{equation*}
\theta_{0,\varepsilon} = \frac{\gamma_0}{\varepsilon^2 b_0^2} \chi_{\varepsilon b_0 \mathcal{O}_0^\varepsilon} + \sum_{j=1}^{2} \frac{\gamma_j}{\varepsilon^2 b_j^2} \sum_{k=0}^{m-1} \chi_{\mathcal{D}_{jk}^{\varepsilon}},
\end{equation*}
with the sets $
\mathcal{D}_{jk}^\varepsilon$  defined by
\begin{equation}
\mathcal{D}_{jk}^\varepsilon=\begin{cases}
    Q_{\frac{2k\pi}{m}}(\varepsilon b_1 \mathcal{O}_1^\varepsilon + d_1) & \text{if } j=1, \\
Q_{\frac{(2k+\vartheta)\pi}{m}}(\varepsilon b_2 \mathcal{O}_2^\varepsilon + d_2) & \text{if } j=2,
\end{cases}
\end{equation}
produces a uniformly rotating solution to equation~\eqref{1-1} with angular velocity $\Omega(\varepsilon)$. Additionally, the domain
 $\mathcal{O}_0^\varepsilon$ exhibits $m$-fold symmetric, while $\mathcal{O}_1^\varepsilon$ and $\mathcal{O}_2^\varepsilon$ possess $1$-fold symmetry.
\end{theorem}

\vskip1mm
\noindent
The remainder of the paper is structured as follows. In Section~\ref{section2}, we derive the boundary equations governing the dynamics of \(2m+1\) time-periodic multipolar vortex patches. This section also introduces the fundamental function spaces that form our analytical framework. Section~\ref{section3} is devoted to the study of key qualitative properties of the functional \(\mathcal{F}^\alpha(\lambda)\), with a focus on the regularity of associated vortex equilibrium functionals and the linearization near equilibrium configurations. In Section~\ref{section4}, we prove the existence of nested polygonal vortex patch solutions for the generalized surface quasi-geostrophic (gSQG) equation with \(\alpha \in (1,2)\), as stated in Theorem~\ref{thm:polygon}. Section~\ref{section5} presents the proof of Theorem~\ref{thm:polygon2}, which establishes the existence of \(2m+1\) nested polygonal vortex patches in the classical SQG setting. As direct applications of Theorems~\ref{thm:polygon} and~\ref{thm:polygon2}, we deduce Theorem~\ref{thm:informal-1polygon}, which demonstrates the existence of co-rotating nested polygonal vortex patch solutions for both the two-dimensional SQG and gSQG equations in the full plane for \(\alpha \in (1,2)\). Additionally, we investigate the convexity of each patch by computing the curvature and derive asymptotic expressions characterizing the geometry of the patch boundaries.


\section{Boundary equations}\label{section2}

In this section, we investigate a configuration consisting of a finite number of point vortices exhibiting co-rotating motion for the gSQG equation. Building on the analytical framework developed in \cite{Cas1,edison2,multipole}, we first derive the contour dynamics equations that characterize the steady states of \(2m+1\) vortex patches governed by the gSQG equation \eqref{1-1} for \( 1 \leq \alpha < 2 \). We then introduce the appropriate functional spaces to ensure the existence  of solutions to  the problem \eqref{1-1}. Finally, we establish our main result, Theorem \ref{thm:informal-1polygon}, by employing a  a modified version of the implicit function theorem.

\smallskip

\subsection{Co-rotating nested polygons patches}

We begin by defining our domains. Let $m \in \mathbb{N}$ be a positive integer. We introduce three bounded, simply connected domains $\mathcal{O}_0^\varepsilon$, $\mathcal{O}_1^\varepsilon$, and $\mathcal{O}_2^\varepsilon$ each of which contains the origin and lies entirely within the ball $B(0,2)$.
A crucial property is that $\mathcal{O}_0^\varepsilon$ exhibits $m$-fold rotational symmetry. This means its shape remains unchanged under rotations by integer multiples of $2\pi/m$
\begin{equation}\label{Nsymm}
Q_{2\pi n/m} \mathcal{O}_0^\varepsilon = \mathcal{O}_0^\varepsilon, \quad \text{for all } n \in \mathbb{Z},
\end{equation}
where $Q_\lambda$ is the counterclockwise rotation operator. Furthermore, we assume that all three domains ($\mathcal{O}_0^\varepsilon$, $\mathcal{O}_1^\varepsilon$, and $\mathcal{O}_2^\varepsilon$) are symmetric with respect to the real axis. With these domains established, and given positive constants $b_0, b_1, b_2$ for scaling, $d_1, d_2$ for translation, and a sufficiently small positive parameter $\varepsilon \in (0, \varepsilon_0)$, we now define their scaled and translated counterparts as follows
\begin{align}
\mathcal{D}_{00}^\varepsilon &:= \varepsilon b_0 \mathcal{O}_0^\varepsilon, \label{Dj0} \\
\mathcal{D}_{1j}^\varepsilon &:= Q_{\frac{2j\pi}{m}} \left( \varepsilon b_1 \mathcal{O}_1^\varepsilon + d_1\mathbf{e}_1 \right), \quad j = 0, \ldots, m-1, \label{Dj1} \\
\mathcal{D}_{2j}^\varepsilon &:= Q_{\frac{(2j+\vartheta)\pi}{m}} \left( \varepsilon b_2 \mathcal{O}_2^\varepsilon + d_2\mathbf{e}_1 \right), \quad j = 0, \ldots, m-1, \label{Dj2}
\end{align}
where \( \varepsilon_0 > 0 \) is chosen sufficiently small to ensure that all sets \( \mathcal{D}_{1j}^\varepsilon \) and \( \mathcal{D}_{2j}^\varepsilon \) are pairwise disjoint. Let \( \gamma_0, \gamma_1, \gamma_2 \in \mathbb{R} \setminus \{0\} \). Consider the initial vorticity given by
\begin{equation}\label{intial-vort}
\theta_{0}^\varepsilon = \frac{\gamma_0}{\varepsilon^2 b_0^2} \chi_{\mathcal{D}_{00}^\varepsilon}
+ \frac{\gamma_1}{\varepsilon^2 b_1^2} \sum_{j=0}^{m-1} \chi_{\mathcal{D}_{1j}^\varepsilon}
+ \frac{\gamma_2}{\varepsilon^2 b_2^2} \sum_{j=0}^{m-1} \chi_{\mathcal{D}_{2j}^\varepsilon}.
\end{equation}
Note that then the vorticity in \eqref{intial-vort} converges to a  point vortex model,  when $\varepsilon \to 0$   and $|\mathcal{O}_j^\varepsilon| \to |\mathbb{D}|$, for $j=0,1,2$
\begin{equation}\label{q0}
\theta^0_0(x) = \pi \left( \gamma_0 \delta_{x_0(0)}(x)
+ \gamma_1 \sum_{k=0}^{m-1} \delta_{x_{1k}(0)}(x)
+ \gamma_2 \sum_{k=0}^{m-1} \delta_{x_{2k}(0)}(x) \right),
\end{equation}
where the vortex locations are given by
\begin{equation}\label{z}
x_0(0) = 0, \quad
x_{jk}(0) :=
\begin{cases}
 Q_{\frac{2k\pi}{m}} d_1\mathbf{e}_1, & \text{if } j = 1,\ 0 \leq k < m, \\
 Q_{\frac{(2k+\vartheta)\pi}{m}} d_2\mathbf{e}_1, & \text{if } j = 2,\ 0 \leq k < m,
\end{cases}
\end{equation}
with \( d_2 > d_1 > 0 \) and \( \vartheta \in \{0,1\} \), where \( \vartheta = 0 \) corresponds to the aligned configuration and \( \vartheta = 1 \) to the staggered configuration. Assuming that $x_0(t)=x_0(0)$ and  $x_{jk}(t) = Q_{\Omega t}x_{jk}(0)$, one may easily check that the system of $2m+1$  equations in \eqref{alg-sysP}  can be reduced to
\begin{equation*}
\begin{split}
\gamma_j{\mathcal{F}}_j^{\alpha}(\lambda)&=0,\quad j=0,1,2,
\end{split}
\end{equation*}
where $\lambda=(\Omega,\gamma_2)$ and
\begin{equation}\label{eq:functional}
\begin{split}
&{\mathcal{F}}_0^{\alpha}(\lambda):=\frac{\widehat{C}_\alpha}{2}\left(\frac{\gamma_1}{d_1^{1+\alpha}}\sum_{k=0}^{m-1}\cos\left(\frac{2k\pi}{m}  \right)+\frac{\gamma_2}{d_2^{1+\alpha}}\sum_{k=0}^{m-1}\cos\left(\frac{2k\pi}{m}  -\frac{\vartheta\pi }{m}\right)\right)\sin (x),
\\
&{\mathcal{F}}_1^{\alpha}(\lambda):=\Omega d_1\sin (x)-\frac{\widehat{C}_\alpha}{2d_1^{1+\alpha}}\left({\gamma_0}+\frac{\gamma_1}{2}\sum_{k=1}^{m-1}\frac{1 -  d ^{\pm 1} \cos\big({\frac{(2k\pm\vartheta)\pi }{m}}\big)}{\big(1+( d ^{\pm 1})^2 - 2 d ^{\pm 1} \cos\big({\frac{(2k\pm\vartheta)\pi }{m}}\big) \big)^{\frac\alpha2+1}}\right)\sin (x),
\\
&{\resizebox{.98\hsize}{!}{${\mathcal{F}}_2^{\alpha}(\lambda):=\Omega d_2\sin (x)-\frac{\widehat{C}_\alpha}{2d_2^{1+\alpha}}\left({\gamma_0}+\gamma_1\displaystyle\sum_{k=0}^{m-1}\frac{1 -  d ^{\pm 1} \cos\big({\frac{(2k\pm\vartheta)\pi }{m}}\big)}{\big(1+( d ^{\pm 1})^2 - 2 d ^{\pm 1} \cos\big({\frac{(2k\pm\vartheta)\pi }{m}}\big) \big)^{\frac\alpha2+1}}+\frac{\gamma_2}{2}\sum_{k=1}^{m-1}\frac{1 }{\big|2\sin\left(\frac{2k\pi  }{m}\right) \big|^{\alpha}}\right)\sin (x)$}},
\end{split}
\end{equation}
with $d:=d_2/d_1$. Notice also that in the limit, we recover the classical point vortex system composed of \( 2m + 1 \) vortices: one  vortex of strength \( 2\pi \gamma_0 \), \( m \) vortices of strength \( 2\pi \gamma_1 \) arranged symmetrically around it, and another \( m \) vortices of strength \( 2\pi \gamma_2 \) uniformly rotating around their centroid with a critical angular velocity given by
\begin{equation*}
\Omega^*_\alpha =\frac{\widehat{C}_\alpha}{2(d_1^{\alpha+2} + d_2^{\alpha+2})} \left(
\gamma_0 + \gamma_1 \left( \tfrac{1}{2} S_\alpha + T_\alpha^-(d, \vartheta) \right)
+ \gamma_2 \left( T_\alpha^+(d, \vartheta) + \tfrac{1}{2} S_\alpha \right)
\right)
\end{equation*}
Assuming that this configuration undergoes a steady, uniform rotation with angular velocity \( \Omega \), it can be shown that the full system of \( 2m + 1 \) equations reduces to the following system of two scalar equations
\begin{equation}\label{syst-pj2}
 {\resizebox{.97\hsize}{!}{$ \Omega d_j-\frac{\widehat{C}_\alpha}{2d_j^{1+\alpha}}\left({\gamma_0}+{\gamma_j}\displaystyle\sum_{k=1}^{m-1}{ \Big(2\sin\big({\frac{k\pi  }{m}}\big)  \Big)^{-\alpha}}+\gamma_{3-j}\displaystyle\sum_{k=0}^{m-1}\frac{1 -  \frac{d_{3-j}}{d_j} \cos\big({\frac{(2k\pm\vartheta)\pi }{m}}\big)}{\left(1+\left( \frac{d_{3-j}}{d_j} \right)^2 - 2\frac{d_{3-j}}{d_j} \cos\left({\frac{(2k\pm\vartheta)\pi }{m}}\right) \right)^{\frac\alpha2+1}}\right)=0, \, j=1,2.$}}
\end{equation}
This system is linear in both \( \Omega \) and \( \gamma_2 \). Therefore, provided a suitable non-degeneracy condition given by
\[( d ^{\alpha+2}-1){\gamma_0}+ \big(\tfrac12S_\alpha  d ^{\alpha+2} - T_\alpha^-( d ,\vartheta)\big)\gamma_1\neq 0.\]
and
\[\det\big(D_\lambda \mathcal F^\alpha(\lambda)\big)=
-\frac{\widehat{C}_\alpha d_1}{2d_2^{\alpha+1}}\Big(\frac{ S_\alpha}{2}- d^{\alpha+2} T_\alpha^+( d ,\vartheta)\Big)\neq 0 ,\]
 one may explicitly solve \eqref{syst-pj2} for \( \Omega \) and \( \gamma_2 \neq 0 \), where
 the constant \(\widehat{C}_\alpha\) is defined as
\begin{equation}\label{eqn:kalpha2}
\widehat{C}_\alpha := \alpha C_\alpha = \frac{2^\alpha \Gamma(1 + \alpha/2)}{\Gamma(1 - \alpha/2)}.
\end{equation}

To analyze this setting further, we reduce the system to three equations posed on the boundaries of \( \mathcal{O}_0^\varepsilon, \mathcal{O}_1^\varepsilon, \mathcal{O}_2^\varepsilon \). Initially, the vorticity \( \theta_{0}^\varepsilon \) involves \( 2m + 1 \) patches, each governed by the dynamics in equation~\eqref{1-1}. Using the symmetry
\begin{equation}\label{eq:d0}
\mathcal{D}_{1j}^\varepsilon = Q_{\frac{2j\pi}{m}} \mathcal{D}_{10}^\varepsilon,
\quad
\mathcal{D}_{2j}^\varepsilon = Q_{\frac{2j\pi}{m}} \mathcal{D}_{20}^\varepsilon,
\end{equation}
we consider all patches rotating rigidly around the origin with angular velocity \( \Omega \), in a counterclockwise direction.

In particular, we look for a solution \( \theta^\varepsilon(t) \) of \eqref{1-1} that corresponds to a uniformly rotating structure

\[
    \theta^\varepsilon(x,t) = \theta_0^\varepsilon\big(Q_{\Omega t}(x-d_i\mathbf{e}_1)\big),
\]
where \( Q_{\Omega t} \) represents a counterclockwise rotation by angle \( \Omega t \), and \( \Omega \) is the (constant) angular velocity of the system.
Substituting the ansatz into \eqref{1-1} leads to the following system of boundary conditions:
\begin{equation*}
\begin{split}
&\gamma_{0}\big(v_0(x)+\Omega x^\perp\big)\cdot  \nabla \theta_{0}^\varepsilon(x) = 0,\quad  \forall\, x \in \partial \mathcal{O}_{00}^\varepsilon, \\
&\gamma_{1}\big(v_0(x)+\Omega (x - d_1 \mathbf{e}_1)^\perp\big)\cdot  \nabla \theta_{0}^\varepsilon(x) = 0,\quad  \forall\, x \in \partial \mathcal{O}_{1n}^\varepsilon,\quad n = 0,\ldots,m-1, \\
&\gamma_{2}\big(v_0(x)+\Omega (x - d_2 \mathbf{e}_1)^\perp\big)\cdot  \nabla \theta_{0}^\varepsilon(x) = 0,\quad  \forall\, x \in \partial \mathcal{O}_{2n}^\varepsilon,\quad n = 0,\ldots,m-1,
\end{split}
\end{equation*}
where \( v^\varepsilon \) denotes the velocity field associated to \( \theta_{0}^\varepsilon \). These conditions ensure that the velocity on the boundary of each patch is purely tangential, matching a rigid co-rotating motion.
This implies the boundary velocity satisfies
\begin{equation}\label{eq:boundary}
\begin{split}
&\gamma_{0}\big(v_0(x)+\Omega x^\perp\big)\cdot \mathbf{n}(x) = 0,\quad  \forall\, x \in \partial \mathcal{O}_{00}^\varepsilon, \\
&\gamma_{1}\big(v_0(x)+\Omega (x - d_1 \mathbf{e}_1)^\perp\big)\cdot \mathbf{n}(x) = 0,\quad  \forall\, x \in \partial \mathcal{O}_{1n}^\varepsilon,\quad n = 0,\ldots,m-1, \\
&\gamma_{2}\big(v_0(x)+\Omega (x - d_2 \mathbf{e}_1)^\perp\big)\cdot \mathbf{n}(x) = 0,\quad  \forall\, x \in \partial \mathcal{O}_{2n}^\varepsilon,\quad n = 0,\ldots,m-1,
\end{split}
\end{equation}
where each boundary \( \partial \mathcal{O}_{j}^\varepsilon \) is parameterized by a function \( z_j(x) \), and \( \mathbf{n}(x) \) denotes the outward unit normal vector at the point \( x \in [0, 2\pi) \).

We are interested in nested polygonal vortex patches close to the unit disk \( \mathbb{D} \), with boundary perturbations of order \( \varepsilon b_j \). The boundary of each patch can be written as:
\[
z_j(x) = \varepsilon b_j R_j(x) (\cos x, \sin x), \quad j = 0, \ldots, m-1,
\]
where
\[
R_j(x) = 1 + \varepsilon |\varepsilon|^\alpha b_j^{1+\alpha} f_j(x), \quad x \in [0, 2\pi).
\]
Our goal is to derive a system of equations for the functions \( R_j(x) \), and reformulate it in terms of the profile functions \( f_j(x) \). Constructing \( 2m+1 \) co-rotating vortex patches for the gSQG equations with \( 1 \leq \alpha < 2 \) is thus equivalent to solving the following nonlinear system, which we explore in the next section
\begin{equation}
\begin{split}
&\Omega \left( \varepsilon b_{0}R_{0}(x)R_0^{\prime }(x)\right)\\
&
{\resizebox{.98\hsize}{!}{$+\frac{\gamma_0
C_\alpha}{\varepsilon^{1+\alpha}b_0^{1+\alpha} }\displaystyle\fint
\frac{\left((R_0(x)R_0(y)+R_0'(x)R_0'(y))%
\sin(x-y)+(R_0(x)R_0'(y)-R_0'(x)R_0(y))\cos(x-y)\right)dy}{\left(
\left(R_0(x)-R_0(y)\right)^2+4R_0(x)R_0  (y)\sin^2\left(\frac{x-y}{2}\right)%
\right)^{\frac{\alpha}{2}}}$}} \\
&
{\resizebox{.98\hsize}{!}{$+\displaystyle\sum_{\ell=1}^2 \gamma_\ell\displaystyle\sum_{k=0}^{m-1}\frac{ C_\alpha}{\varepsilon  b_\ell}
\displaystyle\fint
\frac{\left((R_0(x)R_\ell(y)+R_0'(x)R_\ell'(y))%
\sin(x-y)+(R_0(x)R_\ell'(y)-R_0'(x)R_\ell(y))\cos\left(x-y+2k\pi  /m-\delta_{2\ell}\vartheta\pi /m\right)\right)dy}{|(z_0(x)+(d_0,0))-\nu_{k\ell 0}(z_\ell(y)+(d_\ell,0))|^\alpha}$}} \\
& \qquad =0,
\end{split}
\label{f1}
\end{equation}%
\begin{equation}
\begin{split}
&\Omega \left(\varepsilon b_{j}R_{j}(x)R_j^{\prime }(x)-d_j R_{j}^{\prime
}(x)\cos (x)+d_j R_{j}(x)\sin (x)\right)   \\
&
{\resizebox{.98\hsize}{!}{$+\frac{\gamma_j
C_\alpha}{\varepsilon^{1+\alpha}b_j^{1+\alpha} }\displaystyle\fint
\frac{\left((R_j(x)R_j(y)+R_j'(x)R_j'(y))%
\sin(x-y)+(R_j(x)R_j'(y)-R_j'(x)R_j(y))\cos(x-y)\right)dy}{\left(
\left(R_j(x)-R_j(y)\right)^2+4R_j(x)R_j  (y)\sin^2\left(\frac{x-y}{2}\right)%
\right)^{\frac{\alpha}{2}}}$}} \\
&
{\resizebox{.98\hsize}{!}{$+\frac{\gamma_0
C_\alpha}{\varepsilon b_0}\displaystyle\fint
\frac{\left((R_j(x)R_0(y)+R_j'(x)R_0'(y))%
\sin(x-y)+(R_j(x)R_0'(y)-R_j'(x)R_0(y))\cos(x-y+\delta_{2j}\vartheta\pi /m)\right)dy}{|(z_j(x)+(d_j,0))-\nu_{k 0 j}(z_0(y)+(d_0,0))|^\alpha}$}} \\
&
{\resizebox{.98\hsize}{!}{$+\displaystyle\sum_{\ell=1}^2 \gamma_\ell \displaystyle\sum_{k=\delta_{\ell j}}^{m-1}\frac{ C_\alpha}{\varepsilon  b_\ell}
\displaystyle\fint
\frac{\left((R_j(x)R_\ell(y)+R_j'(x)R_\ell'(y))%
\sin(x-y)+(R_j(x)R_\ell'(y)-R_j'(x)R_\ell(y))\cos\left(x-y+2k\pi  /m-(\delta_{2\ell}-\delta_{2j})\vartheta\pi /m\right)\right)dy}{|(z_j(x)+(d_j,0))-\nu_{k\ell j}(z_\ell(y)+(d_\ell,0))|^\alpha}$}} \\
& \qquad =0,
\end{split}
\label{f2}
\end{equation}%
for \( j = 1, 2 \), we adopt the convention \( d_0 = 0 \), and define
\[
\nu_{k\ell j} := Q_{\frac{2k\pi}{m} - (\delta_{2\ell} - \delta_{2j})\frac{\vartheta\pi}{m}},
\]
where \( \delta_{ij} \) denotes the Kronecker delta
\[
\delta_{ij} :=
\begin{cases}
  1 & \text{if } i = j, \\
  0 & \text{if } i \neq j.
\end{cases}
\]

\subsection{Notation and functional spaces}


We need to establish some notations that will be utilized consistently throughout this paper. Let \( C \) denote a generic positive constant, which may vary from one occurrence to another.  For simplicity, we define
\[
\fint_{0}^{2\pi} f(\tau)\, d\tau \equiv \frac{1}{2\pi} \int_{0}^{2\pi} f(\tau)\, d\tau,
\]
to represent the average (mean) value of the function \( f \) over the unit circle.
 The primary strategy for demonstrating the existence of \(2m+1\)  nested polygonal vortex patches for the gSQG equation \eqref{1-1} with $\alpha\in[1,2)$, involves applying the implicit function theorem to the system of equations represented by \eqref{f1} and \eqref{f2}. To accomplish this, it is crucial to verify that the functional \(\mathcal{F}^\alpha\) meets the necessary regularity conditions. The relevant functional spaces utilized in this analysis are outlined below
\begin{equation*}
	X^N=\left\{ f\in H^N, \ f(x)= \sum\limits_{n=2}^{\infty}a_n\cos(n x)\right\},
\end{equation*}

\begin{equation*}
	Y^{N}=\left\{ g\in H^{N}, \ g(x)= \sum\limits_{n=1}^{\infty}a_n \sin (n x)\right\},
\end{equation*}
and
\begin{equation*}
Y_{0}^{N}=Y^{N}/\text{span}\{\sin (x)\}=\left\{ g\in Y^{N},\
g(x)=\sum\limits_{n=2}^{\infty }a_n\sin(n x)\right\} ,
\end{equation*}%
and for any \( m \geq 2 \), we define the following Banach spaces in the context of the SQG equations  $(\alpha=1)$
\begin{equation*}
\begin{split}
\mathcal{X}^{N+\log} :=X^{N+\log}_m\times X^{N+\log}_1\times X^{N+\log}_1 ,
\end{split}
\end{equation*}
with
\[ X_m^{N+\log}:=
        \big\{f\in X_1^{N+\log}: f\big(Q_{\frac{2\pi}{m}}x\big)=Q_{\frac{2\pi}{m}}f(x)\big\}, \]
        \begin{equation*}
	\begin{split}
 {\resizebox{.98\hsize}{!}{$ X^{N+\log}_1=\left\{ f\in H^N, \ f(x)= \sum\limits_{n=2}^{\infty} a_n\cos(n x), \ \left\|\displaystyle\int_0^{2\pi}\frac{\partial^N f(x-y)-\partial^N f(x)}{|\sin(\frac{y}{2})|}dy\right\|_{L^2}<\infty \right\}.$}}
	\end{split}
\end{equation*}
Similarly, we shall define the Banach spaces for the gSQG equations with $\alpha\in(1,2)$
\[\mathcal{X}^{N+\alpha-1} :=X^{N+\alpha-1}_m\times X^{N+\alpha-1}_1\times X^{N+\alpha-1}_1,\]
with
\[ X_m^{N+\alpha-1}:=
        \big\{f\in X_1^{N+\alpha-1}: f\big(Q_{\frac{2\pi}{m}}x\big)=Q_{\frac{2\pi}{m}}f(x)\big\}, \]

\begin{equation*}
	\begin{split}
 {\resizebox{.98\hsize}{!}{$ X^{N+\alpha-1}_1=\left\{ f\in H^N, \ f(x)= \sum\limits_{n=2}^{\infty}a_n\cos(n x), \ \left\|\displaystyle\int_0^{2\pi}\frac{\partial^N f(x-y)-\partial^N f(x)}{|\sin(\frac{y}{2})|^\alpha}dy\right\|_{L^2}<\infty \right\}.$}}
	\end{split}
\end{equation*}
Additionally, we also need the following spaces
\[\mathcal{Y}^N:=Y^N_m\times Y^N_1\times Y^N_1,\]
with
 \[Y^{N}_m:=
        \big\{g\in  Y_1^{N} : g\big(Q_{\frac{2\pi}{m}}y\big)=g(y)\big\},\]
The norms for the spaces \(X^N\) and \(Y^N\) are defined using the \(H^N\)-norm. In the cases of \( X^{N+\log}_1 \) and \( X^{N+\alpha-1}_1 \), their norms are defined as the sum of the \( H^N \)-norm and an additional integral term, as specified in their respective definitions. It is important to note that, for all \( \nu > 0 \), the following continuous embeddings hold
\[
X^{N+\nu} \hookrightarrow X^{N+\log}_1 \hookrightarrow X^N.
\]
Throughout this paper, we will consistently assume that \( N \geq 3 \).

We denote by \( \mathcal{B}_X \) the unit ball in the Banach space \( X \), defined as
\[
\mathcal{B}_X := \{ f \in X : \|f\|_X < 1 \}.
\]

\begin{equation*}
    \mathcal{B}_X:= \begin{cases}
       &f\in \mathcal{X}^{N+\log}: \|f\|_{\mathcal{X}^{N+\log}(\mathbb{T})}< 1,\qquad \alpha=1, \\
       & f\in \mathcal{X}^{N+\alpha-1}: \|f\|_{\mathcal{X}^{N+\alpha-1}(\mathbb{T})}< 1 ,\qquad \alpha\in(1,2) ,
    \end{cases}
    \end{equation*}

\section{Regularity and linearization analysis on the functional}\label{section3}

We consider domains \(\mathcal{O}_j^\varepsilon\) obtained as perturbations of the unit disk, where the perturbation amplitude scales with \(\varepsilon b_j\).
Specifically,  we consider the following expansion
\begin{equation}\label{conf0}
R_j(x) = 1 + \varepsilon\abs{\varepsilon}^{\alpha} b_j^{1+\alpha} f_j(x).
\end{equation}
The coefficient \(|\varepsilon|^{\alpha}\) in \eqref{conf0} arises from the singularity of the gSQG kernel. To desingularize the system with respect to \(\varepsilon\), we adopt the approach from \cite{Hmidi-Mateu} and transform the equations \eqref{f1} and \eqref{f2} into the following forms
\begin{equation*}
\mathcal{F}^\alpha_0(\varepsilon,f,\lambda):=\mathcal{F}^\alpha_{01}+\mathcal{F}^\alpha_{02}+\mathcal{F}^\alpha_{03},
\end{equation*}
and
\begin{equation*}
\mathcal{F}^\alpha_j(\varepsilon,f,\lambda):=\mathcal{F}^\alpha_{j1}+\mathcal{F}^\alpha_{j2}+\mathcal{F}^\alpha_{j3}+\mathcal{F}^\alpha_{j4},\, \mbox{for} \, j=1,2 ,
\end{equation*}
where
\begin{equation}\label{f01}
\begin{split}
  \mathcal{F}^\alpha_{0 1}&=  \Omega  \,|\varepsilon |^{2+\alpha }b_{0}^{2+\alpha
}f_{0}^{\prime }(x) ,
\end{split}
\end{equation}
\begin{equation}\label{2-2}
	\begin{split}
    		\mathcal{F}^\alpha_{0 2}=&{\resizebox{.94\hsize}{!}{$\frac{C_\alpha \gamma_0}{\varepsilon|\varepsilon|^{\alpha}b_0^{1+\alpha}}\displaystyle\fint\frac{(  1+\varepsilon|\varepsilon|^\alpha b_0^{1+\alpha} f_0( y))\sin( x- y)d y}{\left(
|\varepsilon|^{2+2\alpha}b_0^{2+2\alpha}\left(f_0(x)-f_0(y)\right)^2+4(1+%
\varepsilon|\varepsilon|^\alpha b^{1+\alpha}_0
f_0(x))(1+\varepsilon|\varepsilon|^\alpha b^{1+\alpha}_0
f_0(y))\sin^2\left(\frac{x-y}{2}\right)\right)^{\frac{\alpha}{2}}}$}}\\
&
 {\resizebox{.94\hsize}{!}{$ +C_\alpha \gamma_0\displaystyle\fint \frac{(f'_0( y)-f'_0( x))\cos( x- y)d y}{\left(
|\varepsilon|^{2+2\alpha}b_0^{2+2\alpha}\left(f_0(x)-f_0(y)\right)^2+4(1+%
\varepsilon|\varepsilon|^\alpha b^{1+\alpha}_0
f_0(x))(1+\varepsilon|\varepsilon|^\alpha b^{1+\alpha}_0
f_0(y))\sin^2\left(\frac{x-y}{2}\right)\right)^{\frac{\alpha}{2}}}$}}\\		
		&
 {\resizebox{.94\hsize}{!}{$ +\frac{C_\alpha \gamma_0\varepsilon |\varepsilon|^\alpha b_0^{1+\alpha} f'_0( x)}{  1+\varepsilon|\varepsilon|^\alpha b_0^{1+\alpha} f_0( x)}\displaystyle\fint \frac{(f_0( x)-f_0( y))\cos( x- y)d y}{\left(
|\varepsilon|^{2+2\alpha}b_0^{2+2\alpha}\left(f_0(x)-f_0(y)\right)^2+4(1+%
\varepsilon|\varepsilon|^\alpha b^{1+\alpha}_0
f_0(x))(1+\varepsilon|\varepsilon|^\alpha b^{1+\alpha}_0
f_0(y))\sin^2\left(\frac{x-y}{2}\right)\right)^{\frac{\alpha}{2}}}$}}\\	
		&
{\resizebox{.94\hsize}{!}{$+\frac{C_\alpha \gamma_0\varepsilon|\varepsilon|^\alpha b_0^{1+\alpha} }{  1+\varepsilon|\varepsilon|^\alpha b_0^{1+\alpha} f_0( x)}\displaystyle\fint \frac{f'_0( x)f'_0( y)\sin( x- y)d y}{\left(
|\varepsilon|^{2+2\alpha}b_0^{2+2\alpha}\left(f_0(x)-f_0(y)\right)^2+4(1+%
\varepsilon|\varepsilon|^\alpha b^{1+\alpha}_0
f_0(x))(1+\varepsilon|\varepsilon|^\alpha b^{1+\alpha}_0
f_0(y))\sin^2\left(\frac{x-y}{2}\right)\right)^{\frac{\alpha}{2}}}$}}\\	
		=&
:\mathcal{F}^\alpha_{021}+\mathcal{F}^\alpha_{022}+\mathcal{F}^\alpha_{023}+\mathcal{F}^\alpha_{024} ,
	\end{split}
\end{equation}
\begin{equation}
\begin{split}
\quad&{\resizebox{.98\hsize}{!}{$ \mathcal{F}^\alpha_{03}=\displaystyle\sum_{\ell=1}^2 \gamma_\ell \displaystyle\sum_{k=0}^{m-1}\frac{ C_\alpha
(1+\varepsilon|\varepsilon|^\alpha b_\ell^{1+\alpha}
f_\ell(x))}{\varepsilon b_\ell (1+\varepsilon|\varepsilon|^\alpha
b^{1+\alpha}_0 f_0(x)) }\displaystyle\fint
\frac{(1+\varepsilon|\varepsilon|^\alpha b_\ell^{1+\alpha}
f_\ell(y))\sin\left(x-y+2k\pi  /m-\delta_{2\ell}\vartheta\pi /m\right)dy}{|(z_0(x)+(d_0,0))-\nu_{k \ell 0}(z_\ell(y)+(d_\ell,0))|^\alpha}$}} \\
&
+\displaystyle\sum_{\ell=1}^2 \gamma_\ell \displaystyle\sum_{k=0}^{m-1}\frac{ C_\alpha \varepsilon
|\varepsilon|^{2\alpha}b_\ell^{1+2\alpha}}{1+\varepsilon|\varepsilon|^%
\alpha b^{1+\alpha}_0 f_0(x)}\displaystyle\fint  \frac{
f_\ell'(x)f_\ell'(y)\sin\left(x-y+2k\pi  /m-\delta_{2\ell}\vartheta\pi /m\right)dy}{|(z_0(x)+(d_0,0))-\nu_{k \ell 0}(z_\ell(y)+(d_\ell,0))|^\alpha} \\
&
+{\resizebox{.95\hsize}{!}{$\displaystyle\sum_{\ell=1}^2 \gamma_\ell \displaystyle\sum_{k=0}^{m-1}\frac{ C_\alpha
|\varepsilon|^\alpha (1+\varepsilon|\varepsilon|^\alpha b_\ell^{1+\alpha}
f_\ell(x))}{1+\varepsilon|\varepsilon|^\alpha b^{1+\alpha}_0
f_0(x)}\displaystyle\fint
\frac{b_\ell^{\alpha}(f'_\ell(y)-f'_\ell(x))\cos\left(x-y+2k\pi  /m-\delta_{2\ell}\vartheta\pi /m\right)dy}{|(z_0(x)+(d_0,0))-\nu_{k \ell 0}(z_\ell(y)+(d_\ell,0))|^\alpha}$}}\\
&
{\resizebox{.98\hsize}{!}{$+\displaystyle\sum_{\ell=1}^2 \gamma_\ell \displaystyle\sum_{k=0}^{m-1}\frac{ C_\alpha
\varepsilon|\varepsilon|^{2\alpha}f_0'(x)}{1+\varepsilon|\varepsilon|^\alpha
b^{1+\alpha}_0 f_0(x)}\displaystyle\fint
\frac{b_\ell^{1+2\alpha}(f_\ell(x)-f_\ell(y))\cos\left(x-y+2k\pi  /m-\delta_{2\ell}\vartheta\pi /m\right)dy}{|(z_0(x)+(d_0,0))-\nu_{k \ell 0}(z_\ell(y)+(d_\ell,0))|^\alpha}$}} \\
 =&: \mathcal{F}^\alpha_{031}+ \mathcal{F}^\alpha_{032}+ \mathcal{F}^\alpha_{033}+ \mathcal{F}^\alpha_{034},
\end{split}
\label{2-4}
\end{equation}%
and
\begin{equation}
\begin{split}
& \mathcal{F}^\alpha_{j1}=\Omega \left( |\varepsilon |^{2+\alpha }b_{j}^{2+\alpha
}f_{j}^{\prime }(x)-d_j \left( \frac{\varepsilon|\varepsilon|^{\alpha}b_{j}^{1+\alpha }f_{j}^{\prime }(x)\cos (x)}{1+\varepsilon|\varepsilon|^{\alpha}b_{j}^{1+\alpha }f_{j}(x)}-\sin (x)\right) \right) ,
\end{split}
\label{fj1}
\end{equation}
\begin{equation}
\begin{split}
& {\resizebox{.98\hsize}{!}{$
 \mathcal{F}^\alpha_{j2}=\frac{C_\alpha\gamma_j}{\varepsilon|\varepsilon|^{\alpha}b_j^{1+%
\alpha}}\displaystyle\fint
\frac{(1+\varepsilon|\varepsilon|^\alpha b^{1+\alpha}_j
f_j(y))\sin(x-y)dy}{\left(
|\varepsilon|^{2+2\alpha}b_j^{2+2\alpha}\left(f_j(x)-f_j(y)\right)^2+4(1+%
\varepsilon|\varepsilon|^\alpha b^{1+\alpha}_j
f_j(x))(1+\varepsilon|\varepsilon|^\alpha b^{1+\alpha}_j
f_j(y))\sin^2\left(\frac{x-y}{2}\right)\right)^{\frac{\alpha}{2}}}$}} \\
& +C_{\alpha }\gamma_j\,{\resizebox{.9\hsize}{!}{$
\displaystyle\fint  \frac{(f'_j(y)-f'_j(x))\cos(x-y)dy}{\left(
|\varepsilon|^{2+2\alpha}b_j^{2+2\alpha}\left(f_j(x)-f_j(y)\right)^2+4(1+%
\varepsilon|\varepsilon|^\alpha b^{1+\alpha}_j
f_j(x))(1+\varepsilon|\varepsilon|^\alpha b^{1+\alpha}_j
f_j(y))\sin^2\left(\frac{x-y}{2}\right)\right)^{\frac{\alpha}{2}}}$}} \\
& +{\resizebox{.96\hsize}{!}{$\frac{C_\alpha
\gamma_j\varepsilon|\varepsilon|^\alpha b^{1+\alpha}_j
f'_j(x)}{1+\varepsilon|\varepsilon|^\alpha
b^{1+\alpha}_jf_j(x)}\displaystyle\fint
\frac{(f_j(x)-f_j(y))\cos(x-y)dy}{\left(
|\varepsilon|^{2+2\alpha}b_j^{2+2\alpha}\left(f_j(x)-f_j(y)\right)^2+4(1+%
\varepsilon|\varepsilon|^\alpha b^{1+\alpha}_j
f_j(x))(1+\varepsilon|\varepsilon|^\alpha b^{1+\alpha}_j
f_j(y))\sin^2\left(\frac{x-y}{2}\right)\right)^{\frac{\alpha}{2}}}$}} \\
& +{\resizebox{.96\hsize}{!}{$\frac{C_\alpha\gamma_j\varepsilon|%
\varepsilon|^\alpha b^{1+\alpha}_j }{1+\varepsilon|\varepsilon|^\alpha
b^{1+\alpha}_j f_j(x)}\displaystyle\fint
\frac{f'_j(x)f'_j(y)\sin(x-y)dy}{\left(
|\varepsilon|^{2+2\alpha}b_j^{2+2\alpha}\left(f_j(x)-f_j(y)\right)^2+4(1+%
\varepsilon|\varepsilon|^\alpha b^{1+\alpha}_j
f_j(x))(1+\varepsilon|\varepsilon|^\alpha b^{1+\alpha}_j
f_j(y))\sin^2\left(\frac{x-y}{2}\right)\right)^{\frac{\alpha}{2}}}$}} \\
& =: \mathcal{F}^\alpha_{j21}+ \mathcal{F}^\alpha_{j22}+ \mathcal{F}^\alpha_{j23}+ \mathcal{F}^\alpha_{j24},
\end{split}
\label{2-2b}
\end{equation}%
\begin{equation}
\begin{split}
&  \mathcal{F}^\alpha_{j3}=\frac{\gamma_{0} C_\alpha
(1+\varepsilon|\varepsilon|^\alpha b_{0}^{1+\alpha}
f_{0}(x))}{\varepsilon b_{0} (1+\varepsilon|\varepsilon|^\alpha
b^{1+\alpha}_j f_j(x)) }\displaystyle\fint
\frac{(1+\varepsilon|\varepsilon|^\alpha b_{0}^{1+\alpha}
f_{0}(y))\sin\left(x-y+\delta_{2j}\vartheta\pi /m\right)dy}{|(z_j(x)+(d_j,0))-\nu_{0 0 j}(z_0(y)+(d_0,0))|^\alpha} \\
&
+\frac{\gamma_{0} C_\alpha \varepsilon
|\varepsilon|^{2\alpha}b_{0}^{1+2\alpha}}{1+\varepsilon|\varepsilon|^%
\alpha b^{1+\alpha}_j f_j(x)}\displaystyle\fint  \frac{
f_{0}'(x)f_{0}'(y)\sin\left(x-y+\delta_{2j}\vartheta\pi /m\right)dy}{|(z_j(x)+(d_j,0))-\nu_{0 0 j}(z_0(y)+(d_0,0))|^\alpha} \\
&  +\frac{\gamma_{0} C_\alpha
|\varepsilon|^\alpha (1+\varepsilon|\varepsilon|^\alpha b_{0}^{1+\alpha}
f_{0}(x))}{1+\varepsilon|\varepsilon|^\alpha b^{1+\alpha}_j
f_j(x)}\displaystyle\fint
\frac{b_{0}^{\alpha}(f'_{0}(y)-f'_{0}(x))\cos\left(x-y+\delta_{2j}\vartheta\pi /m\right)dy}{|(z_j(x)+(d_j,0))-\nu_{0 0 j}(z_0(y)+(d_0,0))|^\alpha}\\
& +\frac{\gamma_{0} C_\alpha
\varepsilon|\varepsilon|^{2\alpha}f_j'(x)}{1+\varepsilon|\varepsilon|^\alpha
b^{1+\alpha}_j f_j(x)}\displaystyle\fint
\frac{b_{0}^{1+2\alpha}(f_{0}(x)-f_{0}(y))\cos\left(x-y+\delta_{2j}\vartheta\pi /m\right)dy}{|(z_j(x)+(d_j,0))-\nu_{0 0 j}(z_0(y)+(d_0,0))|^\alpha} \\
& =: \mathcal{F}^\alpha_{j31}+ \mathcal{F}^\alpha_{j32}+ \mathcal{F}^\alpha_{j33}+ \mathcal{F}^\alpha_{j34},
\end{split}
\label{2-3b}
\end{equation}%
\begin{equation}
\begin{split}
\mathcal{F}^\alpha_{j4}=&{\resizebox{.9\hsize}{!}{$ \displaystyle\sum_{\ell=1}^2 \gamma_\ell \displaystyle\sum_{k=\delta_{\ell j}}^{m-1}\frac{ C_\alpha
(1+\varepsilon|\varepsilon|^\alpha b_\ell^{1+\alpha}
f_\ell(x))}{\varepsilon b_\ell (1+\varepsilon|\varepsilon|^\alpha
b^{1+\alpha}_j f_j(x)) }\displaystyle\fint
\frac{(1+\varepsilon|\varepsilon|^\alpha b_\ell^{1+\alpha}
f_\ell(y))\sin\left(x-y+2k\pi  /m-(\delta_{2\ell}-\delta_{2j})\vartheta\pi /m\right)dy}{|(z_j(x)+(d_j,0))-\nu_{k \ell j}(z_\ell(y)+(d_\ell,0))|^\alpha}$}} \\
&
{\resizebox{.9\hsize}{!}{$ +\displaystyle\sum_{\ell=1}^2 \gamma_\ell \displaystyle\sum_{k=\delta_{\ell j}}^{m-1}\frac{ C_\alpha \varepsilon
|\varepsilon|^{2\alpha}b_\ell^{1+2\alpha}}{1+\varepsilon|\varepsilon|^%
\alpha b^{1+\alpha}_j f_j(x)}\displaystyle\fint  \frac{
f_\ell'(x)f_\ell'(y)\sin\left(x-y+2k\pi  /m-(\delta_{2\ell}-\delta_{2j})\vartheta\pi /m\right)dy}{|(z_j(x)+(d_j,0))-\nu_{k \ell j}(z_\ell(y)+(d_\ell,0))|^\alpha}$}} \\
&  {\resizebox{.9\hsize}{!}{$ +\displaystyle\sum_{\ell=1}^2 \gamma_\ell \displaystyle\sum_{k=\delta_{\ell j}}^{m-1}\frac{ C_\alpha
|\varepsilon|^\alpha (1+\varepsilon|\varepsilon|^\alpha b_\ell^{1+\alpha}
f_\ell(x))}{1+\varepsilon|\varepsilon|^\alpha b^{1+\alpha}_j
f_j(x)}\displaystyle\fint
\frac{b_\ell^{\alpha}(f'_\ell(y)-f'_\ell(x))\cos\left(x-y+2k\pi  /m-(\delta_{2\ell}-\delta_{2j})\vartheta\pi /m\right)dy}{|(z_j(x)+(d_j,0))-\nu_{k \ell j}(z_\ell(y)+(d_\ell,0))|^\alpha}$}}\\
& {\resizebox{.9\hsize}{!}{$ +\displaystyle\sum_{\ell=1}^2 \gamma_\ell \displaystyle\sum_{k=\delta_{\ell j}}^{m-1}\frac{ C_\alpha
\varepsilon|\varepsilon|^{2\alpha}f_j'(x)}{1+\varepsilon|\varepsilon|^\alpha
b^{1+\alpha}_j f_j(x)}\displaystyle\fint
\frac{b_\ell^{1+2\alpha}(f_\ell(x)-f_\ell(y))\cos\left(x-y+2k\pi  /m-(\delta_{2\ell}-\delta_{2j})\vartheta\pi /m\right)dy}{|(z_j(x)+(d_j,0))-\nu_{k \ell j}(z_\ell(y)+(d_\ell,0))|^\alpha}$}} \\
 =&: \mathcal{F}^\alpha_{j41}+ \mathcal{F}^\alpha_{j42}+ \mathcal{F}^\alpha_{j43}+ \mathcal{F}^\alpha_{j44},
\end{split}
\label{2-4b}
\end{equation}%
where the terms $\mathcal{F}^\alpha_{j}\left(\varepsilon, f, \lambda \right)$,  for $j=1,2$, with the convention $d_0=0$ and
\[\nu_{k\ell j}=Q_{2k\pi  /m-(\delta_{2\ell}-\delta_{2j})\vartheta\pi /m} ,\]
 where $\delta_{ij}$ is  the Kronecker delta function. Define the nonlinear functional as follows
$$
\mathcal{F}^\alpha(\varepsilon,f,\lambda):=\big(\mathcal{F}^\alpha_0(\varepsilon,f,\lambda),\mathcal{F}^\alpha_1(\varepsilon,f,\lambda), \mathcal{F}^\alpha_2(\varepsilon,f,\lambda)\big),
$$
where $\mathcal{F}^\alpha_0$ is given by \eqref{f01}--\eqref{2-4} and $\mathcal{F}^\alpha_j$     by  \eqref{fj1}--\eqref{2-4b}, $f=(f_0,f_1,f_2)$, $\lambda=(\Omega,\gamma_2)$, for the gSQG equations with $\alpha\in[1,2)$.  This leads to a nonlinear
system
\begin{equation*}
    \mathcal{F}^\alpha(\varepsilon,f,\lambda)=0,
\end{equation*}
where $\mathcal{F}^\alpha_{j}:\left(-\varepsilon_0, \varepsilon_0\right) \times \mathcal{B}_X\times \Lambda \rightarrow \mathcal{Y}^{N-1}$. Here $\Lambda$ is a small neighborhood of $\lambda^\ast$, which is the solution to the point vortex system \eqref{alg-sysP}.

To apply the implicit function theorem at \(\varepsilon = 0\), it is essential to extend the functions $\mathcal{F}^\alpha$
which were introduced in Section \ref{section2}, to a domain that includes \(\varepsilon \leq 0\), while ensuring that these functions maintain \(C^1\) regularity. The first step in this process involves verifying the continuity of \(\mathcal{F}^\alpha\).

\begin{proposition}\label{p3-1}
	There exists $\varepsilon_0>0$ and a small neighborhood $\Lambda$ of $\lambda^*$  such that the functional $\mathcal{F}^\alpha$ can be extended from $\left(-\varepsilon_0, \varepsilon_0\right) \times \mathcal{B}_X\times \Lambda$ to $ \mathcal{Y}^{N-1}$ as  a continuous functional.
\end{proposition}

\begin{proof}
 Using  \eqref{2-2}, for $j=1,2$, we can decompose the functional $\mathcal{F}^\alpha_{j1}:\left(-\varepsilon_0, \varepsilon_0\right) \times \mathcal{B}_X\times \Lambda \rightarrow \mathcal{Y}^{N-1}$
as follows
\begin{equation}
\begin{split}
 \mathcal{F}^\alpha_{j1}=\Omega d_j\sin (x)+\varepsilon|\varepsilon|^\alpha \mathcal{R}_{j1}(\varepsilon,f_j) ,
\end{split}
\label{2-1b}
\end{equation}
where $\mathcal{R}^\alpha_{j1}$ is continuous.  Consequently, we can infer that \(\mathcal{F}^\alpha_{j 1}\) is continuous. Now we consider $\mathcal{F}^\alpha_{j2}$,
note that  the most singular component in $\mathcal{F}^\alpha_{j2}$  is given by $\mathcal{F}^\alpha_{j21}$. To deal with this term we proceed as follows
\begin{equation*}
   {\resizebox{.96\hsize}{!}{$ \mathcal{F}^\alpha_{j21}:=\frac{C_\alpha \gamma_j}{\varepsilon|\varepsilon|^{\alpha}b_j^{1+\alpha}}\displaystyle\fint  \frac{(  1+\varepsilon|\varepsilon|^{\alpha} b_{j}^{1+\alpha} f_j( y))\sin( x- y)}{\left(
|\varepsilon|^{2+2\alpha}b_j^{2+2\alpha}\left(f_j(x)-f_j(y)\right)^2+4(1+%
\varepsilon|\varepsilon|^\alpha b^{1+\alpha}_j
f_j(x))(1+\varepsilon|\varepsilon|^\alpha b^{1+\alpha}_j
f_j(y))\sin^2\left(\frac{x-y}{2}\right)\right)^{\frac{\alpha}{2}}} d y$}},
\end{equation*}
the potential singularity arising from \(\varepsilon=0\) may only manifest when we compute the zeroth-order derivative of \(\mathcal{F}^\alpha_{j21}\). Throughout the proof, we will often use the following Taylor's formula
\begin{equation}\label{taylor}
	\frac{1}{(A+B)^\alpha}=\frac{1}{A^\alpha}-\alpha\int_0^1\frac{B}{(A+tB)^{1+\alpha}}dt.
\end{equation}
Now, for a particular choice of  $A$ and $B$ given by
\begin{equation*}
A(x,y):=4\sin ^{2}\left( \frac{x-y}{2}\right)
\end{equation*}%
and
\begin{equation*}
{\resizebox{.95\hsize}{!}{$ B(f_j,x,y):=|\varepsilon |^{1+\alpha}b_{j}^{1+\alpha}(f_{j}(x)-f_{j}(y))^{2}+\sin ^{2}\left( \frac{%
x-y}{2}\right) \left( 4
f_{j}(x)+4f_{j}(y)+4|\varepsilon |^{1+\alpha}b_{j}^{1+\alpha}f_{j}(x)f_{j}(y)\right),$}}
\end{equation*}%
Then, we decompose the kernel into two parts
\begin{equation}\label{2-6}
	\begin{split}
		\mathcal{F}^\alpha_{j21}&{\resizebox{.95\hsize}{!}{$	=\frac{C_{\alpha}\gamma_j}{\varepsilon|\varepsilon|^{\alpha}b_j^{1+\alpha}}\displaystyle\fint\frac{ \sin{( x- y)}d y}{\left(  A( x, y)+\varepsilon|\varepsilon|^{\alpha}b_j^{1+\alpha}B\left( f_j, x, y\right)\right)^{\frac{\alpha}{2}}}+C_\alpha\gamma_j\displaystyle\fint \frac{f_j( y)\sin{( x- y)}d y}{\left(   A( x, y)+\varepsilon|\varepsilon|^{\alpha}b_j^{1+\alpha} B\left(   f_j,  x,  y\right)\right)^{\frac{\alpha}{2}}}$}}\\
		&=\frac{C_\alpha \gamma_j}{\varepsilon|\varepsilon|^{\alpha}b_j^{1+\alpha}}\int\!\!\!\!\!\!\!\!\!\; {}-{} \frac{\sin( x- y)d y}{ A(x,y)^{\frac{\alpha}{2}}}-\frac{\alpha C_\alpha \gamma_j}{2}\int\!\!\!\!\!\!\!\!\!\; {}-{} \int_0^1 \frac{B\sin( x- y)dt d y}{\left(  A( x,  y)+t\varepsilon|\varepsilon|^\alpha b_j^{1+\alpha}B\left( f_j,  x,  y\right)\right)^{\frac{\alpha+2}{2}}}\\
  &\qquad \qquad+C_\alpha\gamma_j\int\!\!\!\!\!\!\!\!\!\;{}-{} \frac{f_j( y)\sin{( x- y)}d y}{\left(   A( x, y)+\varepsilon|\varepsilon|^{\alpha} b_j^{1+\alpha}B\left(   f_j,  x,  y\right)\right)^{\frac{\alpha}{2}}}\\
		&=-\frac{\alpha C_\alpha \gamma_j}{2}\int\!\!\!\!\!\!\!\!\!\; {}-{} \int_0^1 \frac{B\sin( x- y)dt d y}{(A(x,y))^{\frac{\alpha}{2}+1}}\\
  &
\ \ \ \ + \frac{C_\alpha\gamma_j\varepsilon|\varepsilon|^\alpha \alpha(\alpha+2)}{4}\int\!\!\!\!\!\!\!\!\!\; {}-{}\int_0^1 \int_0^1 \frac{t B^{2}\sin( x- y)d\tau dt d y}{\left(  A( x,  y)+t\tau\varepsilon|\varepsilon|^\alpha b_j^{1+\alpha}B\left( f_j,  x,  y\right)\right)^{\frac{\alpha+4}{2}}}\\
		&\ \ \ \ {\resizebox{.9\hsize}{!}{$	+C_\alpha\gamma_j\displaystyle\fint\frac{f_j( y)\sin( x- y)d y}{A(x,y)^{\frac{\alpha}{2}}}-\frac{C_\alpha \alpha \gamma_j\varepsilon|\varepsilon|^\alpha}{2}\displaystyle
  \fint\int_0^1 \frac{B f_j( y)\sin( x- y)dt d y}{\left(  A( x,  y)+\varepsilon|\varepsilon|^\alpha b_j^{1+\alpha} B\left( f_j,  x,  y\right)\right)^{\frac{\alpha+2}{2}}}$}}\\
        &{\resizebox{.95\hsize}{!}{$= -\frac{\alpha C_\alpha \gamma_j}{2}\displaystyle\fint  \frac{f_j( y)A(x,y)\sin( x- y)d y}{A(x,y)^{\frac{\alpha+2}{2}}}+C_\alpha \gamma_j\displaystyle\fint \frac{f_j( y)\sin( x- y)d y}{A(x,y)^{\frac{\alpha}{2}}}+\varepsilon|\varepsilon|^\alpha\mathcal{R}_{j21}(\varepsilon,f_j)$}}\\
		&=C_\alpha \gamma_j\left(1-\frac{\alpha}{2}\right)\int\!\!\!\!\!\!\!\!\!\; {}-{} \frac{f_j( y)\sin( x- y)d y}{\left|4\sin(\frac{ x- y}{2})\right|^\alpha}+\varepsilon|\varepsilon|^\alpha\mathcal{R}_{j21}(\varepsilon,f_j),
	\end{split}
\end{equation}
where $\mathcal{R}_{j21}$ is not singular with respect to $\varepsilon$.

Next, we proceed to differentiate $\mathcal{F}^\alpha_j$ with respect to $ x$ up to $\partial^{N-1}$ times. Our initial focus will be on the most singular term, namely, $ \partial ^{N-1}\mathcal{F}^\alpha_{j22}$
\begin{equation*}
\begin{split}
& {\resizebox{.96\hsize}{!}{$\partial ^{N-1}\mathcal{F}^\alpha_{j22}=C_\alpha\gamma _j\displaystyle\fint\frac{%
(\partial ^{N}f_j(y)-\partial ^Nf_j(x))\cos (x-y)dy}{\left( A( x, y)+\varepsilon|\varepsilon|^{\alpha}b_j^{1+\alpha} B\left( f_j, x, y\right)\right)^{\frac{\alpha}{2}}}  -C_\alpha\gamma _j\varepsilon |\varepsilon |^\alpha b_j^{1+\alpha}\displaystyle\fint\frac{\cos (x-y)}{\left( A( x, y)+\varepsilon|\varepsilon|^{\alpha}b_j^{1+\alpha} B\left( f_j, x, y\right)\right)^{\frac{\alpha+2}{2}}} $}}\\
& \ \ \ \ {\resizebox{.96\hsize}{!}{$\times \left(\varepsilon |\varepsilon|^\alpha
b_j^{1+\alpha}(f_j(x)-f_j(y))(f'_j(x)-f'_j(y))+2((1+\varepsilon|\varepsilon|^\alpha
b_j^{1+\alpha}f_j(x))f'_j(y)+(1+\varepsilon|\varepsilon|^\alpha
b_j^{1+\alpha}f_j(y))f'_j(x))\sin^2(\frac{x-y}{2})\right)$}} \\
& \ \ \ \ \times (\partial ^{N-1}f_j(y)-\partial ^{N-1}f_j(x))dy+l.o.t,
\end{split}%
\end{equation*}%
where \(l.o.t\) refers to lower order terms. We now utilize the fact that \(\| \partial^M f_j \|_{L^\infty} \leq C \| f_j \|_{X^{N+\alpha-1}} < \infty\) for \(M=0, 1, 2\), given that \(f_j(x) \in X^{N+\alpha-1}\) for \(N \geq 3\). By applying the Hölder inequality in conjunction with the mean value theorem, we obtain
\begin{equation*}
\begin{split}
\left\Vert \partial ^{N-1}\mathcal{F}^\alpha_{j22}\right\Vert _{L^{2}}& \leq C\left\Vert \int
\!\!\!\!\!\!\!\!\!\;{}-{}\frac{\partial ^Nf_j(x)-\partial ^Nf_j(y)}{%
|4\sin (\frac{x-y}{2})|^\alpha}dy\right\Vert _{L^{2}}+C\left\Vert \int
\!\!\!\!\!\!\!\!\!\;{}-{}\frac{\partial ^{N-1}f_j(x)-\partial
^{N-1}f_j(y)}{|4\sin (\frac{x-y}{2})|^\alpha}dy\right\Vert _{L^{2}} \\
& \leq C\Vert f_j\Vert _{X^{N+\alpha-1 }}+C\Vert f_j\Vert _{X^{N+\alpha-2}}<\infty .
\end{split}%
\end{equation*}%
Now, observe that \(\mathcal{F}^\alpha_{j23}\) is less singular than \(\mathcal{F}^\alpha_{j22}\). Therefore, it is straightforward to establish an upper bound for \(\left\Vert \partial^{N-1}\mathcal{F}^\alpha_{j23}\right\Vert_{L^{2}}\). Next, we turn our attention to the final term \(\mathcal{F}^\alpha_{j24}\) and differentiate it \(N-1\) times with respect to \(x\) to derive
\begin{equation*}
\begin{split}
&
\partial^{N-1}\mathcal{F}^\alpha_{j24}=\frac{-C_\alpha\gamma_j\varepsilon^{2+2\alpha}
b_j^{2+2\alpha}\,\partial^{N-1}f_j(x)}{(1+\varepsilon|\varepsilon|^{\alpha}b_j^{1+\alpha}
f_j(x))^2}\displaystyle\fint
\frac{f'_j(x)f'_j(y)\sin(x-y)dy}{\left( A( x, y)+\varepsilon|\varepsilon|^{\alpha}b_j^{1+\alpha} B\left( f_j, x, y\right)\right)^{\frac{\alpha}{2}}} \\
&
\quad +\frac{C_\alpha\gamma _j\varepsilon |\varepsilon |^\alpha b_j^{1+\alpha}}{1+\varepsilon
|\varepsilon |^\alpha b_j^{1+\alpha}f_j(x)}\displaystyle\fint\frac{(f_j^{\prime }(x)\partial ^Nf_j(y)+\partial
^Nf_j(x)f_j^{\prime }(y))\sin
(x-y)dy}{\left( A( x, y)+\varepsilon|\varepsilon|^{\alpha}b_j^{1+\alpha} B\left( f_j, x, y\right)\right)^{\frac{\alpha}{2}}} \\
& \quad -\frac{2C_\alpha\gamma _j\varepsilon^{2+2\alpha}
b_j^{2+2\alpha}}{1+\varepsilon
|\varepsilon |^\alpha b_j^{1+\alpha}f_j(x)}\displaystyle\fint\frac{\sin
(x-y)}{\left( A( x, y)+\varepsilon|\varepsilon|^{\alpha}b_j^{1+\alpha} B\left( f_j, x, y\right)\right)^{\frac{\alpha+2}{2}}} \\
&
\ \ \ \ {\resizebox{.96\hsize}{!}{$\times \left(\varepsilon |\varepsilon|^\alpha
b_j^{1+\alpha}(f_j(x)-f_j(y))(f'_j(x)-f'_j(y))+2((1+\varepsilon|\varepsilon|^\alpha
b_j^{1+\alpha}f_j(x))f'_j(y)+(1+\varepsilon|\varepsilon|^\alpha
b_j^{1+\alpha}f_j(y))f'_j(x))\sin^2(\frac{x-y}{2})\right)$}} \\
&
\qquad \qquad \times (f_j^{\prime }(x)\partial ^{N-1}f_j(y)+\partial
^{N-1}f_j(x)f_j^{\prime }(y))+l.o.t.
\end{split}%
\end{equation*}%
Using the definition of the space \(X^{N+\alpha-1}\), we can establish the following estimate
\begin{equation*}
\begin{split}
&{\resizebox{.96\hsize}{!}{$\left\Vert \partial ^{N-1}\mathcal{F}^\alpha_{j24}\right\Vert _{L^{2}} \leq C\varepsilon
|\varepsilon |^\alpha\left( \varepsilon |\varepsilon |^\alpha\Vert f_j^{\prime }\Vert
_{L^{\infty }}^{2}\Vert \partial ^{N-1}f_j\Vert _{L^{2}}+\Vert
f_j^{\prime }\Vert _{L^{\infty }}\Vert \partial ^Nf_j\Vert
_{L^{2}}+\varepsilon |\varepsilon |^\alpha\Vert f_j\Vert _{L^{\infty }}\Vert
f_j^{\prime }\Vert _{L^{\infty }}^{2}\Vert \partial ^{N-1}f_j\Vert
_{L^{2}}\right)$}} \\
&\qquad \qquad\quad \leq C\varepsilon |\varepsilon |^\alpha\Vert f_j\Vert _{X^{N+\alpha-1 }}<\infty .
\end{split}%
\end{equation*}%
Based on the preceding computations, we can conclude that the nonlinear functional \(\mathcal{F}^\alpha_{i2}\) is an element of \(\mathcal{Y}^{N-1}_0\).

Next, we will demonstrate the continuity of \(\mathcal{F}^\alpha_{j2}\), particularly focusing on the most singular term, \(\mathcal{F}^\alpha_{j22}\). For this purpose, we introduce the following notation. For any general function \(g_j(x)\), we define
\begin{equation*}
\Delta f_j=f_j(x)-f_j(y),\ \mbox{where}\ f_j=f_j(x),\ \mbox{and}\quad
\tilde{f}_j=f_j(y),\,\mbox{for}\,\, j=1,2 ,
\end{equation*}%
and
\begin{equation*}
D_{\alpha }(f_j)=b_j^{2+2\alpha }\varepsilon ^{2+2\alpha }\Delta
f_j^{2}+4(1+\varepsilon |\varepsilon |^{\alpha }b_j^{1+\alpha
}f_j)(1+\varepsilon |\varepsilon |^{\alpha }b_j^{1+\alpha }\tilde{f}%
_j)\sin ^{2}\left( \frac{x-y}{2}\right) .
\end{equation*}%
Consequently, for \(f_{j1}\) and \(f_{j2}\) belonging to \(X^{N+\alpha-1}\) with \(j=1,2\), we have
\begin{equation*}
\begin{split}
\mathcal{F}^\alpha_{j22}(\varepsilon ,f_{j1})& -\mathcal{F}^\alpha_{j22}(\varepsilon ,f_{j2})=C_\alpha\gamma _j\int
\!\!\!\!\!\!\!\!\!\;{}-{}\frac{(\Delta f_{j1}^{\prime }-\Delta
f_{j2}^{\prime })\cos (x-y)dy}{D_{1}(f_{j1})^{\frac{1}{2}}} \\
& +\left(C_\alpha \gamma _j\displaystyle\fint\frac{\Delta
f_{j2}^{\prime }\cos (x-y)dy}{D_{1}(f_{j1})^{\frac{\alpha}{2}}}-C_\alpha\gamma _j\int
\!\!\!\!\!\!\!\!\!\;{}-{}\frac{\Delta f_{j2}^{\prime }\cos (x-y)dy}{%
D_{1}(f_{j2})^{\frac{\alpha}{2}}}\right) \\
& =I_{1}+I_{2}.
\end{split}%
\end{equation*}%
It is clear that \(I_{1}\) can be bounded as follows
\[
\| I_{1} \|_{Y^{N-1}} \leq C \| f_{j1} - f_{j2} \|_{X^{N+\alpha-1}}.
\]
To estimate \(I_{2}\), we can apply the mean value theorem in the following manner
\begin{equation}
\begin{split}
& \frac{1}{D_{\alpha }(f_{j1})^{\frac{\alpha }{2}}}-\frac{1}{D_{\alpha
}(f_{j2})^{\frac{\alpha }{2}}}=\frac{\alpha }{2}\frac{D_{\alpha
}(f_{j2})-D_{\alpha }(f_{j1})}{D_{\alpha }(\delta _{x,y}f_{j1}+(1-\delta
_{x,y})f_{j2})^{1-\frac{\alpha }{2}}D_{\alpha }(f_{j1})^{\frac{\alpha }{2}%
}D_{\alpha }(f_{j2})^{\frac{\alpha }{2}}} \\
& {\resizebox{.98\hsize}{!}{$=\frac{\alpha}{2}\frac{b_j^{2+2\alpha}|%
\varepsilon|^{2+2\alpha}(\Delta f_{j2}^2-\Delta
f_{j1}^2)+4\varepsilon|\varepsilon|^\alpha b^{1+\alpha}_j
((f_{j2}-f_{j1})(1+\varepsilon|\varepsilon|^\alpha b^{1+\alpha}_j \tilde
f_{j2})+(\tilde f_{j2}-\tilde f_{j1})(1+\varepsilon|\varepsilon|^\alpha
b^{1+\alpha}_j
f_{j1}))\sin^2(\frac{x-y}{2})}{D_\alpha(\delta_{x,y}f_{j1}+(1-%
\delta_{x,y})f_{j2})^{1-\frac{\alpha}{2}}D_\alpha(f_{j1})^\frac{\alpha}{2}D_%
\alpha(f_{j2})^\frac{\alpha}{2}}$}},
\end{split}
\label{2-7}
\end{equation}%
for some \(\delta_{x,y} \in (0,1)\), it holds that \(D_{\alpha}(g) \sim \sin^2\left(\frac{x-y}{2}\right) \sim \frac{|x-y|^2}{4}\) as \(|x-y|\) approaches \(0\) and
\begin{equation*}
\begin{split}
& {\resizebox{.99\hsize}{!}{$ \partial^{N-1}I_2\sim C\displaystyle\fint
{}-{}\frac{\partial^{N-1}f_{j2}(x)-\partial^{N-1}f_{j2}(y)}{|\sin(%
\frac{x-y}{2})|^\alpha}\times \left(\frac{|\varepsilon|^{2+2\alpha}b_j^{2+2\alpha}(\Delta
f_{j2}^2-\Delta
f_{j1}^2)}{|x-y|^{1+\alpha}}+4\varepsilon|\varepsilon|^\alpha b_j^{1+\alpha}(f_{j2}-f_{j1}+\tilde
f_{j2}-\tilde f_{j1})\right)dy$}} \\
& \qquad \qquad \qquad +l.o.t.
\end{split}%
\end{equation*}%
Consequently, we can easily establish that
\[
\| I_{2} \|_{Y^{N-1}} \leq C \| f_{i1} - \mathcal{F}^1_{j2} \|_{X^{N+\alpha-1}},
\]
which leads to the conclusion that \(\mathcal{F}^\alpha_{j2}\) is continuous. We can now use  \eqref{2-6} and apply Taylor's formula to \(\mathcal{F}^\alpha_{j22}\), \(\mathcal{F}^\alpha_{j23}\) and \(\mathcal{F}^\alpha_{j24}\) to derive
\begin{equation}\label{2-8}
{\resizebox{.96\hsize}{!}{$	\mathcal{F}^\alpha_{j2}=C_\alpha\gamma _j\left(1-\frac{\alpha}{2}\right)\displaystyle\fint\frac{f_j(x-y)\sin
(y)dy}{|\sin (\frac{y}{2})|^\alpha}-C_\alpha \gamma _j\displaystyle\fint
\frac{(f_j^{\prime }(x)-f_j^{\prime }(x-y))\cos
(y)dy}{|\sin (\frac{y}{2})|^\alpha}+\varepsilon |\varepsilon |^\alpha\mathcal{R}%
_{j2}(\varepsilon ,f_j),$}}
\end{equation}%
where \(\mathcal{R}_{2}\) remains continuous as well.

Next, we turn our attention to the final component of the functional $\mathcal{F}^\alpha_{j}$, namely $\mathcal{F}^\alpha_{j3}$. Once again to manage the singularities at $\varepsilon =0$, we apply the Taylor formula (\ref{taylor}) on the element $\mathcal{F}^\alpha_{j3}$ with
\begin{equation}\label{eq:a}
    A_{k\ell j}=\left(d_j-d_\ell\cos\left(2k\pi  /m-(\delta_{2\ell}-\delta_{2j})\vartheta\pi /m\right)\right)^2 +d_\ell\sin^2\left(2k\pi  /m-(\delta_{2\ell}-\delta_{2j})\vartheta\pi /m\right),
\end{equation}
 \begin{equation}\label{eq:b}
 \begin{split}
         B_{k\ell j}&={\resizebox{.88\hsize}{!}{$2\left(d_j-d_\ell\cos\left(2k\pi  /m-(\delta_{2\ell}-\delta_{2j})\vartheta\pi /m\right)\right)\left(b_j\cos(x)-b_\ell\cos\left(y-2k\pi  /m+(\delta_{2\ell}-\delta_{2j})\vartheta\pi /m\right)\right)$}}\\
         &
        \qquad {\resizebox{.88\hsize}{!}{$+2d_\ell\sin\left(2k\pi  /m-(\delta_{2\ell}-\delta_{2j})\vartheta\pi /m\right)\left(b_j\sin(x)-b_\ell\sin\left(y-2k\pi  /m+(\delta_{2\ell}-\delta_{2j})\vartheta\pi /m\right)\right)$}}.
 \end{split}
 \end{equation}
    Since $\sin(\cdot)$ is an odd function, from \eqref{2-4b} and Taylor formula \eqref{taylor} we have
\begin{equation}\label{333}
\begin{split}
     \mathcal{F}^\alpha_{j41}  &=\sum_{\ell=1}^{2} \gamma_\ell\displaystyle\sum_{k=\delta_{\ell j}}^{m-1}\frac{ C_\alpha(1+\varepsilon|\varepsilon|^\alpha b_\ell^{1+\alpha}
f_\ell(x))}{\varepsilon b_\ell (1+\varepsilon|\varepsilon|^\alpha
b^{1+\alpha}_j f_j(x)) }\int\!\!\!\!\!\!\!\!\!\; {}-{} \frac{\sin(x-y+2k\pi  /m-(\delta_{2\ell}-\delta_{2j})\vartheta\pi /m)dy}{\left( A_j+\varepsilon B_j+O(\varepsilon^2) \right)^{\frac{\alpha}{2}}}\\
&
{\resizebox{.93\hsize}{!}{$\qquad+\sum\limits_{\ell=1}^{2}\gamma_\ell\displaystyle\sum_{k=\delta_{\ell j}}^{m-1}\frac{ C_\alpha
(1+\varepsilon|\varepsilon|^\alpha b_\ell^{1+\alpha}
f_\ell(x))}{ b_\ell (1+\varepsilon|\varepsilon|^\alpha
b^{1+\alpha}_j f_j(x)) }\int\!\!\!\!\!\!\!\!\!\; {}-{} \frac{|\varepsilon|^\alpha b_\ell^{1+\alpha}f(y)\sin(x-y+2k\pi  /m-(\delta_{2\ell}-\delta_{2j})\vartheta\pi /m)dy}{\left( A_j+\varepsilon B_j+O(\varepsilon^2) \right)^{\frac{\alpha}{2}}}$}}\\
    		&
            =\sum_{\ell=1}^{2}\gamma_\ell\displaystyle\sum_{k=\delta_{\ell j}}^{m-1}\frac{ C_\alpha
(1+\varepsilon|\varepsilon|^\alpha b_\ell^{1+\alpha}
f_\ell(x))}{\varepsilon b_\ell (1+\varepsilon|\varepsilon|^\alpha
b^{1+\alpha}_j f_j(x)) } \int\!\!\!\!\!\!\!\!\!\; {}-{} \frac{\sin(x-y+2k\pi  /m-(\delta_{2\ell}-\delta_{2j})\vartheta\pi /m)dy}{ A_j^{\frac{\alpha}{2}}}\\
&
{\resizebox{.93\hsize}{!}{$\qquad -\frac{\alpha}{2}\displaystyle\sum_{\ell=1}^{2} \gamma_\ell\displaystyle\sum_{k=\delta_{\ell j}}^{m-1}\frac{ C_\alpha(1+\varepsilon|\varepsilon|^\alpha b_\ell^{1+\alpha}
f_\ell(x))}{ b_\ell (1+\varepsilon|\varepsilon|^\alpha
b^{1+\alpha}_j f_j(x)) } \int\!\!\!\!\!\!\!\!\!\; {}-{}\int_0^1 \frac{(B_j+O(\varepsilon))\sin(x-y+2k\pi  /m-(\delta_{2\ell}-\delta_{2j})\vartheta\pi /m)dtdy}{\left( A_j+\varepsilon tB_j+O(\varepsilon^2) \right)^{\frac{\alpha}{2}+1}}$}}\\
    		&
            \ \ \ \ {\resizebox{.9\hsize}{!}{$+\sum\limits_{\ell=1}^{2}\gamma_\ell\displaystyle\sum_{k=\delta_{\ell j}}^{m-1}\frac{ C_\alpha b_\ell^{\alpha}|\varepsilon|^{\alpha}
(1+\varepsilon|\varepsilon|^\alpha b_\ell^{1+\alpha}
f_\ell(x))}{  (1+\varepsilon|\varepsilon|^\alpha
b^{1+\alpha}_j f_j(x)) }\int\!\!\!\!\!\!\!\!\!\; {}-{} \frac{f(y)\sin(x-y+2k\pi  /m-(\delta_{2\ell}-\delta_{2j})\vartheta\pi /m)dy}{\left( A_j+\varepsilon B_j+O(\varepsilon^2) \right)^{\frac{\alpha}{2}}}$}}\\
    		&
                =-\frac{\alpha}{2}\sum_{\ell=1}^{2} \gamma_\ell\displaystyle\sum_{k=\delta_{\ell j}}^{m-1}\frac{ C_\alpha}{ b_\ell  } \int\!\!\!\!\!\!\!\!\!\; {}-{} \frac{B_j\sin(x-y+2k\pi  /m-(\delta_{2\ell}-\delta_{2j})\vartheta\pi /m)dy}{A_j^{\frac{\alpha}{2}+1}}+\varepsilon\mathcal{R}_{j41}(\varepsilon,f_\ell,f_j)\\
    		&
            {\resizebox{.95\hsize}{!}{$=-\frac{\alpha C_\alpha}{2}\displaystyle\sum_{\ell=1}^{2} \gamma_\ell\displaystyle\sum_{k=\delta_{\ell j}}^{m-1} \frac{(d_j-d_\ell\cos\left(2k\pi  /m-(\delta_{2\ell}-\delta_{2j})\vartheta\pi /m\right))\sin (x)}{\left((d_j-d_\ell\cos(2k\pi  /m-(\delta_{2\ell}-\delta_{2j})\vartheta\pi /m))^2 +d_\ell\sin^2(k\pi  /m-(\delta_{2\ell}-\delta_{2j})\vartheta\pi /m)\right)^{\frac{\alpha}{2}+1}}+\varepsilon\mathcal{R}_{j41}(\varepsilon,f_\ell,f_j)$}},
\end{split}
\end{equation}
Similarly, for $\mathcal{F}^\alpha_{j31}$ we define
 $$A_{0 0 j}=d_j^2,$$
 \begin{equation*}
   B_{0 0 j}=2d_j\left(b_j\cos(x)-b_0\cos\left(y-\delta_{2j}\vartheta\pi /m\right)\right).
 \end{equation*}
 Thus, using again Taylor formula \eqref{taylor}
\begin{equation}\label{fj31}
    \begin{split}
     \mathcal{F}^\alpha_{j31}&=   \frac{\gamma_{0} C_\alpha
(1+\varepsilon|\varepsilon|^\alpha b_{0}^{1+\alpha}
f_{0}(x))}{\varepsilon b_{0} (1+\varepsilon|\varepsilon|^\alpha
b^{1+\alpha}_j f_j(x)) }\displaystyle\fint
\frac{(1+\varepsilon|\varepsilon|^\alpha b_{0}^{1+\alpha}
f_{0}(y))\sin\left(x-y+\delta_{2j}\vartheta\pi /m\right)dy}{\left(A_{00j}+\varepsilon B_{00j}\right)^{\alpha/2}}\\
&
=\frac{\gamma_{0} C_\alpha
}{\varepsilon b_{0} }\displaystyle\fint
\frac{\sin\left(x-y+\delta_{2j}\vartheta\pi /m\right)dy}{\left(A_{00j}+\varepsilon B_{00j}\right)^{\alpha/2}}+\varepsilon\mathcal{R}_{j31}(\varepsilon,f_0,f_j)\\
&
=-\frac{\alpha}{2}\frac{\gamma_{0} C_\alpha
}{ b_{0} }\displaystyle\fint
\frac{B_{00j}\sin\left(x-y+\delta_{2j}\vartheta\pi /m\right)dy}{\left(A_{00j}\right)^{\alpha/2+1}}+\varepsilon\mathcal{R}_{j31}(\varepsilon,f_0,f_j)\\
&
=-\frac{\alpha\gamma_{0} C_\alpha}{2}
\displaystyle\fint
\frac{-2\cos\left(y-\delta_{2j}\vartheta\pi /m\right)\sin\left(x-y+\delta_{2j}\vartheta\pi /m\right)dy}{d_j^{1+\alpha}}+\varepsilon\mathcal{R}_{31}(\varepsilon,f_0,f_j)\\
&
=-\frac{\alpha\gamma_{0} C_\alpha}{2d_j^{1+\alpha}}\left(-2\sin\left(x+\delta_{2j}\vartheta\pi /m\right)
\displaystyle\fint
\cos\left(y-\delta_{2j}\vartheta\pi /m\right)\cos(y)dy\right.\\
&
\left.\qquad+2\cos\left(x+\delta_{2j}\vartheta\pi /m\right)\displaystyle\fint\cos\left(y-\delta_{2j}\vartheta\pi /m\right)\sin(y)\right)+\varepsilon\mathcal{R}_{j31}(\varepsilon,f_0,f_j)\\
&
 {\resizebox{.93\hsize}{!}{$=-\frac{\alpha\gamma_{0} C_\alpha}{2d_j^{1+\alpha}}\left(-\sin\left(x+\delta_{2j}\vartheta\pi /m\right)\cos\left(\delta_{2j}\vartheta\pi /m\right)
\pi+\cos\left(x+\delta_{2j}\vartheta\pi /m\right)\sin\left(\delta_{2j}\vartheta/m\right)\right)+\varepsilon\mathcal{R}_{j31}(\varepsilon,f_0,f_j)$}}\\
&
=-\frac{\alpha\gamma_{0} C_\alpha}{2d_j^{1+\alpha}}\sin(x)+\varepsilon\mathcal{R}_{j31}(\varepsilon,f_0,f_j) .
    \end{split}
\end{equation}
Here, $\mathcal{R}_{j31}$ refers to a bounded function that is also continuous.
Notice that for $z_j\in  \partial \mathcal{O}_{j}^\varepsilon$ we have that
\begin{equation*}
|(z_j(x)+(d_j,0))-\nu_{k 0 j}(z_0(y)+(d_0,0))|^\alpha ,
\end{equation*}%
and
\begin{equation*}
|(z_j(x)+(d_j,0))-\nu_{k \ell j}(z_\ell(y)+(d_\ell,0))|^\alpha ,
\end{equation*}%
have positive lower bounds,
as a result, we conclude that $\mathcal{F}^\alpha_{j3}$ and $\mathcal{F}^\alpha_{j4}$ is less singular than $\mathcal{F}^\alpha_{j2}$. Moreover, it follows that $\mathcal{F}^\alpha_{j3}(\varepsilon, f, \lambda)$ and $\mathcal{F}^\alpha_{j4}(\varepsilon, f, \lambda)$ remains continuous.
Thus, we conclude the proof of the continuity of the functional $\mathcal{F}^\alpha_{j}$.

Now, we deal with the remaining term, namely $\mathcal{F}^\alpha_0$. By using \eqref{f01}, we can rewrite $\mathcal{F}^\alpha_{0 1}:\left(-\varepsilon_0, \varepsilon_0\right) \times \mathcal{B}_X\times \Lambda \rightarrow \mathcal{Y}^{N-1}$
as
\begin{equation}\label{2-1}
\begin{split}
  \mathcal{F}^\alpha_{0 1}=\varepsilon|\varepsilon|^\alpha\mathcal{R}_{01}(\varepsilon,f_0) ,
\end{split}
\end{equation}
where \(\mathcal{R}_{01}(\varepsilon, f_0)\) is continuous. Consequently, we can infer that \(\mathcal{F}^\alpha_{0 1}\) is also continuous.  Similarly, following the same lines as \eqref{2-8}  we can deduce that
\begin{equation}\label{f02}
    \mathcal{F}^\alpha_{02}=C_\alpha\gamma _{0}\left(1-\frac{\alpha}{2}\right)\displaystyle\fint   \frac{f_0( x- y)\sin( y)d y}{\left(4\sin^2(\frac{ y}{2})\right)^{\frac{\alpha}{2}}}-C_\alpha \gamma_0\displaystyle\fint \frac{(f'_0( x)-f'_0( x- y))\cos( y)d y}{\left(4\sin^2(\frac{ y}{2})\right)^{\frac{\alpha}{2}}}+\varepsilon|\varepsilon|^\alpha\mathcal{R}_{02}(\varepsilon ,f_0) ,
\end{equation}

Now, we compute $\mathcal{F}^\alpha_{031}$ using Taylor formula \eqref{taylor} and the following
\begin{equation}\label{eq:a2}
    A_{k\ell 0}=d_\ell^2 ,
\end{equation}
 \begin{equation}\label{eq:b2}
 \begin{split}
         B_{k\ell 0}&={\resizebox{.88\hsize}{!}{$-2d_\ell\cos\left(2k\pi  /m-\delta_{2\ell}\vartheta\pi /m\right)\left(b_0\cos(x)-b_\ell\cos\left(y-2k\pi  /m+\delta_{2\ell}\vartheta\pi /m\right)\right)$}}\\
         &
        \qquad {\resizebox{.88\hsize}{!}{$+2d_\ell\sin\left(2k\pi  /m-\delta_{2\ell}\vartheta\pi /m\right)\left(b_0\sin(x)-b_\ell\sin\left(y-2k\pi  /m+\delta_{2\ell}\vartheta\pi /m\right)\right)$}} ,
 \end{split}
 \end{equation}
we have
\begin{equation}\label{f031}
    \begin{split}
        \mathcal{F}^\alpha_{031}&{\resizebox{.95\hsize}{!}{$=\displaystyle\sum_{\ell=1}^2 \gamma_\ell \displaystyle\sum_{k=0}^{m-1}\frac{ C_\alpha
(1+\varepsilon|\varepsilon|^\alpha b_\ell^{1+\alpha}
f_\ell(x))}{\varepsilon b_\ell (1+\varepsilon|\varepsilon|^\alpha
b^{1+\alpha}_0 f_0(x)) }\displaystyle\fint
\frac{(1+\varepsilon|\varepsilon|^\alpha b_\ell^{1+\alpha}
f_\ell(y))\sin\left(x-y+2k\pi  /m-\delta_{2\ell}\vartheta\pi /m\right)dy}{(A_{k\ell 0}+\varepsilon B_{k\ell 0})^{\alpha/2}}$}}\\
&
=\displaystyle\sum_{\ell=1}^2 \gamma_\ell \displaystyle\sum_{k=0}^{m-1}\frac{ C_\alpha
}{\varepsilon b_\ell  }\displaystyle\fint
\frac{\sin\left(x-y+2k\pi  /m-\delta_{2\ell}\vartheta\pi /m\right)dy}{(A_{k\ell 0}+\varepsilon B_{k\ell 0})^{\alpha/2}}+\varepsilon\mathcal{R}_{031}(\varepsilon,f_\ell,f_0)\\
&
=\displaystyle\sum_{\ell=1}^2 \gamma_\ell \displaystyle\sum_{k=0}^{m-1}\frac{ C_\alpha
}{\varepsilon b_\ell  }\displaystyle\fint
\frac{\sin\left(x-y+2k\pi  /m-\delta_{2\ell}\vartheta\pi /m\right)dy}{A_{k\ell 0}^{\alpha/2}}\\
&
{\resizebox{.95\hsize}{!}{$\quad -\displaystyle\sum_{\ell=1}^2 \gamma_\ell \displaystyle\sum_{k=0}^{m-1}\frac{ \alpha C_\alpha
}{2 b_\ell  }\displaystyle\fint\int_0^1
\frac{B_{k\ell 0}\sin\left(x-y+2k\pi  /m-\delta_{2\ell}\vartheta\pi /m\right)dt\,dy}{(A_{k\ell 0}+\varepsilon t B_{k\ell 0})^{\frac{\alpha}{2}+1}}+\varepsilon\mathcal{R}_{031}(\varepsilon,f_\ell,f_0)$}}\\
&
= -\frac{ \alpha C_\alpha
}{2   }\displaystyle\sum_{\ell=1}^2 \gamma_\ell \displaystyle\sum_{k=0}^{m-1}\frac{ 1
}{ b_\ell  }\displaystyle\fint
\frac{B_{k\ell 0}\sin\left(x-y+2k\pi  /m-\delta_{2\ell}\vartheta\pi /m\right)dy}{A_{k\ell 0}^{\frac{\alpha}{2}+1}}+\varepsilon\mathcal{R}_{031}(\varepsilon,f_\ell,f_0)\\
&
=\frac{\alpha C_\alpha}{2}\displaystyle\sum_{\ell=1}^{2} \gamma_\ell\displaystyle\sum_{k=0}^{m-1} \frac{\cos\left(2k\pi  /m-(\delta_{2\ell}-\delta_{2j})\vartheta\pi /m\right)\sin (x)}{d_\ell^{1+\alpha}}+\varepsilon\mathcal{R}_{031}(\varepsilon,f_\ell,f_0)\\
        &
        =\varepsilon\mathcal{R}_{031}(\varepsilon,f_\ell,f_0) ,
    \end{split}
\end{equation}
where, $\mathcal{R}_{j31}$ is a bounded function that is also continuous. Following the same lines as the continuity of $\mathcal{F}^\alpha_{j2}$ and $\mathcal{F}^\alpha_{j3}$, we can infer  that $\mathcal{F}^\alpha_{03}(\varepsilon, f, \lambda)$ is also continuous.
Thus, we conclude the proof of the continuity of the functional $\mathcal{F}^\alpha$.
\end{proof}
By combining \eqref{2-1b}, \eqref{2-8}, \eqref{333},\eqref{fj31}, \eqref{2-1} and \eqref{f031}, we obtain the following expressions
\begin{equation}\label{3-7}
    \begin{cases}
&{\resizebox{.93\hsize}{!}{$\mathcal{F}^\alpha_{0}(\lambda)=C_\alpha\gamma _{0}\left(1-\frac{\alpha}{2}\right)\displaystyle\fint   \frac{f_0( x- y)\sin( y)d y}{\left(4\sin^2(\frac{ y}{2})\right)^{\frac{\alpha}{2}}}-C_\alpha \gamma_0\displaystyle\fint \frac{(f'_0( x)-f'_0( x- y))\cos( y)d y}{\left(4\sin^2(\frac{ y}{2})\right)^{\frac{\alpha}{2}}}+\varepsilon\mathcal{R}_{0}(\varepsilon ,f)$}}\\
        &
        {\mathcal{F}}_1^{\alpha}(\lambda)=\Omega d_1\sin (x)-\frac{\widehat{C}_\alpha}{2d_1^{\alpha+1}}\Big[\gamma_0+\frac{\gamma_1}{2}S_\alpha +\gamma_2T_\alpha^+( d ,\vartheta)\Big]\sin(x)\\
        &
    {\resizebox{.93\hsize}{!}{$   \qquad \qquad + C_\alpha\gamma _1\left(1-\frac{\alpha}{2}\right)\displaystyle\fint\frac{f_1(x-y)\sin
(y)dy}{\left(4\sin^2(\frac{ y}{2})\right)^{\frac{\alpha}{2}}}-C_\alpha \gamma _1\displaystyle\fint
\frac{(f_1^{\prime }(x)-f_1^{\prime }(x-y))\cos
(y)dy}{\left(4\sin^2(\frac{ y}{2})\right)^{\frac{\alpha}{2}}}+\varepsilon\mathcal{R}_{1}(\varepsilon ,f)$}}
\\
&{\mathcal{F}}_2^{\alpha}(\lambda)=\Omega d_2\sin (x) -\frac{\widehat{C}_\alpha}{2d_2^{\alpha+1}}\Big[\gamma_0+\gamma_1 T_\alpha^-( d ,\vartheta)+\frac{\gamma_2}{2}S_\alpha \Big]\sin(x)\\
&
{\resizebox{.93\hsize}{!}{$\qquad \qquad + C_\alpha\gamma _2\left(1-\frac{\alpha}{2}\right)\displaystyle\fint\frac{f_2(x-y)\sin
(y)dy}{\left(4\sin^2(\frac{ y}{2})\right)^{\frac{\alpha}{2}}}-C_\alpha \gamma _2\displaystyle\fint
\frac{(f_2^{\prime }(x)-f_2^{\prime }(x-y))\cos
(y)dy}{\left(4\sin^2(\frac{ y}{2})\right)^{\frac{\alpha}{2}}}+\varepsilon\mathcal{R}_{2}(\varepsilon ,f).$}}
    \end{cases}
\end{equation}%
where $\mathcal{R}_{j}$, for $j=0,1,2$, are bounded and smooth
and $T_\alpha^-( d ,\vartheta)$ and $S_\alpha$ are defined by
\begin{equation}\label{S}
  \frac{S_\alpha}{2}:=  \frac12\sum_{k=1}^{m-1}{ \Big(2\sin\big({\frac{k\pi  }{m}}\big)  \Big)^{-\alpha}},
\end{equation}
\begin{equation}\label{t}
   {T_\alpha^\pm( d ,\vartheta)}:= \sum_{k=0}^{m-1}\frac{1 -  d ^{\pm 1} \cos\big({\frac{(2k\pm\vartheta)\pi }{m}}\big)}{\big(1+( d ^{\pm 1})^2 - 2 d ^{\pm 1} \cos\big({\frac{(2k\pm\vartheta)\pi }{m}}\big) \big)^{\frac\alpha2+1}}
\end{equation}
The subsequent proposition focuses on the $C^1$ smoothness of the
 functional $\mathcal{F}^\alpha.$
\begin{proposition}\label{lem2-3}
There exists $\varepsilon_0>0$ and a small neighborhood $\Lambda$ of $\lambda^*$  such that  the Gateaux derivatives \( \partial_{f_0} \mathcal{F}^\alpha_0(\varepsilon, f, \lambda) h_0 : (-\varepsilon_0, \varepsilon_0) \times \mathcal{B}_X \times \Lambda \rightarrow \mathcal{Y}^{N-1} \), \( \partial_{f_j} \mathcal{F}^\alpha_j(\varepsilon, f, \lambda) h_j : (-\varepsilon_0, \varepsilon_0) \times \mathcal{B}_X \times \Lambda \rightarrow \mathcal{Y}^{N-1}  \) and  \( \partial_{f_\ell} \mathcal{F}^\alpha_j(\varepsilon, f, \lambda) h_\ell : (-\varepsilon_0, \varepsilon_0) \times \mathcal{B}_X \times \Lambda \rightarrow \mathcal{Y}^{N-1}  \)  exist and are continuous.
\end{proposition}
\begin{proof}
We claim $\partial_{ f_j} \mathcal{F}^\alpha_{j2}(\varepsilon, f_j)h_j=\partial_{ f_j} \mathcal{F}^\alpha_{j11}+\partial_{ f_j} \mathcal{F}^\alpha_{j22}+\partial_{ f_j} \mathcal{F}^\alpha_{j23}+\partial_{ f_j}\mathcal{F}^\alpha_{j24}$ is continuous, where
	\begin{equation}\label{2-110}
		\begin{split}		
    			&{\resizebox{.99\hsize}{!}{$\partial_{ f_0} \mathcal{F}^\alpha_{021}=C_\alpha\gamma _0\displaystyle\fint \frac{h_0( y)\sin( x- y)d y}{\left( |\varepsilon|^{2+2\alpha}b_0^{2+2\alpha}\left(f_0( x)-f_0( y)\right)^2+4(1+\varepsilon|\varepsilon|^\alpha b_0^{1+\alpha} f_0( x))(1+\varepsilon|\varepsilon|^\alpha b_0^{1+\alpha} f_0( y))\sin^2\left(\frac{ x- y}{2}\right)\right)^{\frac{\alpha}{2}}}$}}\\
			& \ \ \ \ {\resizebox{.9\hsize}{!}{$-\frac{\alpha C_\alpha \gamma_0}{2}\displaystyle \fint \frac{(1+\varepsilon|\varepsilon|^\alpha b_0^{1+\alpha}  f_0( y))\sin( x- y)}{\left( |\varepsilon|^{2+2\alpha}b_0^{2+2\alpha}\left(f_0( x)-f_0( y)\right)^2+4(1+\varepsilon|\varepsilon|^\alpha b_0^{1+\alpha} f_0( x))(1+\varepsilon|\varepsilon|^\alpha b_0^{1+\alpha}  f_0( y))\sin^2\left(\frac{ x- y}{2}\right)\right)^{\frac{\alpha+2}{2}}}$}}\\
			& \ \ \ \ \times\bigg(2\varepsilon|\varepsilon|^\alpha b_0^{1+\alpha}(f_0( x)-f_0( y))(h_0( x)-h_0( y))\\
			& \ \ \ \ \ \ \ \ +4(h_0( x)(1+\varepsilon|\varepsilon|^\alpha b_0^{1+\alpha}  f_0( y))+h_0( y)(1+\varepsilon|\varepsilon|^\alpha b_0^{1+\alpha} f_0( x)))\sin^2\left(\frac{ x- y}{2}\right)\bigg)d y,
		\end{split}
	\end{equation}
    \begin{equation}\label{2-120}
    	\begin{split}
    		&\partial_{ f_0} \mathcal{F}^\alpha_{022}=C_\alpha\gamma_0{\resizebox{.85\hsize}{!}{$\displaystyle\fint \frac{(h'_0( y)-h'_0( x))\cos( x- y)d y}{\left( |\varepsilon|^{2+2\alpha}\left(f_0( x)-f_0( y)\right)^2+4(1+\varepsilon|\varepsilon|^\alpha b_0^{1+\alpha} f_0( x))(1+\varepsilon|\varepsilon|^\alpha b_0^{1+\alpha} f_0( y))\sin^2\left(\frac{ x- y}{2}\right)\right)^{\frac{\alpha}{2}}}$}}\\
    		&  \ \ \ {\resizebox{.99\hsize}{!}{$-\frac{\alpha C_\alpha\gamma_0\varepsilon|\varepsilon|^\alpha b_0^{1+\alpha}}{2}\displaystyle \fint \frac{(g'_0( y)-g'_0( x))\cos( x- y)}{\left( |\varepsilon|^{2+2\alpha}\left(f_0( x)-f_0( y)\right)^2+4(1+\varepsilon|\varepsilon|^\alpha b_0^{1+\alpha} f_0( x))(1+\varepsilon|\varepsilon|^\alpha b_0^{1+\alpha} f_0( y))\sin^2\left(\frac{ x- y}{2}\right)\right)^{\frac{\alpha+2}{2}}}$}}\\
    			&  \ \ \ \times\bigg(\varepsilon|\varepsilon|^\alpha b_0^{1+\alpha}(2(f_0( x)-f_0( y))(h_0( x)-h_0( y))\\
			& \ \ \ \ \ \ \ \ +4(h_0( x)(1+\varepsilon|\varepsilon|^\alpha b_0^{1+\alpha}  f_0( y))+h_0( y)(1+\varepsilon|\varepsilon|^\alpha b_0^{1+\alpha} f_0( x)))\sin^2\left(\frac{ x- y}{2}\right)\bigg)d y,
    	\end{split}
    \end{equation}
    \begin{equation}\label{2-130}
    	\begin{split}
    		&\partial_{ f_0} \mathcal{F}^\alpha_{023}={\resizebox{.9\hsize}{!}{$\frac{C_\alpha \gamma_0 \varepsilon|\varepsilon|^\alpha b_0^{1+\alpha} h'_0( x)}{1+\varepsilon|\varepsilon|^\alpha b_0^{1+\alpha} f_0( x)}\displaystyle\fint \frac{(f_0( x)-f_0( y))\cos( x- y)d y}{\left( |\varepsilon|^{2+2\alpha}b_0^{2+2\alpha}\left(f_0( x)-f_0( y)\right)^2+4(1+\varepsilon|\varepsilon|^\alpha b_0^{1+\alpha} f_0( x))(1+\varepsilon|\varepsilon|^\alpha b_0^{1+\alpha} f_0( y))\sin^2\left(\frac{ x- y}{2}\right)\right)^{\frac{\alpha}{2}}}$}}\\
    		&
            \ \ \ \ +{\resizebox{.93\hsize}{!}{$\frac{C_\alpha  \gamma_0\varepsilon|\varepsilon|^\alpha b_0^{1+\alpha}
            f'_0( x)}{1+\varepsilon|\varepsilon|^\alpha b_0^{1+\alpha} f_0( x)}\displaystyle\fint\frac{(h_0( x)-h_0( y))\cos( x- y)d y}{\left( |\varepsilon|^{2+2\alpha}b_0^{2+2\alpha}\left(f_0( x)-f_0( y)\right)^2+4(1+\varepsilon|\varepsilon|^\alpha b_0^{1+\alpha} f_0( x))(1+\varepsilon|\varepsilon|^\alpha b_0^{1+\alpha} f_0( y))\sin^2\left(\frac{ x- y}{2}\right)\right)^{\frac{\alpha}{2}}}$}}\\
            &
            \ \ \ \ -{\resizebox{.9\hsize}{!}{$\frac{C_\alpha  \gamma_0|\varepsilon|^{2+2\alpha}b_0^{2+2\alpha} f'_0( x)h_0( x)}{(1+\varepsilon|\varepsilon|^\alpha b_0^{1+\alpha}
            f_0( x))^2}\displaystyle\fint \frac{(f_0( x)-f_0( y))\cos( x- y)d y}{\left( |\varepsilon|^{2+2\alpha}b_0^{2+2\alpha}\left(f_0( x)-f_0( y)\right)^2+4(1+\varepsilon|\varepsilon|^\alpha b_0^{1+\alpha} f_0( x))(1+\varepsilon|\varepsilon|^\alpha b_0^{1+\alpha} f_0( y))\sin^2\left(\frac{ x- y}{2}\right)\right)^{\frac{\alpha}{2}}}$}}\\
    		&
            \ \ \ \ -{\resizebox{.9\hsize}{!}{$\frac{\alpha C_\alpha \gamma_0 |\varepsilon|^{2+2\alpha}b_0^{2+2\alpha} f'_0( x)}{2(1+\varepsilon|\varepsilon|^\alpha b_0^{1+\alpha} f_0( x))}\displaystyle\fint \frac{(f_0( x)-f_0( y))\cos( x- y)}{\left( |\varepsilon|^{2+2\alpha}b_0^{2+2\alpha}\left(f_0( x)-f_0( y)\right)^2+4(1+\varepsilon|\varepsilon|^\alpha b_0^{1+\alpha} f_0( x))(1+\varepsilon|\varepsilon|^\alpha b_0^{1+\alpha} f_0( y))\sin^2\left(\frac{ x- y}{2}\right)\right)^{\frac{\alpha+2}{2}}}$}}\\
    		&
            \ \ \ \ \times\bigg(\varepsilon|\varepsilon|^\alpha b_0^{1+\alpha}(2(f_0( x)-f_0( y))(h_0( x)-h_0( y))\\
			& \ \ \ \ \ \ \ \ +4(h_0( x)(1+\varepsilon|\varepsilon|^\alpha b_0^{1+\alpha}  f_0( y))+h_0( y)(1+\varepsilon|\varepsilon|^\alpha b_0^{1+\alpha} f_0( x)))\sin^2\left(\frac{ x- y}{2}\right)\bigg)d y
    	\end{split}
    \end{equation}
    and
    \begin{equation}\label{2-140}
    	\begin{split}
    		&\partial_{ f_0} \mathcal{F}^\alpha_{024}={\resizebox{.9\hsize}{!}{$\frac{C_\alpha \gamma_0 |\varepsilon|^{1+\alpha}b_0^{1+\alpha}}{1+\varepsilon|\varepsilon|^\alpha b_0^{1+\alpha} f_0( x)}\displaystyle\fint \frac{(h'_0( y)f'_0( x)+h'_0( x)f'_0( y))\sin( x- y)d y}{\left( |\varepsilon|^{2+2\alpha}b_0^{2+2\alpha} \left(f_0( x)-f_0( y)\right)^2+4(1+\varepsilon|\varepsilon|^\alpha b_0^{1+\alpha} f_0( x))(1+\varepsilon|\varepsilon|^\alpha b_0^{1+\alpha} f_0( y))\sin^2\left(\frac{ x- y}{2}\right)\right)^{\frac{\alpha}{2}}}$}}\\
    		& \ \ \ \ -{\resizebox{.93\hsize}{!}{$\frac{\alpha C_\alpha \gamma_0|\varepsilon|^{2+2\alpha}b_0^{2+2\alpha}}{2(1+\varepsilon|\varepsilon|^\alpha b_0^{1+\alpha} f_0( x))}\displaystyle\fint  \frac{f'_0( x)f'_0( y)\sin( x- y)}{\left( |\varepsilon|^{2+2\alpha}b_0^{2+2\alpha} \left(f_0( x)-f_0( y)\right)^2+4(1+\varepsilon|\varepsilon|^\alpha b_0^{1+\alpha} f_0( x))(1+\varepsilon|\varepsilon|^\alpha b_0^{1+\alpha} f_0( y))\sin^2\left(\frac{ x- y}{2}\right)\right)^{\frac{\alpha+2}{2}}}$}}\\
    		& \ \ \ \ \times\bigg(\varepsilon|\varepsilon|^\alpha b_0^{1+\alpha}(2(f_0( x)-f_0( y))(h_0( x)-h_0( y))\\
			& \ \ \ \ \ \ \ \ +4(h_0( x)(1+\varepsilon|\varepsilon|^\alpha b_0^{1+\alpha}  f_0( y))+h_0( y)(1+\varepsilon|\varepsilon|^\alpha b_0^{1+\alpha} f_0( x)))\sin^2\left(\frac{ x- y}{2}\right)\bigg)d y.\\
			&\ \ \ \ \ \ \ \ \ {\resizebox{.93\hsize}{!}{$-\frac{C_{\alpha}\gamma_0|\varepsilon|^{1+\alpha}b_0^{1+\alpha}f'_{j}( x)}{(1+\varepsilon|\varepsilon|^{\alpha}b_0^{1+\alpha}f_{j}( x))^{2}}\displaystyle\fint  \frac{f'_{j}( y)h_0( x)\sin{( x- y)}}{\left( |\varepsilon|^{2+2\alpha}b_0^{2+2\alpha} \left(f_0( x)-f_0( y)\right)^2+4(1+\varepsilon|\varepsilon|^\alpha b_0^{1+\alpha} f_0( x))(1+\varepsilon|\varepsilon|^\alpha b_0^{1+\alpha} f_0( y))\sin^2\left(\frac{ x- y}{2}\right)\right)^{\frac{\alpha}{2}}}\, d y.$}}
    	\end{split}
    \end{equation}
    \begin{equation}\label{partial_{j}20}
  \partial_{ f_0} \mathcal{F}^\alpha_{01}(\varepsilon ,f_0,\lambda) = O(\varepsilon) .
\end{equation}
\begin{equation}\label{partial_{j}30}
  \partial_{ f_0}  \mathcal{F}^\alpha_{03}(\varepsilon ,f_0\lambda) = O(\varepsilon) .
\end{equation}
and
\begin{equation}\label{partial_j20}
  \partial_{ f_\ell} \mathcal{F}^\alpha_{0}(\varepsilon ,f,\lambda) = O(\varepsilon) .
\end{equation}
 We also can verify that $\partial_{ f_j} \mathcal{F}^\alpha_{j2}(\varepsilon, f_j)h_j=\partial_{ f_j} \mathcal{F}^\alpha_{j11}+\partial_{ f_j} \mathcal{F}^\alpha_{j22}+\partial_{ f_j} \mathcal{F}^\alpha_{j23}+\partial_{ f_j}\mathcal{F}^\alpha_{j24}$ is continuous. To do this we shall compute the Gateaux derivatives as follows
	\begin{equation}\label{2-11}
		\begin{split}		
    			&{\resizebox{.99\hsize}{!}{$\partial_{ f_j} \mathcal{F}^\alpha_{j21}=C_\alpha\gamma _j\displaystyle\fint \frac{h_j( y)\sin( x- y)d y}{\left( |\varepsilon|^{2+2\alpha}b_j^{2+2\alpha}\left(f_j( x)-f_j( y)\right)^2+4(1+\varepsilon|\varepsilon|^\alpha b_j^{1+\alpha} f_j( x))(1+\varepsilon|\varepsilon|^\alpha b_j^{1+\alpha} f_j( y))\sin^2\left(\frac{ x- y}{2}\right)\right)^{\frac{\alpha}{2}}}$}}\\
			& \ \ \ \ {\resizebox{.9\hsize}{!}{$-\frac{\alpha C_\alpha \gamma_j}{2}\displaystyle \fint \frac{(1+\varepsilon|\varepsilon|^\alpha b_j^{1+\alpha}  f_j( y))\sin( x- y)}{\left( |\varepsilon|^{2+2\alpha}b_j^{2+2\alpha}\left(f_j( x)-f_j( y)\right)^2+4(1+\varepsilon|\varepsilon|^\alpha b_j^{1+\alpha} f_j( x))(1+\varepsilon|\varepsilon|^\alpha b_j^{1+\alpha}  f_j( y))\sin^2\left(\frac{ x- y}{2}\right)\right)^{\frac{\alpha+2}{2}}}$}}\\
			& \ \ \ \ \times\bigg(2\varepsilon|\varepsilon|^\alpha b_j^{1+\alpha}(f_j( x)-f_j( y))(h_j( x)-h_j( y))\\
			& \ \ \ \ \ \ \ \ +4(h_j( x)(1+\varepsilon|\varepsilon|^\alpha b_j^{1+\alpha}  f_j( y))+h_j( y)(1+\varepsilon|\varepsilon|^\alpha b_j^{1+\alpha} f_j( x)))\sin^2\left(\frac{ x- y}{2}\right)\bigg)d y,
		\end{split}
	\end{equation}
    \begin{equation}\label{2-12}
    	\begin{split}
    		&\partial_{ f_j} \mathcal{F}^\alpha_{j22}=C_\alpha\gamma_j{\resizebox{.85\hsize}{!}{$\displaystyle\fint \frac{(h'_j( y)-h'_j( x))\cos( x- y)d y}{\left( |\varepsilon|^{2+2\alpha}\left(f_j( x)-f_j( y)\right)^2+4(1+\varepsilon|\varepsilon|^\alpha b_j^{1+\alpha} f_j( x))(1+\varepsilon|\varepsilon|^\alpha b_j^{1+\alpha} f_j( y))\sin^2\left(\frac{ x- y}{2}\right)\right)^{\frac{\alpha}{2}}}$}}\\
    		&  \ \ \ {\resizebox{.99\hsize}{!}{$-\frac{\alpha C_\alpha\gamma_j\varepsilon|\varepsilon|^\alpha b_j^{1+\alpha}}{2}\displaystyle \fint \frac{(f'_j( y)-f'_j( x))\cos( x- y)}{\left( |\varepsilon|^{2+2\alpha}\left(f_j( x)-f_j( y)\right)^2+4(1+\varepsilon|\varepsilon|^\alpha b_j^{1+\alpha} f_j( x))(1+\varepsilon|\varepsilon|^\alpha b_j^{1+\alpha} f_j( y))\sin^2\left(\frac{ x- y}{2}\right)\right)^{\frac{\alpha+2}{2}}}$}}\\
    			&  \ \ \ \times\bigg(\varepsilon|\varepsilon|^\alpha b_j^{1+\alpha}(2(f_j( x)-f_j( y))(h_j( x)-h_j( y))\\
			& \ \ \ \ \ \ \ \ +4(h_j( x)(1+\varepsilon|\varepsilon|^\alpha b_j^{1+\alpha}  f_j( y))+h_j( y)(1+\varepsilon|\varepsilon|^\alpha b_j^{1+\alpha} f_j( x)))\sin^2\left(\frac{ x- y}{2}\right)\bigg)d y,
    	\end{split}
    \end{equation}
    \begin{equation}\label{2-13}
    	\begin{split}
    		&\partial_{ f_j} \mathcal{F}^\alpha_{j23}={\resizebox{.9\hsize}{!}{$\frac{C_\alpha \gamma_j \varepsilon|\varepsilon|^\alpha b_j^{1+\alpha} h'_j( x)}{1+\varepsilon|\varepsilon|^\alpha b_j^{1+\alpha} f_j( x)}\displaystyle\fint \frac{(f_j( x)-f_j( y))\cos( x- y)d y}{\left( |\varepsilon|^{2+2\alpha}b_j^{2+2\alpha}\left(f_j( x)-f_j( y)\right)^2+4(1+\varepsilon|\varepsilon|^\alpha b_j^{1+\alpha} f_j( x))(1+\varepsilon|\varepsilon|^\alpha b_j^{1+\alpha} f_j( y))\sin^2\left(\frac{ x- y}{2}\right)\right)^{\frac{\alpha}{2}}}$}}\\
    		&
            \ \ \ \ +{\resizebox{.93\hsize}{!}{$\frac{C_\alpha  \gamma_j\varepsilon|\varepsilon|^\alpha b_j^{1+\alpha}
            f'_j( x)}{1+\varepsilon|\varepsilon|^\alpha b_j^{1+\alpha} f_j( x)}\displaystyle\fint\frac{(h_j( x)-h_j( y))\cos( x- y)d y}{\left( |\varepsilon|^{2+2\alpha}b_j^{2+2\alpha}\left(f_j( x)-f_j( y)\right)^2+4(1+\varepsilon|\varepsilon|^\alpha b_j^{1+\alpha} f_j( x))(1+\varepsilon|\varepsilon|^\alpha b_j^{1+\alpha} f_j( y))\sin^2\left(\frac{ x- y}{2}\right)\right)^{\frac{\alpha}{2}}}$}}\\
            &
            \ \ \ \ -{\resizebox{.9\hsize}{!}{$\frac{C_\alpha  \gamma_j|\varepsilon|^{2+2\alpha}b_j^{2+2\alpha} f'_j( x)h_j( x)}{(1+\varepsilon|\varepsilon|^\alpha b_j^{1+\alpha}
            f_j( x))^2}\displaystyle\fint \frac{(f_j( x)-f_j( y))\cos( x- y)d y}{\left( |\varepsilon|^{2+2\alpha}b_j^{2+2\alpha}\left(f_j( x)-f_j( y)\right)^2+4(1+\varepsilon|\varepsilon|^\alpha b_j^{1+\alpha} f_j( x))(1+\varepsilon|\varepsilon|^\alpha b_j^{1+\alpha} f_j( y))\sin^2\left(\frac{ x- y}{2}\right)\right)^{\frac{\alpha}{2}}}$}}\\
    		&
            \ \ \ \ -{\resizebox{.9\hsize}{!}{$\frac{\alpha C_\alpha \gamma_j |\varepsilon|^{2+2\alpha}b_j^{2+2\alpha} f'_j( x)}{2(1+\varepsilon|\varepsilon|^\alpha b_j^{1+\alpha} f_j( x))}\displaystyle\fint \frac{(f_j( x)-f_j( y))\cos( x- y)}{\left( |\varepsilon|^{2+2\alpha}b_j^{2+2\alpha}\left(f_j( x)-f_j( y)\right)^2+4(1+\varepsilon|\varepsilon|^\alpha b_j^{1+\alpha} f_j( x))(1+\varepsilon|\varepsilon|^\alpha b_j^{1+\alpha} f_j( y))\sin^2\left(\frac{ x- y}{2}\right)\right)^{\frac{\alpha+2}{2}}}$}}\\
    		&
            \ \ \ \ \times\bigg(\varepsilon|\varepsilon|^\alpha b_j^{1+\alpha}(2(f_j( x)-f_j( y))(h_j( x)-h_j( y))\\
			& \ \ \ \ \ \ \ \ +4(h_j( x)(1+\varepsilon|\varepsilon|^\alpha b_j^{1+\alpha}  f_j( y))+h_j( y)(1+\varepsilon|\varepsilon|^\alpha b_j^{1+\alpha} f_j( x)))\sin^2\left(\frac{ x- y}{2}\right)\bigg)d y
    	\end{split}
    \end{equation}
    and
    \begin{equation}\label{2-14}
    	\begin{split}
    		&\partial_{ f_j} \mathcal{F}^\alpha_{j24}={\resizebox{.9\hsize}{!}{$\frac{C_\alpha \gamma_j |\varepsilon|^{1+\alpha}b_j^{1+\alpha}}{1+\varepsilon|\varepsilon|^\alpha b_j^{1+\alpha} f_j( x)}\displaystyle\fint \frac{(h'_j( y)f'_j( x)+h'_j( x)f'_j( y))\sin( x- y)d y}{\left( |\varepsilon|^{2+2\alpha}b_j^{2+2\alpha} \left(f_j( x)-f_j( y)\right)^2+4(1+\varepsilon|\varepsilon|^\alpha b_j^{1+\alpha} f_j( x))(1+\varepsilon|\varepsilon|^\alpha b_j^{1+\alpha} f_j( y))\sin^2\left(\frac{ x- y}{2}\right)\right)^{\frac{\alpha}{2}}}$}}\\
    		& \ \ \ \ -{\resizebox{.93\hsize}{!}{$\frac{\alpha C_\alpha \gamma_j|\varepsilon|^{2+2\alpha}b_j^{2+2\alpha}}{2(1+\varepsilon|\varepsilon|^\alpha b_j^{1+\alpha} f_j( x))}\displaystyle\fint  \frac{f'_j( x)f'_j( y)\sin( x- y)}{\left( |\varepsilon|^{2+2\alpha}b_j^{2+2\alpha} \left(f_j( x)-f_j( y)\right)^2+4(1+\varepsilon|\varepsilon|^\alpha b_j^{1+\alpha} f_j( x))(1+\varepsilon|\varepsilon|^\alpha b_j^{1+\alpha} f_j( y))\sin^2\left(\frac{ x- y}{2}\right)\right)^{\frac{\alpha+2}{2}}}$}}\\
    		& \ \ \ \ \times\bigg(\varepsilon|\varepsilon|^\alpha b_j^{1+\alpha}(2(f_j( x)-f_j( y))(h_j( x)-h_j( y))\\
			& \ \ \ \ \ \ \ \ +4(h_j( x)(1+\varepsilon|\varepsilon|^\alpha b_j^{1+\alpha}  f_j( y))+h_j( y)(1+\varepsilon|\varepsilon|^\alpha b_j^{1+\alpha} f_j( x)))\sin^2\left(\frac{ x- y}{2}\right)\bigg)d y.\\
			&\ \ \ \ \ \ \ \ \ {\resizebox{.93\hsize}{!}{$-\frac{C_{\alpha}\gamma_j|\varepsilon|^{1+\alpha}b_j^{1+\alpha}f'_{j}( x)}{(1+\varepsilon|\varepsilon|^{\alpha}b_j^{1+\alpha}f_{j}( x))^{2}}\displaystyle\fint  \frac{f'_{j}( y)h_j( x)\sin{( x- y)}}{\left( |\varepsilon|^{2+2\alpha}b_j^{2+2\alpha} \left(f_j( x)-f_j( y)\right)^2+4(1+\varepsilon|\varepsilon|^\alpha b_j^{1+\alpha} f_j( x))(1+\varepsilon|\varepsilon|^\alpha b_j^{1+\alpha} f_j( y))\sin^2\left(\frac{ x- y}{2}\right)\right)^{\frac{\alpha}{2}}}\, d y.$}}
    	\end{split}
    \end{equation}
    Similarly, one gets
    \begin{equation}\label{partial_{j}2}
  \partial_{ f_j} \mathcal{F}^\alpha_{j,1}(\varepsilon ,f,\lambda) =\partial_{ f_j}  \mathcal{F}^\alpha_{j,3}(\varepsilon ,f,\lambda) =\partial_{ f_j}  \mathcal{F}^\alpha_{j,4}(\varepsilon ,f,\lambda) = O(\varepsilon) .
\end{equation}
\begin{equation}\label{partial_j5}
  \partial_{ f_0} \mathcal{F}^\alpha_{j}(\varepsilon ,f,\lambda) = O(\varepsilon) .
\end{equation}
and
\begin{equation}\label{partial_j4}
  \partial_{ f_\ell} \mathcal{F}^\alpha_{j}(\varepsilon ,f,\lambda) = O(\varepsilon) .
\end{equation}
We only deal with the term $\mathcal{F}^\alpha_{j2l}$. The treatment of the term $\mathcal{F}^\alpha_{02l}$   is similar, so we omit it.
The first step is to demonstrate
    \begin{equation*}
        \lim\limits_{t\to0}\left\|\frac{\mathcal{F}^\alpha_{j2l}(\varepsilon, f_j+th_j)-\mathcal{F}^\alpha_{j2l}(\varepsilon, f_j)}{t}-\partial_{f_j}\mathcal{F}^\alpha_{j2l}(\varepsilon, g,h_j)\right\|_{Y^{N-1}}\to 0
    \end{equation*}
for $l = 1, \dots, 4$, we will focus on the most singular case, namely when $l = 2$. In this scenario, we have the following result
\begin{equation*}
		\begin{split}
			&\frac{\mathcal{F}^\alpha_{j22}(\varepsilon, f_j+th_j)-\mathcal{F}^\alpha_{j22}(\varepsilon, f_j)}{t}-\partial_{f_j}\mathcal{F}^\alpha_{j22}(\varepsilon, f_j,h_j)\\
			&=\frac{1}{t}\int\!\!\!\!\!\!\!\!\!\; {}-{}{\resizebox{.9\hsize}{!}{$(f'_j( x)-f'_j( y))\cos( x- y)\bigg(\frac{1}{D_\alpha(f_j+th_j)^{\alpha/2}}-\frac{1}{D_\alpha(f_j)}+t\frac{\Delta f_j\Delta h_j+2(h_j\tilde{R}_j+\tilde{h}_jR_j)\sin^2(\frac{ x- y}{2})}{D_\alpha(f_j)}\bigg)d y$}}\\
			&\ \ \ \ +\int\!\!\!\!\!\!\!\!\!\; {}-{}(h'_j( x)-h'_j( y))\cos( x- y)\bigg(\frac{1}{D_\alpha(f_j+th_j)}-\frac{1}{D_\alpha(f_j)}\bigg)d y\\
			&=\mathcal{G}^\alpha_{j21}+\mathcal{G}^\alpha_{j22}.
		\end{split}
	\end{equation*}
By taking $\partial^{N-1}$ derivatives of $\mathcal{G}^\alpha_{j21}$, we can infer the following
    \begin{equation*}
    	\begin{split}
    	  \partial^{N-1}\mathcal{G}^\alpha_{j21}&=\frac{1}{t}\int\!\!\!\!\!\!\!\!\!\; {}-{}
            \bigg(\frac{1}  {D_\alpha(f_j+th_j)^{\alpha/2}}-\frac{1}{D_\alpha(f_j)^{\alpha/2}}+t\frac{\Delta f_j\Delta     h+2(\tilde{R}_j h_j+\tilde{h}_j R_j)\sin^2(\frac{ x- y}{2})}{D_\alpha(f_j)^{\alpha/2}}\bigg)\\
    	  & \ \ \ \ \times(\partial^Nf_j( x)-\partial^Nf_j( y))\cos( x- y)d y+l.o.t.
    	\end{split}
    \end{equation*}
Using the mean value theorem, we derive the following result
    \begin{equation*}
    	\frac{1}{D_\alpha(f_j+th_j)^{\alpha/2}}-\frac{1}{D_\alpha(f_j)^{\alpha/2}}+t\frac{\Delta f_j\Delta h_j+2(\tilde{R}_j h_j+\tilde{h}_j R_j)\sin^2(\frac{ x- y}{2})}{D_\alpha(f_j)^{\alpha/2}}\sim \frac{Ct^2}{|\sin(\frac{ x- y}{2})|^{\alpha/2}}\varphi(\varepsilon,f_j,h_j),
    \end{equation*}
    with $\|\varphi(\varepsilon,f_j,h_j)\|_{L^\infty}<\infty$. Hence, we can conclude that
    \begin{equation*}
    	\|\mathcal{G}^\alpha_{j21}\|_{\mathcal{Y}^{N-1}}\le Ct\left\|\int\!\!\!\!\!\!\!\!\!\; {}-{}\frac{\partial^Nf_j( x)-\partial^Nf_j( y)}{|\sin(\frac{ x- y}{2})|^{\alpha/2}}d y\right\|_{L^2}\le Ct\|f_j\|_{X^{N+\alpha-1}}.
    \end{equation*}
   Following the same reasoning as in \eqref{2-7}, we can similarly show that $\norm{\mathcal{G}^\alpha_{j22}}_{Y^{N-1}}$ is bounded by $Ct\norm{f_j}_{X^{N+\alpha-1}}$. Therefore, by letting $t \rightarrow 0$, we conclude the first step. The second step consists of proving the continuity of $\partial_{f_j} \mathcal{F}^\alpha_{j22}(\varepsilon, f,\lambda)$, which also relies on \eqref{2-7}.

By applying the same approach as described earlier, we infer that
    \begin{equation}\label{2-15}
    	\partial_{f_j} \mathcal{F}^\alpha_{j1}(\varepsilon ,f,\lambda)h_j=|\varepsilon|\partial_{f_j}\mathcal{R}_{j1}(\varepsilon ,f,\lambda) ,
    \end{equation}
          \begin{equation}\label{deriv_fi3}
    	\partial_{f_j}\mathcal{F}^\alpha_{j3}(\varepsilon ,f,\lambda)h_j=|\varepsilon|\partial_{f_j}\mathcal{R}_{j3}(\varepsilon ,f,\lambda)
    \end{equation}
        \begin{equation}\label{deriv_fi4}
    	\partial_{f_j}\mathcal{F}^\alpha_{j4}(\varepsilon ,f,\lambda)h_j=|\varepsilon|\partial_{f_j}\mathcal{R}_{j4}(\varepsilon ,f,\lambda),
    \end{equation}
    Also we have
           \begin{equation}\label{deriv_fj}
    	\partial_{f_0}\mathcal{F}^\alpha_{j}(\varepsilon ,f,\lambda)h_0=|\varepsilon|\partial_{f_0}\mathcal{R}_{j}(\varepsilon ,f,\lambda) ,
    \end{equation}
    and
       \begin{equation}\label{deriv_fj2}
    	\partial_{f_\ell}\mathcal{F}^\alpha_{j}(\varepsilon ,f,\lambda)h_\ell=|\varepsilon|\partial_{f_\ell}\mathcal{R}_{j}(\varepsilon ,f,\lambda) ,
    \end{equation}
 these functions are continuous, where $\mathcal{R}_{j,2}(\varepsilon ,f,\lambda)$, $\mathcal{R}_{j,3}(\varepsilon ,f,\lambda)$, and $\mathcal{R}_{j}(\varepsilon ,f,\lambda)$ are bounded and $C^1$ functions. This completes the proof of Proposition \ref{lem2-3}.
\end{proof}
From \eqref{2-110}-\eqref{2-14}, by setting $\varepsilon = 0$ and $f_j \equiv 0$, we obtain
\begin{equation}\label{gateaux}
	\partial_{f_0} \mathcal{F}^\alpha_{0}(0,0,\lambda)h_0=C_\alpha\gamma_0\left(1-\frac{\alpha}{2}\right)\int\!\!\!\!\!\!\!\!\!\; {}-{} \frac{h_0( x- y)\sin( y)d y}{\left(4\sin^2(\frac{ y}{2})\right)^{\frac{\alpha}{2}}}-C_\alpha \gamma_0\int\!\!\!\!\!\!\!\!\!\; {}-{} \frac{(h'_0( x)-h'_0( x- y))\cos( y)d y}{\left(4\sin^2(\frac{ y}{2})\right)^{\frac{\alpha}{2}}} .
\end{equation}
\begin{equation}\label{gateaux2}
	\partial_{f_j} \mathcal{F}^\alpha_{j}(0,0,\lambda)h_j=C_\alpha\gamma_j\left(1-\frac{\alpha}{2}\right)\int\!\!\!\!\!\!\!\!\!\; {}-{} \frac{h_j( x- y)\sin( y)d y}{\left(4\sin^2(\frac{ y}{2})\right)^{\frac{\alpha}{2}}}-C_\alpha \gamma_j\int\!\!\!\!\!\!\!\!\!\; {}-{} \frac{(h'_j( x)-h'_j( x- y))\cos( y)d y}{\left(4\sin^2(\frac{ y}{2})\right)^{\frac{\alpha}{2}}} .
\end{equation}
\begin{equation}\label{gateaux1b}
	\partial_{f_j} \mathcal{F}^\alpha_{0}(0,0,\lambda)h_j=0 .
\end{equation}
\begin{equation}\label{gateaux3}
	\partial_{f_\ell} \mathcal{F}^\alpha_{0}(0,0,\lambda)h_\ell=0 .
\end{equation}
\begin{equation}\label{gateaux4}
	\partial_{f_0} \mathcal{F}^\alpha_{j}(0,0,\lambda)h_0=0 .
\end{equation}
and
\begin{equation}\label{gateaux5}
	\partial_{f_\ell} \mathcal{F}^\alpha_{j}(0,0,\lambda)h_\ell=0 .
\end{equation}
  Having established the $C^1$ regularity of the functional.  We then define the following nonlinear operator
\begin{equation*}
\mathcal{F}^\alpha(\varepsilon,f,\lambda):=\big(\mathcal{F}^\alpha_0(\varepsilon,f,\lambda),\mathcal{F}^\alpha_1(\varepsilon,f,\lambda),\mathcal{F}^\alpha_2(\varepsilon,f,\lambda)\big).
\end{equation*}
Note that identify roots of the nonlinear functional $\mathcal{F}^\alpha=0$ are equivalent to identifying roots $\lambda$ of   the $\mathcal{P}_j^\alpha(\lambda)=0$, defined in \eqref{alg-sysP}, here $\lambda$ denotes the point vortex parameters. As a result, we can produce a collection of trivial solutions for \eqref{1-1} with $\alpha \in (1,2)$.

\begin{proposition}\label{equivalence}
  The equation $ \mathcal{F}^\alpha(0, 0, \lambda)=0$  can be reduced to a system of two real equation
  \begin{equation}\label{alg-sysP}
{\resizebox{.97\hsize}{!}{$ \mathcal{P}^\alpha_j(\lambda):=\Omega d_j-\frac{\widehat{C}_\alpha}{2d_j^{1+\alpha}}\left({\gamma_0}+{\gamma_j}\displaystyle\sum_{k=1}^{m-1}{ \Big(2\sin\big({\frac{k\pi  }{m}}\big)  \Big)^{-\alpha}}+\gamma_{3-j}\displaystyle\sum_{k=0}^{m-1}\frac{1 -  \frac{d_{3-j}}{d_j} \cos\big({\frac{(2k\pm\vartheta)\pi }{m}}\big)}{\left(1+\left( \frac{d_{3-j}}{d_j} \right)^2 - 2\frac{d_{3-j}}{d_j} \cos\left({\frac{(2k\pm\vartheta)\pi }{m}}\right) \right)^{\frac\alpha2+1}}\right)=0, \, j=1,2.$}}
\end{equation}
  where $\widehat{C}_\alpha$ is defined in \eqref{eqn:kalpha2}.
\end{proposition}
\begin{proof}
This result follows directly from the expansions of the functional in \eqref{3-7}, together with the definition of the function $\mathcal{P}_j^\alpha(\lambda)$ provided in \eqref{alg-sysP}.
\end{proof}

\begin{proposition}
\label{critical} The system  \eqref{3-7} has a unique solution $\lambda^*=(\Omega^*,\gamma_2^*)$ given by
\begin{equation}
\gamma_2^*:=\frac{\big( d ^{\alpha+2}-1\big){\gamma_0}+ \big(\frac12S_\alpha  d^{\alpha+2} - T_\alpha^-( d ,\vartheta)\big)\gamma_1}{\frac12S_\alpha- T_\alpha^+( d ,\vartheta) d ^{\alpha+2}}  \label{pos}
\end{equation}%
with angular velocity
\begin{equation}
\Omega ^{\ast }_\alpha:=\frac{\widehat{C}_\alpha}{2(d_1^{\alpha+2}+d_2^{\alpha+2})}\Big(\gamma_0+\gamma_1\big(\tfrac12 S_\alpha+ T_\alpha^-( d ,\vartheta)\big)+\gamma_2^*\big(T_\alpha^+( d ,\vartheta)+\tfrac12 S_\alpha\big) \Big),  \label{vel}
\end{equation}
where $S_\alpha$ and $T^\pm_\alpha$ are defined in \eqref{S} and \eqref{t}, respectively and $d=\frac{d_2}{d_1}>0$ .
\end{proposition}
\begin{proof}
By using the computations obtained at the end of the Proposition \ref{p3-1}, specifically the system \eqref{3-7}  evaluating in $\varepsilon=0$, $f_j=0$ gives
  \begin{equation}\label{p-gG2}
 \begin{cases}
{\mathcal{F}}_0^{\alpha}(0,0,\Omega ,\lambda)&=0,
\\
{\mathcal{F}}_1^{\alpha}(0,0,\Omega ,\lambda)&=\Omega d_1 \sin(x)-\frac{\widehat{C}_\alpha}{2d_1^{\alpha+1}}\Big[\gamma_0+\frac{\gamma_1}{2}S_\alpha +\gamma_2T_\alpha^+( d ,\vartheta)\Big]\sin(x),
\\
{\mathcal{F}}_2^{\alpha}(0,0,\Omega ,\lambda)&=\Omega d_2 \sin(x) -\frac{\widehat{C}_\alpha}{2d_2^{\alpha+1}}\Big[\gamma_0+\gamma_1T_\alpha^-( d ,\vartheta)+\frac{\gamma_2}{2}S_\alpha \Big]\sin(x).
\end{cases}
\end{equation}
Therefore $\mathcal{F}^\alpha(0,0,\Omega ,\lambda)=0$, if and only if, we have the
following two equalities
\begin{equation}
{\mathcal{F}}_1^{\alpha}(0,0,\Omega ,\lambda)=\Omega d_1 -\frac{\widehat{C}_\alpha}{2d_1^{\alpha+1}}\Big[\gamma_0+\frac{\gamma_1}{2}S_\alpha +\gamma_2T_\alpha^+( d ,\vartheta)\Big]=0  \label{f1b}
\end{equation}%
and
\begin{equation}
{\mathcal{F}}_2^{\alpha}(0,0,\Omega ,\lambda)=\Omega d_2 -\frac{\widehat{C}_\alpha}{2d_2^{\alpha+1}}\Big[\gamma_0+\gamma_1T_\alpha^-( d ,\vartheta)+\frac{\gamma_2}{2}S_\alpha \Big]=0.  \label{f2b}
\end{equation}%
After straightforward computations using \eqref{f1b}  yields
\begin{equation*}
\gamma_2 ^{\ast }:=\frac{\big( d ^{\alpha+2}-1\big){\gamma_0}+ \big(\frac12S_\alpha  d^{\alpha+2} - T_\alpha^-( d ,\vartheta)\big)\gamma_1}{\frac12S_\alpha- T_\alpha^+( d ,\vartheta) d ^{\alpha+2}}.
\end{equation*}%
Now, by replacing $\gamma_2^{\ast }$ into \eqref{f2b}, we obtain
\begin{equation*}
\Omega^*_\alpha=\frac{\widehat{C}_\alpha}{2(d_1^{\alpha+2}+d_2^{\alpha+2})}\Big(\gamma_0+\gamma_1\big(\tfrac12 S_\alpha+ T_\alpha^-( d ,\vartheta)\big)+\gamma_2^*\big(T_\alpha^+( d ,\vartheta)+\tfrac12 S_\alpha\big) \Big).
\end{equation*}
\end{proof}
\begin{remark}
    In order to  ensure that $\gamma_2$ is non-vanishing, one has to assume that $\gamma_0$ and $\gamma_1$ verify the condition
 \begin{equation}\label{non-deg}
( d ^{\alpha+2}-1){\gamma_0}+ \big(\tfrac12S_\alpha  d ^{\alpha+2} - T_\alpha^-( d ,\vartheta)\big)\gamma_1\neq 0.
    \end{equation}
\end{remark}
Now, we are in position to prove that $\mathcal{F}^\alpha(\varepsilon,f,\Omega,\gamma_2)$
at the point $(0,0,\Omega^*_\alpha,\gamma_2^*)$ is an isomorphism.
\begin{proposition}
\label{lem2-4} Let $h=(h_0,h_1,h_2)\in  \mathcal{X}^{N+\alpha-1}$ and  $(\dot\Omega,
\dot\gamma_2)\in\mathbb{R}^2$, with $h_{j}(x)=\sum_{n=1}^{\infty
}a_{n}^j\sin (nx)$, $j=0,1,2$, we have the Gateaux derivative is given by
\begin{equation*}
D_{(f,\lambda)}\mathcal{F}^\alpha(0,0,\lambda^*)\begin{pmatrix}
\dot\Omega\\
\dot\gamma_2\\
h
\end{pmatrix}= -\begin{pmatrix}
0  & 0\\
d_1 &  \frac{\widehat{C}_\alpha }{2}\frac{ T_\alpha^+( d ,\vartheta)}{d_1^{1+\alpha}}\\
 d_2 &  \frac{\widehat{C}_\alpha }{2} \frac{ S_\alpha}{2d_2^{1+\alpha}}
\end{pmatrix}
\begin{pmatrix} \dot\Omega\\ \dot \gamma_2\end{pmatrix}\,\sin(x)
+\sum_{n=2}^\infty n\sigma_n\begin{pmatrix}
{\gamma_0}\, a_n^0 \\
{\gamma_1}\, a_n^1
\\
\gamma_2^*\, a_n^2
\end{pmatrix}\,\sin(nx),
\end{equation*}
where $\sigma_n$ is given by
\begin{equation}\label{gamma}
\sigma _{n}:=%
2^{\alpha-1 }\frac{\Gamma (1-\alpha )}{\left(\Gamma (1-%
\frac{\alpha }{2})\right)^2}\left( \frac{\Gamma (1+\frac{\alpha }{2})}{\Gamma (2-\frac{%
\alpha }{2})}-\frac{\Gamma (n+\frac{\alpha }{2})}{\Gamma (n+1-\frac{\alpha }{%
2})}\right) ,
\end{equation}
with
\begin{equation*}
g\triangleq (f_{1},f_{2},\Omega ,\gamma_2)\quad \text{and}\quad g_{0}\triangleq
(0,0,\Omega ^{\ast },\gamma_2^{\ast }),
\end{equation*}
Moreover, the linearized operator $D_{g}\mathcal{F}^\alpha(0,g_{0}): \mathcal{X}^{N+\alpha-1}\times \mathbb{R}\times \mathbb{%
\ R}\rightarrow \mathcal{Y}^{N-1}$ is an isomorphism if the determinant of the Jacobian satisfies
\begin{equation}\label{det-j-poly}
\det\big(D_\lambda \mathcal F^\alpha(\lambda)\big)=
-\frac{\widehat{C}_\alpha d_1}{2d_2^{\alpha+1}}\Big(\frac{ S_\alpha}{2}- d^{\alpha+2} T_\alpha^+( d ,\vartheta)\Big)\neq 0 .
\end{equation}
\end{proposition}

\begin{proof}
Since $\mathcal{F}^\alpha(\varepsilon,f,\lambda):=\big(\mathcal{F}^\alpha_0(\varepsilon,f,\lambda),\mathcal{F}^\alpha_1(\varepsilon,f,\lambda),\mathcal{F}^\alpha_2(\varepsilon,f,\lambda)\big)$, choose $h=(h_0,h_{1},h_{2})\in \mathcal{X}^{N+\alpha-1 }$
to obtain
\begin{equation*}
\begin{split}
D_{f}\mathcal{F}^{\alpha}&(0,0,0,\Omega,\gamma_2)h(x)=\\
&
\begin{pmatrix}
\partial _{f_{0}}\mathcal{F}^{\alpha}_0(0,0,0,\Omega,\gamma_2)h_{0}(x)+\partial
_{f_{1}}\mathcal{F}^{\alpha}_0(0,0,0,\Omega,\gamma_2)h_{1}(x) +\partial
_{f_{2}}\mathcal{F}^{\alpha}_0(0,0,0,\Omega,\gamma_2)h_{2}(x) \\
\partial _{f_{0}}\mathcal{F}^{\alpha}_1(0,0,0,\Omega,\gamma_2)h_{0}(x)+\partial
_{f_{1}}\mathcal{F}^{\alpha}_1(0,0,0,\Omega,\gamma_2)h_{1}(x) +\partial
_{f_{2}}\mathcal{F}^{\alpha}_1(0,0,0,\Omega,\gamma_2)h_{2}(x) \\
\partial _{f_{0}}\mathcal{F}^{\alpha}_2(0,0,0,\Omega,\gamma_2)h_{0}(x)+\partial
_{f_{1}}\mathcal{F}^{\alpha}_2(0,0,0,\Omega,\gamma_2)h_{1}(x) +\partial
_{f_{2}}\mathcal{F}^{\alpha}_2(0,0,0,\Omega,\gamma_2)h_{2}(x)
\end{pmatrix}%
,
\end{split}
\end{equation*}%
with $f=(f_0,f_{1},f_{2})$. By using \eqref{gateaux}, the Gateaux derivative of the
functional $\mathcal{F}^\alpha_{0}$ at the point $(0,0,0,\Omega^*_\alpha,\gamma^*_2)$ in the direction
$h_{0}$ has the following form
\begin{equation*}
	\partial_{f_0} \mathcal{F}^\alpha_{0}(0,0,\lambda)h_0=C_\alpha\gamma_0\left(1-\frac{\alpha}{2}\right)\fint \frac{h_0( x- y)\sin( y)d y}{\left(4\sin^2(\frac{ y}{2})\right)^{\frac{\alpha}{2}}}-C_\alpha \gamma_0\int\!\!\!\!\!\!\!\!\!\; {}-{} \frac{(h'_0( x)-h'_0( x- y))\cos( y)d y}{\left(4\sin^2(\frac{ y}{2})\right)^{\frac{\alpha}{2}}} .
\end{equation*}
On the other hand, the Gateaux derivative of the functional $\mathcal{F}^\alpha_{j}$ at the point $(0,0,0,\Omega,\gamma_2)$ with direction $h_{j}$ was obtained in %
\eqref{gateaux2}, for $j=1,2$ which is given by
\begin{equation*}
	\partial_{f_j} \mathcal{F}^\alpha_{j}(0,0,\lambda)h_j=C_\alpha\gamma_j\left(1-\frac{\alpha}{2}\right)\int\!\!\!\!\!\!\!\!\!\; {}-{} \frac{h_j( x- y)\sin( y)d y}{\left(4\sin^2(\frac{ y}{2})\right)^{\frac{\alpha}{2}}}-C_\alpha \gamma_j\int\!\!\!\!\!\!\!\!\!\; {}-{} \frac{(h'_j( x)-h'_j( x- y))\cos( y)d y}{\left(4\sin^2(\frac{ y}{2})\right)^{\frac{\alpha}{2}}} .
\end{equation*}
Hence, the differential of the nonlinear functional $\mathcal{F}^\alpha$ around $(0,0,0,\Omega,\gamma_2)$ in the direction $h=(h_0,h_{1},h_{2})\in \mathcal{X}^{N+\alpha-1 }$ has
the following form
\begin{equation*}
D_{f}\mathcal{F}^{\alpha}(0,0,0,\Omega,\gamma_2)h(x)=%
\begin{pmatrix}
{\gamma _{0}}\sum\limits_{n=2}^{\infty }a_{n}^{0}\sigma _{n}n\sin (nx), \\
{\gamma _{1}}\sum\limits_{n=2}^{\infty }a_{n}^{1}\sigma _{n}n\sin (nx),\\
{\gamma _{2}}\sum\limits_{n=2}^{\infty }a_{n}^{2}\sigma _{n}n\sin (nx),
\end{pmatrix}%
,
\end{equation*}%
where $f=(f_0,f_{1},f_{2})$. According to \cite[Proposition 2.4]{Cas1}, $\sigma
_{j}$ can take different expressions depending on the value of $\alpha $.
For $\alpha \in (1,2)$, we have that
\begin{equation*}
\sigma _{n}=2^{\alpha-1 }\frac{\Gamma (1-\alpha )}{\left(\Gamma (1-%
\frac{\alpha }{2})\right)^2}\left( \frac{\Gamma (1+\frac{\alpha }{2})}{\Gamma (2-\frac{%
\alpha }{2})}-\frac{\Gamma (n+\frac{\alpha }{2})}{\Gamma (n+1-\frac{\alpha }{%
2})}\right)  ,
\end{equation*}%
An important fact here is that $\{\sigma _{n}\}$ is increasing with respect
to $j$. Moreover, the asymptotic behavior of $\{\sigma _n\},$ when $n$ is
large, is $\sigma _{n}\sim
n^{\alpha -1}$ if $1<\alpha <2$.
From \eqref{3-7}, one knows that
\begin{equation}\label{fu}
    \begin{cases}
&\mathcal{F}^\alpha_{0}(0,0,0,\lambda)=0,\\
        &
        {\mathcal{F}}_1^{\alpha}(0,0,0,\lambda)=\Omega d_1\sin (x)-\frac{\widehat{C}_\alpha}{2d_1^{\alpha+1}}\Big[\gamma_0+\frac{\gamma_1}{2}S_\alpha +\gamma_2T_\alpha^+( d ,\vartheta)\Big]\sin(x),
\\
&{\mathcal{F}}_2^{\alpha}(0,0,0,\lambda)=\Omega d_2\sin (x) -\frac{\widehat{C}_\alpha}{2d_2^{\alpha+1}}\Big[\gamma_0+\gamma_1 T_\alpha^-( d ,\vartheta)+\frac{\gamma_2}{2}S_\alpha \Big]\sin(x).
    \end{cases}
\end{equation}%
Differentiating \eqref{fu} with respect to the angular speed $\Omega $
around the point $(\Omega^*_\alpha,\gamma^*_2)$ yields
\begin{equation*}
\partial _{\Omega }\mathcal{F}_{0}^{\alpha}(0,0,0,\Omega^*_\alpha,\gamma^*_2)=0,
\end{equation*}%
\begin{equation*}
\partial _{\Omega }\mathcal{F}_{1}^{\alpha}(0,0,0,\Omega^*_\alpha,\gamma^*_2)=d_1\sin(x),
\end{equation*}%
\begin{equation*}
\partial _{\Omega }\mathcal{F}_{2}^{\alpha}(0,0,0,\Omega^*_\alpha,\gamma^*_2)=d_2\sin(x),
\end{equation*}%
Next, differentiating \eqref{fu} with respect to $\gamma_2$ at the point $%
(\Omega^*_\alpha,\gamma^*_2),$ we get
\begin{equation*}
\partial _{\gamma_2}\mathcal{F}_{0}^{\alpha}(0,0,0,\Omega^*_\alpha,\gamma^*_2)=0,
\end{equation*}%
\begin{equation*}
\partial _{\gamma_2}\mathcal{F}_{1}^{\alpha}(0,0,0,\Omega^*_\alpha,\gamma^*_2)=-\frac{\widehat{C}_\alpha}{2d_1^{\alpha+1}}T_\alpha^+( d ,\vartheta)\sin(x),
\end{equation*}%
\begin{equation*}
\partial _{\gamma_2}\mathcal{F}_{1}^{\alpha}(0,0,0,\Omega^*_\alpha,\gamma^*_2)=-\frac{\widehat{C}_\alpha}{4d_2^{\alpha+1}}S_\alpha\sin(x),
\end{equation*}%
Therefore, for all $(\dot\Omega,\dot\gamma_{2})\in \mathbb{R}^{2}$, the differential of the mapping ${\mathcal{F}}^{\alpha}:=(\mathcal{F}_0^{\alpha},{\mathcal{F}}_1^{\alpha},{\mathcal{F}}_2^{\alpha})$ with respect to $\lambda=(\Omega,\gamma_2)$ on the direction $(\dot\Omega,\dot\gamma_2)$ is given by
\begin{equation}\label{jacob-mat}
D_{\lambda}{\mathcal{F}}^{\alpha}(\lambda^*)\begin{pmatrix} \dot\Omega\\ \dot \gamma_2\end{pmatrix} =\begin{pmatrix}
d_1 &  -\frac{\widehat{C}_\alpha }{2}\frac{ T_\alpha^+( d ,\vartheta)}{d_1^{1+\alpha}}\\
 d_2 &  -\frac{\widehat{C}_\alpha }{2} \frac{ S_\alpha}{2d_2^{1+\alpha}}
\end{pmatrix}
\begin{pmatrix} \dot\Omega\\ \dot \gamma_2\end{pmatrix}\sin(x).
\end{equation}

If the Jacobian determinant is non-trivial,
\begin{equation}\label{det-j-polyb}
\det\big(D_\lambda \mathcal F^\alpha(\lambda^*)\big)=
-\frac{\widehat{C}_\alpha d_1}{2d_2^{\alpha+1}}\Big(\frac{ S_\alpha}{2}- d^{\alpha+2} T_\alpha^+( d ,\vartheta)\Big)\neq 0,
\end{equation}
then,  the system \eqref{fu} has a unique solution $\lambda^*=( \Omega^*_\alpha,\gamma_2^*)$  given by
\begin{equation}\label{gamma0-om0}
 \begin{split}
\gamma_2^*&:=\frac{\big( d ^{\alpha+2}-1\big){\gamma_0}+ \big(\frac12S_\alpha  d^{\alpha+2} - T_\alpha^-( d ,\vartheta)\big)\gamma_1}{\frac12S_\alpha- T_\alpha^+( d ,\vartheta) d ^{\alpha+2}},\\
\Omega^*_\alpha&:=\frac{\widehat{C}_\alpha}{2(d_1^{\alpha+2}+d_2^{\alpha+2})}\Big(\gamma_0+\gamma_1\big(\tfrac12 S_\alpha+ T_\alpha^-( d ,\vartheta)\big)+\gamma_2^*\big(T_\alpha^+( d ,\vartheta)+\tfrac12 S_\alpha\big) \Big).
\end{split}
    \end{equation}

Now, we obtain the explicit expression of the inverse of $D_{(f,\lambda)}\mathcal{F}^\alpha(0,0,\lambda^*)$, namely $D_{(f,\lambda)}\mathcal{F}^{\alpha}(0,0,\lambda^*)^{-1}$. Let $h=(h_0,h_1,h_2)\in  \mathcal{X}^{N+\alpha-1}$ and  $(\dot\Omega,
\dot\gamma_2)\in\mathbb{R}^2$, with
$$h_{j}(x)=\sum_{n=1}^{\infty
}a_{n}^j\sin (nx), \quad j=0,1,2,$$
 we have the Gateaux derivative is given by
\begin{equation*}
D_{(f,\lambda)}\mathcal{F}^\alpha(0,0,\lambda^*)\begin{pmatrix}
\dot\Omega\\
\dot\gamma_2\\
h
\end{pmatrix}= \begin{pmatrix}
0  & 0\\
d_1 &  -\frac{\widehat{C}_\alpha }{2}\frac{ T_\alpha^+( d ,\vartheta)}{d_1^{1+\alpha}}\\
 d_2 &  -\frac{\widehat{C}_\alpha }{2} \frac{ S_\alpha}{2d_2^{1+\alpha}}
\end{pmatrix}
\begin{pmatrix} \dot\Omega\\ \dot \gamma_2\end{pmatrix}\,\sin(x)
+\sum_{n=2}^\infty n\sigma_n\begin{pmatrix}
{\gamma_0}\, a_n^0 \\
{\gamma_1}\, a_n^1
\\
\gamma_2^*\, a_n^2
\end{pmatrix}\,\sin(nx).
\end{equation*}
Let $g\in \mathcal{Y}^{N-1}$ satisfying the following expansion
\begin{equation*}
g(x)=\sum_{n=1}^{\infty }%
\begin{pmatrix}
A_n^0  \\
A_n^1\\
A_n^2
\end{pmatrix}%
\sin (nx).
\end{equation*}%
with $A_n^0=0$ if $n$ is not a multiple of $m$,  we can easily obtain the inverse of the linearized operator $%
D_{(f,\lambda)}\mathcal{F}^\alpha(0,0,\lambda^*)g(x)$ which has the following expression
\begin{equation*}
    \begin{split}
     &D_{(f,\lambda)}\mathcal{F}^\alpha(0,0,\lambda^*)^{-1}g(x)=\\
     &
    {\resizebox{.98\hsize}{!}{$ \left(
\frac{1}{\gamma_0}\displaystyle\sum_{n=2}^\infty\frac{A_n^0}{n\sigma_n}\, \cos(nx),
\frac{1}{\gamma_1}\displaystyle\sum_{n=2}^\infty\frac{A_n^1}{n\sigma_n}\,\cos(nx),
\frac{1}{\gamma_2^*}\displaystyle\sum_{n=2}^\infty\frac{A_n^2}{n\sigma_n}\, \cos(nx)\, , \frac{\widehat{C}_\alpha }{2d_2^{1+\alpha}} \frac{ d ^{\alpha+1}T_\alpha^+( d ,\vartheta)A_0^2-\frac12{S_\alpha}A_0^1}{\det\big(D_\lambda \mathcal P^\alpha(\lambda^*)\big)},\frac{A_0^2d_1-A_0^1d_2}{\det\big(D_\lambda \mathcal P^\alpha(\lambda^*)\big)}
\right)   $}}
    \end{split}
\end{equation*}
where $\det\big(D_\lambda \mathcal P^\alpha(\lambda^*)\big)$ was calculated in \eqref{det-j-polyb}.
Now we are going to prove that the linearization $D_{(f,\lambda)}\mathcal{F}^{\alpha}(0,0,\lambda^*)$ is an isomorphism from $ \mathcal{X}^{N+\alpha-1 }\times \mathbb{\ R}
\times \mathbb{\ R}$  to $\mathcal{Y}^{N-1}$. From Proposition \ref{lem2-3}
and \eqref{det-j-poly}, it is obvious that $D_{(f,\lambda)}\mathcal{F}^{\alpha}(0,0,\lambda^*)$ is an isomorphism from $ \mathcal{X}^{N+\alpha-1 }\times \mathbb{\ R}
\times \mathbb{\ R}$  into $\mathcal{Y}^{N-1}$. Hence
only the invertibility needs to be considered. In fact, the restricted
linear operator $D_{(f,\lambda)}\mathcal{F}^{\alpha}(0,0,\lambda^*)$  is invertible if and only if the Jacobian determinant  in \eqref{det-j-poly} is  non-vanishing. Thus the desired result is obtained.
\end{proof}

\section{Existence of nested polygons for gSQG equation}\label{section4}

The goal of this section is to provide a full statement of the  Theorem \ref{thm:informal-1polygon}. In other words, we shall describe the set of solutions of the equation $\mathcal{F}^\alpha(\varepsilon,f,\lambda) = 0$ around the point $(\varepsilon, f,\lambda) = (0,0,\lambda^*)$ by a one-parameter smooth curve $\varepsilon\rightarrow (f(\varepsilon),\lambda(\varepsilon))$ using the implicit function theorem.
\begin{theorem}\label{thm:polygon}
  Let  $b_1,b_2,d_1,d_2\in (0,\infty)$ such that $d=d_2/d_1>0$ satisfies \eqref{det-j-poly}, and let $\gamma_0,\gamma_1\in \R\setminus\{0\}$ such that \eqref{non-deg} holds. Then
\begin{enumerate}[label=\rm(\roman*)]
\item There exists $\varepsilon_0 > 0$ and a neighborhood $\Lambda$ of $\lambda^*$ in $\R^2$ such that  $\mathcal{F}^\alpha$
can be extended to a $C^1$ mapping $(-\varepsilon_0,\varepsilon_0)\times  \mathcal{B}_X\times \Lambda \to \mathcal{Y}^{N-1}$.
\item $\mathcal{F}^\alpha(0,0,\lambda^*)=0,
$
where $\lambda^*=(\Omega^*_\alpha,\gamma_2^*)$ is given by  \eqref{pos} and \eqref{vel}.
\item  The linear operator
$D_{(f,\lambda)}\mathcal{F}^\alpha (0,0,\lambda^*)\colon   \mathcal{X}^{N+\alpha-1}\times\R^2 \to  \mathcal{Y}^{N-1}$ is an isomorphism.
\item There exists $\varepsilon_1>0$ and a unique $C^1$ function $(f,\lambda)\colon (-\varepsilon_1,\varepsilon_1)\to   \mathcal{B}_X\times\R^2$ such that
\begin{equation}\label{gf1f2-poly}
\mathcal{F}^\alpha \big(\varepsilon, f(\varepsilon),\lambda(\varepsilon)\big)=0,
\end{equation}
with $\lambda(\varepsilon)=\lambda^*+o(\varepsilon)$ and
\begin{gather}
f(\varepsilon)=(f_0(\varepsilon),f_1(\varepsilon),f_2(\varepsilon))=\Xi_\alpha\,\Big(0,\frac{\varepsilon b_1\mathcal{H}_1^\alpha}{\gamma_1d_1^{\alpha+2}}\sin(x) ,\frac{\varepsilon b_2\mathcal{H}_2^\alpha}{\gamma_2^* d_2^{\alpha+2}}\sin(x)\Big)+o(\varepsilon),
\notag \\
\label{Qjalpha}
{\resizebox{.94\hsize}{!}{$\mathcal{H}_j^\alpha:=
\gamma_0+ \displaystyle\sum_{\ell=1}^{2} \gamma_\ell\displaystyle\sum_{k=\delta_{\ell j}}^{m-1}\fint \frac{2d_jd_\ell\cos(2k\pi  /m-(\delta_{2\ell}-\delta_{2j})\vartheta\pi /m) -d_j^2-d_\ell^{2}\cos(2(2k\pi  /m-(\delta_{2\ell}-\delta_{2j})) }{(d_\ell^{2}+2d_jd_\ell\cos(2k\pi  /m-(\delta_{2\ell}-\delta_{2j})\vartheta\pi /m)-d_j^{2})^{\frac{\alpha}{2}+2}},$}} \\
\Xi_\alpha:=\frac{(\alpha+2)\Gamma(1-\frac\alpha2)\Gamma(3-\frac\alpha2)}{4\Gamma(2-\alpha)}.
\end{gather}
\item For all  $\varepsilon\in (-\varepsilon_1,\varepsilon_1)\setminus\{0\}$,  the domains $\mathcal{O}_j^\varepsilon$, whose boundaries are given by the conformal parametrizations $R_j^\varepsilon(x)=1+\varepsilon|\varepsilon|^\alpha b_j^{1+\alpha} f_j(x)\colon\mathbb{T}\to \partial \mathcal{O}_j^\varepsilon$,   are
 strictly convex.
\end{enumerate}
\end{theorem}

\begin{proof}

 The regularity of the nonlinear operator $\mathcal{F}^\alpha $ follows from Propositions~\ref{p3-1} and \ref{lem2-3}. In order to prove  the reflection symmetry property we shall assume that the $f_1,f_2$ are even functions,
we need to prove that $\mathcal{F}_j^\alpha (\varepsilon,f,\lambda)$ is an odd function, in other words
\begin{equation}\label{ref-sym24}
\mathcal{F}_j^\alpha (\varepsilon,f,\lambda)(-x)=-\mathcal{F}_j^\alpha (\varepsilon,f,\lambda)(x).
\end{equation}
  Since $f_1, f_2$ are even function, then by using \eqref{fj1} one gets
  $$x\mapsto \Omega \left( |\varepsilon |^{2+\alpha }b_{j}^{2+\alpha
}f_{j}^{\prime }(x)-d_j \left( \frac{\varepsilon|\varepsilon|^{\alpha}b_{j}^{1+\alpha }f_{j}^{\prime }(x)\cos (x)}{1+\varepsilon|\varepsilon|^{\alpha}b_{j}^{1+\alpha }f_{j}(x)}-\sin (x)\right) \right)
  $$
   satisfies \eqref{ref-sym24}. Moreover, using \eqref{2-8}, we can check that the $\mathcal{F}_{j1}^\alpha(\varepsilon, f_j)(x)$  are odd  for every $f_j$ even. Notice that $f_j(x)\in \mathcal{B}_X$ is an even function, and $f'_j(x)$ is odd. By changing $y$ to $-y$ in $\mathcal{F}_{j2}^\alpha(\varepsilon, f_j)$, we can deduce that $\mathcal{F}_{j2}^\alpha(\varepsilon, f_j)$ is odd. Similarly, we can deduce that the remaining terms, namely $\mathcal{F}_{j3}^\alpha(\varepsilon, f_j)$ and $\mathcal{F}_{j4}^\alpha(\varepsilon, f_j)$ are also odd functions.

Thus, it remains to check the $m$-fold  symmetry  property of $\mathcal{F}_0^\alpha $, namely, that if
\begin{equation}\label{m-fold-poly}
{f_0(Q_{\frac{2\pi }{m}}x)}=Q_{\frac{2\pi }{m}} f_0(x), \quad \forall x\in \mathbb{T},
\end{equation}
then
\begin{equation}\label{m-fold2-poly}
\mathcal{F}_0^\alpha (\varepsilon,f,\lambda)(Q_{\frac{2\pi }{m}}{x})=\mathcal{F}_0^\alpha (\varepsilon,f,\lambda)({x}), \quad \forall x\in \mathbb{T}.
\end{equation}
From \eqref{f01}  and \eqref{m-fold-poly} one has
\begin{equation*}
\begin{split}
& \mathcal{F}^\alpha_{01}=\Omega  |\varepsilon |^{2+\alpha }b_{0}^{2+\alpha
}f_{0}^{\prime }(x) ,
\end{split}
\end{equation*}
Same computations to the derivation of  \eqref{2-8}, using  \eqref{m-fold-poly} and the fact the rotation matrix $Q_{\frac{2\pi }{m}}$ preserves length and area one has  one gets
\begin{equation*}
\begin{split}
&{\resizebox{.98\hsize}{!}{$\mathcal{F}^\alpha_{02}(\varepsilon,f_0,\lambda)(Q_{\frac{2\pi }{m}}x)=	C_\alpha\gamma _0\left(1-\frac{\alpha}{2}\right)\displaystyle\fint\frac{f_0(Q_{\frac{2\pi }{m}}(x-y))\sin
(y)dy}{|\sin (\frac{y}{2})|^\alpha}-C_\alpha \gamma _0\int
\!\!\!\!\!\!\!\!\!\;{}-{}\frac{(f_0^{\prime }(Q_{\frac{2\pi }{m}}x)-f_0^{\prime }(Q_{\frac{2\pi }{m}}(x-y)))\cos
(y)dy}{|\sin (\frac{y}{2})|^\alpha}+\varepsilon |\varepsilon |^\alpha\mathcal{R}%
_{02}(\varepsilon ,f_0),$}}\\
&
\qquad\qquad\qquad{\resizebox{.84\hsize}{!}{$= C_\alpha\gamma _0\left(1-\frac{\alpha}{2}\right)\displaystyle\fint\frac{f_0(x-y)\sin
(y)dy}{|\sin (\frac{y}{2})|^\alpha}-C_\alpha \gamma _0\int
\!\!\!\!\!\!\!\!\!\;{}-{}\frac{(f_0^{\prime }(x)-f_0^{\prime }(x-y))\cos
(y)dy}{|\sin (\frac{y}{2})|^\alpha}+\varepsilon |\varepsilon |^\alpha\mathcal{R}%
_{02}(\varepsilon ,f_0),$}}\\
&
\qquad\qquad\qquad =\mathcal{F}^\alpha_{02}(\varepsilon,f_0,\lambda)(x).
\end{split}
\end{equation*}%
and using \eqref{f031}
\begin{equation*}
\begin{split}
\mathcal{F}^\alpha_{03}(\varepsilon,f_0,\lambda)(Q_{\frac{2\pi }{m}}x)=\mathcal{F}^\alpha_{03}(\varepsilon,f_0,\lambda)(x) .
\end{split}
\end{equation*}%
Then, by  \eqref{m-fold-poly}, we also deduce that
\begin{equation*}
\mathcal{F}^\alpha_0(\varepsilon, f_0)(Q_{\frac{2\pi }{m}}x)=\mathcal{F}^\alpha_0(\varepsilon, f_0)(x).
\end{equation*}
Then, \eqref{m-fold2-poly} is satisfied. These concludes the proof of {\rm(i)}.
 The proof of {\rm(ii)} follows immediately from Proposition~\ref{lem2-4}, \eqref{p-gG2} and \eqref{gamma0-om0}.
In order to show {\rm(iii)} we  use  Proposition~\ref{equivalence}  and  \eqref{jacob-mat} to get,
 for all $h=(h_0,h_1,h_2)\in  \mathcal{X}^{N+\alpha-1}$ and  $(\dot\Omega,
\dot\gamma_2)\in\mathbb{R}^2$,
\begin{equation}\label{eq:dev}
D_{(f,\lambda)}\mathcal{F}^\alpha(0,0,\lambda^*)\begin{pmatrix}
\dot\Omega\\
\dot\gamma_2\\
h
\end{pmatrix}= \begin{pmatrix}
0  & 0\\
d_1 & - \frac{\widehat{C}_\alpha }{2}\frac{ T_\alpha^+( d ,\vartheta)}{d_1^{1+\alpha}}\\
 d_2 & - \frac{\widehat{C}_\alpha }{2} \frac{ S_\alpha}{2d_2^{1+\alpha}}
\end{pmatrix}
\begin{pmatrix} \dot\Omega\\ \dot \gamma_2\end{pmatrix}\,\sin(x)
+\sum_{n=2}^\infty n\sigma_n\begin{pmatrix}
{\gamma_0}\, a_n^0 \\
{\gamma_1}\, a_n^1
\\
\gamma_2^*\, a_n^2
\end{pmatrix}\,\sin(nx),
\end{equation}
where $\sigma_n$ is given by \eqref{gamma}.
Proposition~\ref{lem2-4} and the assumption \eqref{det-j-poly} then imply {\rm(iii)}.

\vspace{0.2cm}

The existence and uniqueness in {\rm(iv)} follow from the implicit function theorem. In order to compute the asymptotic behavior of the solution, we   differentiate \eqref{gf1f2-poly} with respect to \( \varepsilon \) at the point \( (0, 0, \lambda^*) \) to obtain the following formula
\begin{equation}\label{comp}
\partial_\varepsilon \big(f(\varepsilon),\lambda(\varepsilon)\big)\big|_{\varepsilon=0}=-D_{(f,\lambda)} \mathcal{F}^\alpha\big(0,0,\lambda^*)^{-1}\partial_\varepsilon \mathcal{F}^\alpha(0,0,\lambda^*).
\end{equation}
For any $g\in \mathcal{Y}^{N-1}$ with the expansion
\begin{equation*}
g(x)=\sum_{n=1}^{\infty }%
\begin{pmatrix}
A_n^0  \\
A_n^1\\
A_n^2
\end{pmatrix}%
\sin (nx).
\end{equation*}%
with $A_n^0=0$ if $n$ is not a multiple of $m$, we have
\begin{equation}\label{eq:inv}
    \begin{split}
     &D_{(f,\lambda)}\mathcal{F}^\alpha(0,0,\lambda^*)^{-1}g(x)=\\
     &
    {\resizebox{.98\hsize}{!}{$ \left(
\frac{1}{\gamma_0}\displaystyle\sum_{n=2}^\infty\frac{A_n^0}{n\sigma_n}\, \cos(nx),
\frac{1}{\gamma_1}\displaystyle\sum_{n=2}^\infty\frac{A_n^1}{n\sigma_n}\,\cos(nx),
\frac{1}{\gamma_2^*}\displaystyle\sum_{n=2}^\infty\frac{A_n^2}{n\sigma_n}\, \cos(nx)\, , \frac{\widehat{C}_\alpha }{2d_2^{1+\alpha}} \frac{ d ^{\alpha+1}T_\alpha^+( d ,\vartheta)A_0^2-\frac12{S_\alpha}A_0^1}{\det\big(D_\lambda \mathcal P^\alpha(\lambda^*)\big)},\frac{A_0^2d_1-A_0^1d_2}{\det\big(D_\lambda \mathcal P^\alpha(\lambda^*)\big)}
\right)   $}}
    \end{split}
\end{equation}
where $\det\big(D_\lambda \mathcal P^\alpha(\lambda^*)\big)$ was calculated in \eqref{det-j-poly}.
On the other hand, from   \eqref{3-7} we have
\begin{equation}\label{eq:f}
 \begin{cases}
{\mathcal{F}}_0^{\alpha}(\varepsilon,0,\lambda)&=\varepsilon\mathcal{R}_{0}(\varepsilon ,f),
\\
{\mathcal{F}}_1^{\alpha}(\varepsilon,0,\lambda)&=\Omega d_1 \sin(x)-\frac{\widehat{C}_\alpha}{2d_1^{\alpha+1}}\Big[\gamma_0+\frac{\gamma_1}{2}S_\alpha +\gamma_2T_\alpha^+( d ,\vartheta)\Big]\sin(x)+\varepsilon\mathcal{R}_{0}(\varepsilon ,f),
\\
{\mathcal{F}}_2^{\alpha}(\varepsilon,0,\lambda)&=\Omega d_2 \sin(x) -\frac{\widehat{C}_\alpha}{2d_2^{\alpha+1}}\Big[\gamma_0+\gamma_1T_\alpha^-( d ,\vartheta)+\frac{\gamma_2}{2}S_\alpha \Big]\sin(x)+\varepsilon\mathcal{R}_{0}(\varepsilon ,f).
\end{cases}
\end{equation}
By combining \eqref{f031}, Taylor formula \eqref{taylor}, \eqref{eq:a2} and \eqref{eq:b2}
\begin{equation*}
    A_{k\ell 0}=d_\ell^2 ,
\end{equation*}
 \begin{equation*}
 \begin{split}
         B_{k\ell 0}&={\resizebox{.88\hsize}{!}{$-2d_\ell\cos\left(2k\pi  /m-\delta_{2\ell}\vartheta\pi /m\right)\left(b_0\cos(x)-b_\ell\cos\left(y-2k\pi  /m+\delta_{2\ell}\vartheta\pi /m\right)\right)$}}\\
         &
        \qquad {\resizebox{.88\hsize}{!}{$+2d_\ell\sin\left(2k\pi  /m-\delta_{2\ell}\vartheta\pi /m\right)\left(b_0\sin(x)-b_\ell\sin\left(y-2k\pi  /m+\delta_{2\ell}\vartheta\pi /m\right)\right)$}} ,
 \end{split}
 \end{equation*}
 one gets
\begin{equation}\label{eq:f03}
\begin{split}
\mathcal{F}^\alpha_{031}&=\displaystyle\sum_{\ell=1}^2 \gamma_\ell \displaystyle\sum_{k=0}^{m-1}\frac{ C_\alpha
}{\varepsilon b_\ell  }\displaystyle\fint
\frac{\sin\left(x-y+2k\pi  /m-\delta_{2\ell}\vartheta\pi /m\right)dy}{(A_{k\ell 0}+\varepsilon B_{k\ell 0})^{\alpha/2}}+\varepsilon\mathcal{R}_{031}(\varepsilon,f_\ell,f_0)\\
&
=\displaystyle\sum_{\ell=1}^2 \gamma_\ell \displaystyle\sum_{k=0}^{m-1}\frac{ C_\alpha
}{\varepsilon b_\ell  }\displaystyle\fint
\frac{\sin\left(x-y+2k\pi  /m-\delta_{2\ell}\vartheta\pi /m\right)dy}{A_{k\ell 0}^{\alpha/2}}\\
&
{\resizebox{.95\hsize}{!}{$\quad -\displaystyle\sum_{\ell=1}^2 \gamma_\ell \displaystyle\sum_{k=0}^{m-1}\frac{ \alpha C_\alpha
}{2 b_\ell  }\displaystyle\fint\int_0^1
\frac{B_{k\ell 0}\sin\left(x-y+2k\pi  /m-\delta_{2\ell}\vartheta\pi /m\right)dt\,dy}{(A_{k\ell 0}+\varepsilon t B_{k\ell 0})^{\frac{\alpha}{2}+1}}+\varepsilon\mathcal{R}_{031}(\varepsilon,f_\ell,f_0)$}}\\
&
= -\frac{ \alpha C_\alpha
}{2   }\displaystyle\sum_{\ell=1}^2 \gamma_\ell \displaystyle\sum_{k=0}^{m-1}\frac{ 1
}{ b_\ell  }\displaystyle\fint
\frac{B_{k\ell 0}\sin\left(x-y+2k\pi  /m-\delta_{2\ell}\vartheta\pi /m\right)dy}{(A_{k\ell 0}+\varepsilon t B_{k\ell 0})^{\frac{\alpha}{2}+1}}+\varepsilon\mathcal{R}_{031}(\varepsilon,f_\ell,f_0)\\
&
={\resizebox{.9\hsize}{!}{$\frac{ \alpha (\alpha+2)C_\alpha
}{4   }\displaystyle\sum_{\ell=1}^2 \gamma_\ell \displaystyle\sum_{k=0}^{m-1}\frac{ 1
}{ b_\ell  }\displaystyle\fint\int_0^1\int_0^1
\frac{\varepsilon B^2_{k\ell 0}\sin\left(x-y+2k\pi  /m-\delta_{2\ell}\vartheta\pi /m\right)ds\,dt\,dy}{(A_{k\ell 0}+\varepsilon st B_{k\ell 0})^{\frac{\alpha}{2}+2}}+\varepsilon\mathcal{R}_{031}(\varepsilon,f_\ell,f_0)$}}\\
&
={\resizebox{.9\hsize}{!}{$\frac{ \alpha (\alpha+2)C_\alpha
}{4   }\displaystyle\sum_{\ell=1}^2 \gamma_\ell \displaystyle\sum_{k=0}^{m-1}\frac{ 1
}{ b_\ell  }\displaystyle\fint
\frac{\varepsilon B^2_{k\ell 0}\sin\left(x-y+2k\pi  /m-\delta_{2\ell}\vartheta\pi /m\right)dy}{A_{k\ell 0}^{\frac{\alpha}{2}+2}}+\varepsilon\mathcal{R}_{031}(\varepsilon,f_\ell,f_0)$}}\\
&
=-\frac{ \alpha (\alpha+2)C_\alpha}{4   }\displaystyle\sum_{\ell=1}^2 \frac{\gamma_\ell b_0}{d_\ell^{\alpha+4}}\displaystyle\sum_{k=0}^{m-1}d_\ell^{2}\cos(2(2k\pi  /m-\delta_{2\ell}) ) \sin(x)\cos(x)+\varepsilon\mathcal{R}_{031}(\varepsilon,f_\ell,f_0) \\
&
=\varepsilon\mathcal{R}_{031}(\varepsilon,f_\ell,f_0) .
\end{split}
\end{equation}
By using \eqref{fj31}, with $$A_{0 0 j}=d_j^2,$$
 \begin{equation*}
   B_{0 0 j}=2d_j\left(b_j\cos(x)-b_0\cos\left(y-\delta_{2j}\vartheta\pi /m\right)\right),
 \end{equation*}
 and then applying Taylor formula \eqref{taylor} one gets
\begin{equation}\label{eq:fj3}
    \begin{split}
     \mathcal{F}^\alpha_{j31}&= \frac{\gamma_{0} C_\alpha
}{\varepsilon b_{0} }\displaystyle\fint
\frac{\sin\left(x-y+\delta_{2j}\vartheta\pi /m\right)dy}{\left(A_{00j}+\varepsilon B_{00j}\right)^{\alpha/2}}+\varepsilon\mathcal{R}_{j31}(\varepsilon,f_0,f_j)\\
&
{\resizebox{.93\hsize}{!}{$ = \frac{\gamma_{0} C_\alpha
}{\varepsilon b_{0} }\displaystyle\fint
\frac{\sin\left(x-y+\delta_{2j}\vartheta\pi /m\right)dy}{A_{00j}^{\alpha/2}}-\frac{\alpha}{2}\frac{\gamma_{0} C_\alpha
}{ b_{0} }\displaystyle\fint \int_0^1
\frac{B_{00j}\sin\left(x-y+\delta_{2j}\vartheta\pi /m\right)dy\, dt}{\left(A_{00j}+t\varepsilon B_{00j}\right)^{\alpha/2+1}}+\varepsilon\mathcal{R}_{j31}(\varepsilon,f_0,f_j)$}}\\
&
=-\frac{\alpha}{2}\frac{\gamma_{0} C_\alpha
}{ b_{0} }\displaystyle\fint
\frac{B_{00j}\sin\left(x-y+\delta_{2j}\vartheta\pi /m\right)dy}{A_{00j}^{\alpha/2+1}}\\
&
\qquad + \frac{\alpha (\alpha+2)C_\alpha \gamma_0}{4b_0} \displaystyle\fint \int_0^1 \int_0^1
\frac{\varepsilon B^2_{00j}\sin\left(x-y+\delta_{2j}\vartheta\pi /m\right)dy\,dt\,ds}{\left(A_{00j}+st\varepsilon B_{00j}\right)^{\alpha/2+2}}+\varepsilon\mathcal{R}_{j31}(\varepsilon,f_0,f_j)\\
&
=-\frac{\alpha}{2}\frac{\gamma_{0} C_\alpha
}{ b_{0} }\displaystyle\fint
\frac{B_{00j}\sin\left(x-y+\delta_{2j}\vartheta\pi /m\right)dy}{A_{00j}^{\alpha/2+1}}\\
&
\qquad + \frac{\alpha (\alpha+2)C_\alpha \gamma_0 }{4b_0} \displaystyle\fint
\frac{\varepsilon B^2_{00j}\sin\left(x-y+\delta_{2j}\vartheta\pi /m\right)dy}{A_{00j}^{\alpha/2+2}}+\varepsilon\mathcal{R}_{j31}(\varepsilon,f_0,f_j)\\
&
{\resizebox{.93\hsize}{!}{$= \frac{\alpha (\alpha+2)C_\alpha \gamma_0 }{4b_0} \displaystyle\fint
\frac{\varepsilon \left(2d_j\left(b_j\cos(x)-b_0\cos\left(y-\delta_{2j}\vartheta\pi /m\right)\right)\right)^2\sin\left(x-y+\delta_{2j}\vartheta\pi /m\right)dy}{d_j^{\alpha+4}}+\varepsilon\mathcal{R}_{j31}(\varepsilon,f_0,f_j)$}}\\
&
={\resizebox{.9\hsize}{!}{$ -\frac{\alpha (\alpha+2)C_\alpha \gamma_0 }{4b_0} \displaystyle\fint
\frac{8\varepsilon d^2_j b_0b_j\cos(x)\cos\left(y-\delta_{2j}\vartheta\pi /m\right)\sin\left(x-y+\delta_{2j}\vartheta\pi /m\right)dy}{d_j^{\alpha+4}}+\varepsilon\mathcal{R}_{j31}(\varepsilon,f_0,f_j)$}}\\
&
= -\frac{\alpha (\alpha+2)C_\alpha \gamma_0}{d_j^{\alpha+2}} \varepsilon b_j\cos(x)\sin(x)+\varepsilon\mathcal{R}_{j31}(\varepsilon,f_0,f_j) .
\end{split}
\end{equation}
Now, we   consider \eqref{333}, for all \( \alpha \in (1, 2) \), we have
\begin{equation}\label{eq:fj4}
\begin{split}
     \mathcal{F}^\alpha_{j41}  &
    =\sum_{\ell=1}^{2}\gamma_\ell\displaystyle\sum_{k=\delta_{\ell j}}^{m-1}\frac{ C_\alpha
}{\varepsilon b_\ell  }  \displaystyle\fint \frac{\sin(x-y+2k\pi  /m-(\delta_{2\ell}-\delta_{2j})\vartheta\pi /m)dy}{\left( A_{k\ell j}+\varepsilon B_{k\ell j}+O(\varepsilon^2) \right)^{\frac{\alpha}{2}}}+\varepsilon\mathcal{R}_{j41}(\varepsilon,f_\ell,f_j)\\
&
=\sum_{\ell=1}^{2}\gamma_\ell\displaystyle\sum_{k=\delta_{\ell j}}^{m-1}\frac{ C_\alpha
}{\varepsilon b_\ell  }  \displaystyle\fint  \frac{\sin(x-y+2k\pi  /m-(\delta_{2\ell}-\delta_{2j})\vartheta\pi /m)dy}{\left( A_{k\ell j} \right)^{\frac{\alpha}{2}}}\\
&
{\resizebox{.94\hsize}{!}{$\qquad-\frac{\alpha}{2}\displaystyle\sum_{\ell=1}^{2} \gamma_\ell\displaystyle\sum_{k=\delta_{\ell j}}^{m-1}\frac{ C_\alpha}{ b_\ell  }  \displaystyle\fint  \frac{B_{k\ell j}\sin(x-y+2k\pi  /m-(\delta_{2\ell}-\delta_{2j})\vartheta\pi /m)dy}{\left( A_{k\ell j}+\varepsilon tB_{k\ell j}+O(\varepsilon^2) \right)^{\frac{\alpha}{2}+1}}+\varepsilon\mathcal{R}_{j41}(\varepsilon,f_\ell,f_j)$}}\\
&
=-\frac{\alpha}{2}\sum_{\ell=1}^{2} \gamma_\ell\displaystyle\sum_{k=\delta_{\ell j}}^{m-1}\frac{ C_\alpha}{ b_\ell  } \displaystyle\fint \frac{B_{k\ell j}\sin(x-y+2k\pi  /m-(\delta_{2\ell}-\delta_{2j})\vartheta\pi /m)dy}{\left( A_{k\ell j}\right)^{\frac{\alpha}{2}+1}}\\
&
 {\resizebox{.94\hsize}{!}{$\qquad+\frac{\alpha (\alpha+2)C_\alpha }{4}\displaystyle\sum_{\ell=1}^{2} \gamma_\ell\displaystyle\sum_{k=\delta_{\ell j}}^{m-1}\frac{1}{ b_\ell  } \fint \frac{\varepsilon B_{k\ell j}^2\sin(x-y+2k\pi  /m-(\delta_{2\ell}-\delta_{2j})\vartheta\pi /m)dy}{\left(A_{k\ell j}\right)^{\frac{\alpha}{2}+2}}+\varepsilon\mathcal{R}_{j41}(\varepsilon,f_\ell,f_j)$}},
\end{split}
\end{equation}
where  $A_{k\ell j}$ and $B_{k\ell j}$ is defined in \eqref{eq:a} and \eqref{eq:b}, respectively.
Thus, we can conclude that
\begin{equation}
\label{f0dif-eps}
 {\resizebox{.96\hsize}{!}{$\partial_\varepsilon \mathcal{F}^\alpha_j(0, 0, \lambda^*)(x) = \alpha(\alpha+2)C_\alpha\displaystyle\sum_{\ell=1}^{2} \gamma_\ell\displaystyle\sum_{k=\delta_{\ell j}}^{m-1}b_j\fint\frac{2d_jd_\ell\cos(2k\pi  /m-(\delta_{2\ell}-\delta_{2j})\vartheta\pi /m) -d_j^2-d_\ell^{2}\cos(2(2k\pi  /m-(\delta_{2\ell}-\delta_{2j}))  }{(d_\ell^{2}+2d_jd_\ell\cos(2k\pi  /m-(\delta_{2\ell}-\delta_{2j})\vartheta\pi /m)-d_j^{2})^{\frac{\alpha}{2}+2}} \sin(x)\cos(x).$}}
\end{equation}
Hence, for all \( \alpha \in (1, 2) \), we conclude that
\[
\partial_\varepsilon \mathcal{F}^\alpha_j(0, 0, \lambda^*) \in \mathcal{Y}_0^{N-1}.
\]
Given that the linear operator
\( D_{f} \mathcal{F}^\alpha(0, 0, \lambda^*) \colon \mathcal{X}^{N+\alpha-1} \to \mathcal{Y}_0^{N-1} \) is an isomorphism and, according to the hypothesis, the kernel of the Jacobian operator $D_{\lambda}{\mathcal{F}}^{1}(\lambda)$ is non-trivial, then we can combine \eqref{eq:dev}, \eqref{comp}, \eqref{f0dif-eps}, and Proposition~\ref{lem2-4} to deduce that
\[
\partial_\varepsilon \lambda(\varepsilon) \big|_{\varepsilon=0} = 0 ,
\]
and
\begin{equation*}
{\resizebox{.98\hsize}{!}{$    \partial_\varepsilon f_{j}(\varepsilon) \big|_{\varepsilon=0}(x) = \frac{\alpha(\alpha+2) C_\alpha}{\sigma_2 \gamma_j} \displaystyle\sum_{\ell=1}^{2} \gamma_\ell\displaystyle\sum_{k=\delta_{\ell j}}^{m-1}b_\ell\fint \frac{2d_jd_\ell\cos(2k\pi  /m-(\delta_{2\ell}-\delta_{2j})\vartheta\pi /m) -d_j^2-d_\ell^{2}\cos(2(2k\pi  /m-(\delta_{2\ell}-\delta_{2j})) }{(d_\ell^{2}+2d_jd_\ell\cos(2k\pi  /m-(\delta_{2\ell}-\delta_{2j})\vartheta\pi /m)-d_j^{2})^{\frac{\alpha}{2}+2}}\sin(x).$}}
\end{equation*}
Finally, simple calculations using \eqref{gamma} yields
\begin{equation}
\label{gam1hatg}
\frac{\alpha C_\alpha}{\sigma_2} = \frac{\Gamma\left(1 - \frac{\alpha}{2}\right) \Gamma\left(3 - \frac{\alpha}{2}\right)}{\Gamma(2 - \alpha)}.
\end{equation}
 Combining the two last identities with \eqref{comp}, \eqref{eq:inv} and \eqref{gam1hatg} yields
\begin{equation*}
\partial_\varepsilon \big( f(\varepsilon),\lambda(\varepsilon)\big)\big|_{\varepsilon=0}=\Xi_\alpha\Big(0,\frac{b_1\mathcal{H}_1^0}{\gamma_1d_1^{\alpha+2}}\sin(x),\frac{b_2\mathcal{H}_2^0}{\gamma_2^* d_2^{\alpha+2}}\sin(x)\, ,0,0 \Big).
\end{equation*}
Thus, we conclude  the proof of {\rm (iv)}.

The convexity in {\rm (v)} is an straightforward computation. To this end, we need to verify that the set of solutions \( R_{j}(x) \) parameterizes convex patches. Hence, it suffices to compute the signed curvature of the interface of the patch \( z_{j}(x) = (z_{j}^{1}(x), z_{j}^{2}(x)) = R_j(x)(\cos(x), \sin(x)) \) at the point \( x \) as follows
\begin{equation*}
\begin{split}
\kappa _{j}(x)&
=\frac{(1+\varepsilon ^{1+\alpha}b_{j}^{1+\alpha}f_{j}(x))^{2}+2\varepsilon
^{2+2\alpha}b_{j}^{2+2\alpha}(f_{j}^{\prime }(x))^{2}-\varepsilon ^{1+\alpha}b_{j}^{1+\alpha}f_{j}^{\prime
\prime }(x)(1+\varepsilon ^{1+\alpha}b_{j}^{1+\alpha}f_{j}(x))}{\left( (1+\varepsilon
^{1+\alpha}b_{j}^{1+\alpha}f_{j}(x))^{2}+\varepsilon ^{2+2\alpha}b_{j}^{2+2\alpha}(f_{j}^{\prime
}(x))^{2}\right) ^{\frac{3}{2}}}\\
&
=\frac{1+O(\varepsilon )}{1+O(\varepsilon )}%
>0,
\end{split}%
\end{equation*}%
for $\varepsilon $ small and each $x\in \lbrack 0,2\pi )$. The quantity
obtained is non-negative if $\varepsilon\in (-\varepsilon_1, \varepsilon_1) $. Then, the signed
curvature is strictly positive and we obtain the desired result. Hence the
proof of Theorem \ref{thm:polygon} is completed.
\end{proof}

\section{Existence of nested polygons for SQG equation}\label{section5}

In this section we investigate the existence of a finite number of vortex patches for the surface quasi-geostrophic (SQG) equation  (\(\alpha=1\)). More precisely, we apply the implicit function theorem to construct time-periodic solutions to the SQG equation. Hence,  the next proposition gives  the continuity of $\mathcal{F}^1$. This result is equivalent
to Proposition \ref{p3-1}.

\begin{proposition}\label{p1b}
	There exists $\varepsilon_0>0$ and a small neighborhood $\Lambda$ of $\lambda^*$  such that the functional $\mathcal{F}^1$ can be extended from $\left(-\varepsilon_0, \varepsilon_0\right) \times \mathcal{B}_X\times \Lambda$ to $ \mathcal{Y}^{N-1}$ as  a continuous functional.
\end{proposition}
\begin{proof}
Similarly to the case $\alpha =1$, it is straightforward to see that  $\mathcal{F}^1_{j1}:\left(-\varepsilon_0, \varepsilon_0\right) \times \mathcal{B}_X\times \Lambda \rightarrow \mathcal{Y}^{N-1}$ is continuous and
\begin{equation}
 \mathcal{F}^1_{j1}=\Omega d_j\sin (x)+\varepsilon \abs{\varepsilon}\mathcal{R}_{j1}(\varepsilon,f_j)
\label{2-1s}
\end{equation}
where $\mathcal{R}_{j1}(\varepsilon ,f_j)$ is continuous. Since $f_{j}\in
\mathcal{B}_{X}$, by a simple change of variables $y$ to $-y$, it is easy to see that $%
\mathcal{F}^1_{j2}(\varepsilon ,f_j)$ is odd. By the application of the Taylor
formula \eqref{taylor} in $\mathcal{F}^1_{j21}$, we can treat with the possible
singularities at $\varepsilon =0$. As was noticed in the proof of Proposition \ref%
{p3-1}, the singular term is $\mathcal{F}^1_{j22}$. Hence, we need to compute the $%
\partial ^{N-1}$ derivatives of $\mathcal{F}^1_{j22}$
\begin{equation*}
\begin{split}
& {\resizebox{.98\hsize}{!}{$
\partial^{N-1}\mathcal{F}^1_{j22}=\gamma_{j}\displaystyle\fint
{}-{}\frac{(\partial^Nf_{j}(y)-\partial^Nf_{j}(x))\cos(x-y)dy}{\left(
|\varepsilon|^{4}b_{j}^{4}\left(f_{j}(x)-f_{j}(y)\right)^2+4(1+%
\varepsilon|\varepsilon| b^{2}_{j}
f_{j}(x))(1+\varepsilon|\varepsilon| b^{2}_{j}
f_{j}(y))\sin^2\left(\frac{x-y}{2}\right)\right)^{\frac{1}{2}}}$}} \\
&
\ \ -{\resizebox{.94\hsize}{!}{$\gamma_{j}\varepsilon|\varepsilon|
b_{j}^{2}\displaystyle\fint \frac{\cos(x-y)}{\left(
|\varepsilon|^{4}b_{j}^{4}\left(f_{j}(x)-f_{j}(y)\right)^2+4(1+%
\varepsilon|\varepsilon| b^{2}_{j}
f_{j}(x))(1+\varepsilon|\varepsilon| b^{2}_{j}
f_{j}(y))\sin^2\left(\frac{x-y}{2}\right)\right)^{\frac{3}{2}}}$}} \\
& \ \ \ {\resizebox{.98\hsize}{!}{$\times
\left(\varepsilon|\varepsilon|
b^{2}_{j}(f_{j}(x)-f_{j}(y))(f'_{j}(x)-f'_{j}(y))+2((1+\varepsilon|%
\varepsilon|
b_{j}^{2}f_{j}(x))f'_{j}(y)+(1+\varepsilon|\varepsilon|
b_{j}^{2}f_{j}(y))f'_{j}(x))\sin^2(\frac{x-y}{2})\right)$}} \\
& \ \ \ \times (\partial ^{N-1}f_{j}(y)-\partial ^{N-1}f_{j}(x))dy+l.o.t.
\end{split}%
\end{equation*}%
By Sobolev embedding, we know that $\Vert \partial ^{M}f_{j}\Vert
_{L^{\infty }}\leq C\Vert f_{j}\Vert _{X^{N+\log}}<\infty $ for $%
M=0,1,2 $ and $N\geq 3$. Therefore, by combining the mean value theorem
together with H\"{o}lder inequality, one gets
\begin{equation*}
\begin{split}
\left\Vert \partial ^{N-1}\mathcal{F}^1_{j22}\right\Vert _{L^{2}}& \leq C\left\Vert \int
\!\!\!\!\!\!\!\!\!\;{}-{}\frac{\partial ^Nf_{j}(x)-\partial ^Nf_{j}(y)}{%
|\sin (\frac{x-y}{2})|}dy\right\Vert _{L^{2}}+C\left\Vert \int
\!\!\!\!\!\!\!\!\!\;{}-{}\frac{\partial ^{N-1}f_{j}(x)-\partial
^{N-1}f_{j}(y)}{|\sin (\frac{x-y}{2})|}dy\right\Vert _{L^{2}} \\
& \leq C\Vert f_{j}\Vert _{X^{N+\log}}+C\Vert f_{j}\Vert _{X^{N+\log
-1}}<\infty .
\end{split}%
\end{equation*}%
Hence, we deduce that the range of $\mathcal{F}^1_{j2}$ belongs to $\mathcal{Y}^{N-1}$.

By similar computations  obtained in Proposition \ref{p3-1},
specifically in \eqref{2-7}, one only need to obtain the continuity of the
most singular term $\mathcal{F}^1_{j22}$. In other words, for $f_{j1},f_{j2}\in B_{r}$,
we have $\Vert \mathcal{F}^1_{j22}(\varepsilon ,f_{j1})- \mathcal{F}^1_{j22}(\varepsilon
,f_{j2})\Vert _{\mathcal{Y}^{N-1}}\leq C\Vert f_{j1}-f_{j2}\Vert _{X^{N+\log}}$.
In fact, the continuity of the nonlinear functional $\mathcal{F}^1_{j2}$ can be treated in the same manner as in the proof of Lemma \ref%
{p3-1}. Finally, by invoking the Taylor formula \eqref{taylor}, we obtain the
following expression of $\mathcal{F}^1$ which will be used later
\begin{equation}\label{2-21}
    \begin{cases}
&{\resizebox{.93\hsize}{!}{$\mathcal{F}^1_{0}(\lambda)=\frac{\gamma_{0}}{4}\displaystyle\fint   \frac{f_0( x- y)\sin( y)d y}{|\sin (\frac{y}{2})|}-\frac{\gamma_0}{2}\displaystyle\fint \frac{(f'_0( x)-f'_0( x- y))\cos( y)d y}{|\sin (\frac{y}{2})|}+\varepsilon\mathcal{R}_{0}(\varepsilon ,f)$}}\\
        &
        {\mathcal{F}}_1^{1}(\lambda)=\Omega d_1\sin (x)-\frac{1}{2d_1^{2}}\Big[\gamma_0+\frac{\gamma_1}{2}S_1 +\gamma_2T_1^+( d ,\vartheta)\Big]\sin(x)\\
        &
    {\resizebox{.93\hsize}{!}{$   \qquad \qquad + \frac{\gamma_{1}}{4}\displaystyle\fint\frac{f_1(x-y)\sin
(y)dy}{|\sin (\frac{y}{2})|}-\frac{\gamma_{1}}{2}\displaystyle\fint
\frac{(f_1^{\prime }(x)-f_1^{\prime }(x-y))\cos
(y)dy}{|\sin (\frac{y}{2})|}+\varepsilon\mathcal{R}_{1}(\varepsilon ,f)$}}
\\
&{\mathcal{F}}_2^{1}(\lambda)=\Omega d_2\sin (x) -\frac{1}{2d_2^{2}}\Big[\gamma_0+\gamma_1 T_1^-( d ,\vartheta)+\frac{\gamma_2}{2}S_1 \Big]\sin(x)\\
&
{\resizebox{.93\hsize}{!}{$\qquad \qquad +\frac{\gamma_{2}}{4}\displaystyle\fint\frac{f_2(x-y)\sin
(y)dy}{|\sin (\frac{y}{2})|}-\frac{\gamma_{2}}{2}\displaystyle\fint
\frac{(f_2^{\prime }(x)-f_2^{\prime }(x-y))\cos
(y)dy}{|\sin (\frac{y}{2})|}+\varepsilon\mathcal{R}_{2}(\varepsilon ,f).$}}
    \end{cases}
\end{equation}%
where $\mathcal{R}_{j}$, for $j=0,1,2$, are bounded and smooth
and $T_1^-( d ,\vartheta)$ and $S_1$ are defined by
\begin{equation}\label{Sb}
  \frac{S_1}{2}:=  \frac12\sum_{k=1}^{m-1}{ \Big(2\sin\big({\frac{k\pi  }{m}}\big)  \Big)^{-1}},
\end{equation}
\begin{equation}\label{tb}
   {T_1^\pm( d ,\vartheta)}:= \sum_{k=0}^{m-1}\frac{1 -  d ^{\pm 1} \cos\big({\frac{(2k\pm\vartheta)\pi }{m}}\big)}{\big(1+( d ^{\pm 1})^2 - 2 d ^{\pm 1} \cos\big({\frac{(2k\pm\vartheta)\pi }{m}}\big) \big)^{\frac{3}{2}}}
\end{equation}
 Hence, the
proof is completed.
\end{proof}

Next, we obtain the linearization of the functional $\mathcal{F}^1$ and
show that in fact this functional is a $C^{1}$ function.
\begin{proposition}\label{lem2-3b}
	There exists $\varepsilon_0>0$ and a small neighborhood $\Lambda$ of $\lambda^*$  such that  the Gateaux derivatives \( \partial_{f_0} \mathcal{F}^1_0(\varepsilon, f, \lambda) h_0 : (-\varepsilon_0, \varepsilon_0) \times \mathcal{B}_X \times \Lambda \rightarrow \mathcal{Y}^{N-1} \), \( \partial_{f_j} \mathcal{F}^1_j(\varepsilon, f, \lambda) h_j : (-\varepsilon_0, \varepsilon_0) \times \mathcal{B}_X \times \Lambda \rightarrow \mathcal{Y}^{N-1} \) and  \( \partial_{f_\ell} \mathcal{F}^1_j(\varepsilon, f, \lambda) h_\ell : (-\varepsilon_0, \varepsilon_0) \times \mathcal{B}_X \times \Lambda \rightarrow \mathcal{Y}^{N-1} \)  exist and are continuous.
\end{proposition}
\begin{proof}
    The continuity of those functionals are obtained in the same
way as in Proposition \ref{lem2-3} by employing similar arguments as in the case $%
\alpha \in(1,2)$. For this reason, the proof is left to the reader.
\end{proof}
In fact, proceeding as above, we can derive the following Gateaux
derivatives of the functional $\mathcal{F}^1$ as was done in \eqref{2-110}-\eqref{2-14} by  setting $\varepsilon = 0$ and $f_j \equiv 0$, we obtain
\begin{equation}\label{gateauxb}
	\partial_{f_0} \mathcal{F}^1_{0}(0,0,\lambda)h_0=\frac{\gamma_0}{4}\displaystyle\fint \frac{h_0( x- y)\sin( y)d y}{\abs{\sin(\frac{ y}{2})}}-\frac{\gamma_0}{2}\displaystyle\fint\frac{(h'_0( x)-h'_0( x- y))\cos( y)d y}{\abs{\sin(\frac{ y}{2})}} .
\end{equation}
\begin{equation}\label{gateaux2b}
	\partial_{f_j} \mathcal{F}^1_{j}(0,0,\lambda)h_j=\frac{\gamma_j}{4}\displaystyle\fint \frac{h_j( x- y)\sin( y)d y}{\abs{\sin(\frac{ y}{2})}}-\frac{\gamma_j}{2}\displaystyle\fint \frac{(h'_j( x)-h'_j( x- y))\cos( y)d y}{\abs{\sin(\frac{ y}{2})}} .
\end{equation}
\begin{equation}\label{gateaux1bb}
	\partial_{f_j} \mathcal{F}^1_{0}(0,0,\lambda)h_j=0 .
\end{equation}
\begin{equation}\label{gateaux3b}
	\partial_{f_\ell} \mathcal{F}^1_{0}(0,0,\lambda)h_\ell=0 .
\end{equation}
\begin{equation}\label{gateaux4b}
	\partial_{f_0} \mathcal{F}^1_{j}(0,0,\lambda)h_0=0 .
\end{equation}
and
\begin{equation}\label{gateaux5b}
	\partial_{f_\ell} \mathcal{F}^1_{j}(0,0,\lambda)h_\ell=0 .
\end{equation}
  Having established the $C^1$ regularity of the functional.  We then define the following nonlinear operator
\begin{equation*}
\mathcal{F}^1(\varepsilon,f,\lambda):=\big(\mathcal{F}^1_0(\varepsilon,f,\lambda),\mathcal{F}^1_1(\varepsilon,f,\lambda),\mathcal{F}^1_2(\varepsilon,f,\lambda)\big).
\end{equation*}
Now we are in position to show the property below about point vortices.
\begin{proposition}\label{equivalenceb}
  The equation $ \mathcal{F}^1(0, 0, \lambda)=0$  can be reduced to a system of two real equation
  \begin{equation}\label{solb}
{\resizebox{.97\hsize}{!}{$ \Omega d_j-\frac{1}{2d_j^{2}}\left({\gamma_0}+{\gamma_j}\displaystyle\sum_{k=1}^{m-1}{ \Big(2\sin\big({\frac{k\pi  }{m}}\big)  \Big)^{-1}}+\gamma_{3-j}\displaystyle\sum_{k=0}^{m-1}\frac{1 -  \frac{d_{3-j}}{d_j} \cos\big({\frac{(2k\pm\vartheta)\pi }{m}}\big)}{\left(1+\left( \frac{d_{3-j}}{d_j} \right)^2 - 2\frac{d_{3-j}}{d_j} \cos\left({\frac{(2k\pm\vartheta)\pi }{m}}\right) \right)^{\frac{3}{2}}}\right)=0, \, j=1,2.$}}
\end{equation}
\end{proposition}

\begin{proof}
    This result follows directly from the expansions of the functional in \eqref{2-21}, together with the definition of the function $\mathcal{P}_j^1(\lambda)$ provided in \eqref{alg-sysP}.
\end{proof}

\begin{proposition}
\label{criticalb} The system  \eqref{3-7} has a unique solution $\lambda^*=(\Omega^*_1,\gamma_2^*)$ given by
\begin{equation}
\gamma_2^*:=\frac{\big( d ^{3}-1\big){\gamma_0}+ \big(\frac{1}{2}S_1  d^{3} - T_1^-( d ,\vartheta)\big)\gamma_1}{\frac12S_1- T_1^+( d ,\vartheta) d ^{3}}  \label{posb}
\end{equation}%
with angular velocity
\begin{equation}
\Omega^{\ast }_1:=\frac{1}{2(d_1^{3}+d_2^{3})}\Big(\gamma_0+\gamma_1\big(\tfrac12 S_1+ T_1^-( d ,\vartheta)\big)+\gamma_2^*\big(T_1^+( d ,\vartheta)+\tfrac12 S_1\big) \Big),  \label{velb}
\end{equation}
where $S_1$ and $T^\pm_1$ are defined in \eqref{Sb} and \eqref{tb}, respectively and $d=\frac{d_2}{d_1}>0$ .
\end{proposition}
\begin{proof}
By using the system \eqref{2-21} evaluating in $\varepsilon=0$, $f_j=0$ gives
  \begin{equation}\label{p-gG2b}
 \begin{cases}
{\mathcal{F}}_0^{1}(0,0,\Omega ,\lambda)&=0,
\\
{\mathcal{F}}_1^{1}(0,0,\Omega ,\lambda)&=\Omega d_1 \sin(x)-\frac{1}{2d_1^{2}}\Big[\gamma_0+\frac{\gamma_1}{2}S_1 +\gamma_2T_1^+( d ,\vartheta)\Big]\sin(x),
\\
{\mathcal{F}}_2^{1}(0,0,\Omega ,\lambda)&=\Omega d_2 \sin(x) -\frac{1}{2d_2^{2}}\Big[\gamma_0+\gamma_1T_1^-( d ,\vartheta)+\frac{\gamma_2}{2}S_1 \Big]\sin(x).
\end{cases}
\end{equation}
Therefore $\mathcal{F}^1(0,0,\Omega ,\lambda)=0$, if and only if, we have the
following two equalities
\begin{equation}
{\mathcal{F}}_1^{1}(0,0,\Omega ,\lambda)=\Omega d_1 -\frac{1}{2d_1^{2}}\Big[\gamma_0+\frac{\gamma_1}{2}S_1 +\gamma_2T_1^+( d ,\vartheta)\Big]=0  \label{f1bb}
\end{equation}%
and
\begin{equation}
{\mathcal{F}}_2^{1}(0,0,\Omega ,\lambda)=\Omega d_2 -\frac{1}{2d_2^{2}}\Big[\gamma_0+\gamma_1T_1^-( d ,\vartheta)+\frac{\gamma_2}{2}S_1 \Big]=0.  \label{f2bb}
\end{equation}%
After straightforward computations using \eqref{f1bb} one gets
\begin{equation*}
\gamma_2 ^{\ast }:=\frac{\big( d ^{3}-1\big){\gamma_0}+ \big(\frac12S_1  d^{3} - T_1^-( d ,\vartheta)\big)\gamma_1}{\frac12S_1- T_1^+( d ,\vartheta) d ^{3}}.
\end{equation*}%
Now, by replacing $\gamma_2^{\ast }$ into \eqref{f2bb}, we obtain
\begin{equation*}
\Omega^*_1=\frac{1}{2(d_1^{3}+d_2^{3})}\Big(\gamma_0+\gamma_1\big(\tfrac12 S_1+ T_1^-( d ,\vartheta)\big)+\gamma_2^*\big(T_1^+( d ,\vartheta)+\tfrac12 S_1\big) \Big).
\end{equation*}
\end{proof}

Once again in order to  ensure that $\gamma_2$ is non-vanishing, one has to assume that $\gamma_0$ and $\gamma_1$ verify the condition
 \begin{equation}\label{non-degb}
( d ^{3}-1){\gamma_0}+ \big(\tfrac12S_1  d ^{3} - T_1^-( d ,\vartheta)\big)\gamma_1\neq 0.
    \end{equation}

Now, we are in position to prove that $\mathcal{F}^1(\varepsilon,f,\Omega,\gamma_2)$
at the point $(0,0,\Omega^*_1,\gamma_2^*)$ is an isomorphism.
\begin{proposition}
\label{lem2-4b} Let $h=(h_0,h_1,h_2)\in  \mathcal{X}^{N+\log}$ and  $(\dot\Omega,
\dot\gamma_2)\in\mathbb{R}^2$, with $h_{j}(x)=\sum_{n=1}^{\infty
}a_{n}^j\sin (nx)$, $j=0,1,2$, we have the Gateaux derivative is given by
\begin{equation*}
D_{(f,\lambda)}\mathcal{F}^1(0,0,\lambda^*)\begin{pmatrix}
\dot\Omega\\
\dot\gamma_2\\
h
\end{pmatrix}= \begin{pmatrix}
0  & 0\\
d_1 & - \frac{1 }{2}\frac{ T_1^+( d ,\vartheta)}{d_1^{2}}\\
 d_2 &  -\frac{1}{2} \frac{ S_1}{2d_2^{2}}
\end{pmatrix}
\begin{pmatrix} \dot\Omega\\ \dot \gamma_2\end{pmatrix}\,\sin(x)
+\sum_{n=2}^\infty n\sigma_n\begin{pmatrix}
{\gamma_0}\, a_n^0 \\
{\gamma_1}\, a_n^1
\\
\gamma_2^*\, a_n^2
\end{pmatrix}\,\sin(nx),
\end{equation*}
where $\sigma_n$ is given by
\begin{equation}\label{gammab}
\sigma _{n}:=
\frac{2}{\pi}\sum\limits_{i=1}^{n}\frac{1}{2i-1}  ,
\end{equation}
with
\begin{equation*}
g\triangleq (f_{1},f_{2},\Omega ,\gamma_2)\quad \text{and}\quad g_{0}\triangleq
(0,0,\Omega ^{\ast },\gamma_2^{\ast }),
\end{equation*}
Moreover, the linearized operator $D_{g}\mathcal{F}^1(0,g_{0}): \mathcal{X}^{N+\log}\times \mathbb{R}\times \mathbb{%
\ R}\rightarrow \mathcal{Y}^{N-1}$ is an isomorphism if the determinant of the Jacobian satisfies
\begin{equation}\label{det-j-poly:sqg}
\det\big(D_\lambda \mathcal F^1(\lambda)\big)=
-\frac{ d_1}{2d_2^{3}}\Big(\frac{ S_1}{2}- d^{3} T_1^+( d ,\vartheta)\Big)\neq 0 .
\end{equation}
\end{proposition}
\begin{proof}
The proof is similar to Proposition \ref{lem2-4}. In fact, by using \eqref{gateauxb} and \eqref{gateaux2b}, the differential of the nonlinear functional $\mathcal{F}^1$ around $(0,0,0,\Omega,\gamma_2)$ in the direction $h=(h_0,h_{1},h_{2})\in \mathcal{X}^{N+\log }$ has
the following form
\begin{equation*}
D_{f}\mathcal{F}^1(0,0,0,\Omega,\gamma_2)h(x)=%
\begin{pmatrix}
{\gamma _{0}}\sum\limits_{j=2}^{\infty }a_{j}^{0}\sigma _{j}j\sin (jx), \\
{\gamma _{1}}\sum\limits_{j=2}^{\infty }a_{j}^{1}\sigma _{j}j\sin (jx),\\
{\gamma _{2}}\sum\limits_{j=2}^{\infty }a_{j}^{2}\sigma _{j}j\sin (jx),
\end{pmatrix}%
,
\end{equation*}%
where $f=(f_0,f_{1},f_{2})$. According to \cite[Proposition 2.4]{Cas1},
\begin{equation*}
\sigma _{n}=\frac{2}{\pi }\sum\limits_{i=1}^{n}\frac{1}{2i-1}.
\end{equation*}%
An important fact here is that $\{\sigma _{n}\}$ is increasing with respect
to $n$ and its  asymptotic behavior is $\sigma _{n}\sim \ln (n)$ if $\alpha =1$.
From \eqref{2-21}, one knows that
\begin{equation}\label{fub}
    \begin{cases}
&\mathcal{F}^1_{0}(0,0,0,\lambda)=0,\\
        &
        {\mathcal{F}}_1^1(0,0,0,\lambda)=\Omega d_1\sin (x)-\frac{1}{2d_1^{2}}\Big[\gamma_0+\frac{\gamma_1}{2}S_1 +\gamma_2 T_1^+( d ,\vartheta)\Big]\sin(x),
\\
&{\mathcal{F}}_2^1(0,0,0,\lambda)=\Omega d_2\sin (x) -\frac{1}{2d_2^{2}}\Big[\gamma_0+\gamma_1 T_1^-( d ,\vartheta)+\frac{\gamma_2}{2}S_1 \Big]\sin(x).
    \end{cases}
\end{equation}%
Differentiating \eqref{fub} with respect to  $\Omega $
around the point $(\Omega^*_1,\gamma_2^*)$ yields
\begin{equation*}
\partial _{\Omega }\mathcal{F}_{0}^1(0,0,0,\Omega^*_1,\gamma_2^*)=0,
\end{equation*}%
\begin{equation*}
\partial _{\Omega }\mathcal{F}_{1}^1(0,0,0,\Omega^*_1,\gamma_2^*)=d_1\sin(x),
\end{equation*}%
\begin{equation*}
\partial _{\Omega }\mathcal{F}_{2}^1(0,0,0,\Omega^*_1,\gamma_2^*)=d_2\sin(x),
\end{equation*}%
Next, differentiating \eqref{fub} with respect to $\gamma_2$ at the point $%
(\Omega^*_1,\gamma_2^*),$ we get
\begin{equation*}
\partial _{\gamma_2}\mathcal{F}_{0}^1(0,0,0,\Omega^*_1,\gamma_2^*)=0,
\end{equation*}%
\begin{equation*}
\partial _{\gamma_2}\mathcal{F}_{1}^1(0,0,0,\Omega^*_1,\gamma_2^*)=-\frac{1}{2d_1^{2}}T_1^+( d ,\vartheta)\sin(x),
\end{equation*}%
\begin{equation*}
\partial _{\gamma_2}\mathcal{F}_{1}^1(0,0,0,\Omega^*_1,\gamma_2^*)=-\frac{1}{4d_2^{2}}S_1\sin(x),
\end{equation*}%
Therefore, for all $(\dot\Omega,\dot\gamma_{2})\in \mathbb{R}^{2}$, the differential of the mapping ${\mathcal{F}}^1:=(\mathcal{F}_0^1,{\mathcal{F}}_1^1,{\mathcal{F}}_2^1)$ with respect to $\lambda=(\Omega,\gamma_2)$ on the direction $(\dot\Omega,\dot\gamma_2)$ is given by
\begin{equation*}
D_{\lambda}{\mathcal{F}}^1(\lambda^*)\begin{pmatrix} \dot\Omega\\ \dot \gamma_2\end{pmatrix} =\begin{pmatrix}
d_1 &  -\frac{1 }{2}\frac{ T_\alpha^+( d ,\vartheta)}{d_1^{2}}\\
 d_2 &  -\frac{1}{2} \frac{ S_\alpha}{2d_2^{2}}
\end{pmatrix}
\begin{pmatrix} \dot\Omega\\ \dot \gamma_2\end{pmatrix}\sin(x).
\end{equation*}


Now, we obtain the explicit expression of the inverse of $D_{(f,\lambda)}\mathcal{F}^1(0,0,\lambda^*)$, namely $D_{(f,\lambda)}\mathcal{F}^1(0,0,\lambda^*)^{-1}$. Let $h=(h_0,h_1,h_2)\in  \mathcal{X}^{N+\log}$ and  $(\dot\Omega,
\dot\gamma_2)\in\mathbb{R}^2$, with
$$h_{j}(x)=\sum_{n=1}^{\infty
}a_{n}^j\sin (nx), \quad j=0,1,2,$$
 we have the Gateaux derivative is given by
\begin{equation*}
D_{(f,\lambda)}\mathcal{F}^1(0,0,\lambda^*)\begin{pmatrix}
\dot\Omega\\
\dot\gamma_2\\
h
\end{pmatrix}= \begin{pmatrix}
0  & 0\\
d_1 &  -\frac{1}{2}\frac{ T_1^+( d ,\vartheta)}{d_1^{2}}\\
 d_2 & - \frac{1 }{2} \frac{ S_1}{2d_2^{2}}
\end{pmatrix}
\begin{pmatrix} \dot\Omega\\ \dot \gamma_2\end{pmatrix}\,\sin(x)
+\sum_{n=2}^\infty n\sigma_n\begin{pmatrix}
{\gamma_0}\, a_n^0 \\
{\gamma_1}\, a_n^1
\\
\gamma_2^*\, a_n^2
\end{pmatrix}\,\sin(nx).
\end{equation*}
Let $g\in \mathcal{Y}^{N-1}$ satisfying the following expansion
\begin{equation*}
g(x)=\sum_{n=1}^{\infty }%
\begin{pmatrix}
A_n^0  \\
A_n^1\\
A_n^2
\end{pmatrix}%
\sin (nx).
\end{equation*}%
with $A_n^0=0$ if $n$ is not a multiple of $m$,  we can easily obtain the inverse of the linearized operator $%
D_{(f,\lambda)}\mathcal{F}^1(0,0,\lambda^*)g(x)$ which has the following expression
\begin{equation*}
    \begin{split}
     &D_{(f,\lambda)}\mathcal{F}^1(0,0,\lambda^*)^{-1}g(x)=\\
     &
    {\resizebox{.98\hsize}{!}{$ \left(
\frac{1}{\gamma_0}\displaystyle\sum_{n=2}^\infty\frac{A_n^0}{n\sigma_n}\, \cos(nx),
\frac{1}{\gamma_1}\displaystyle\sum_{n=2}^\infty\frac{A_n^1}{n\sigma_n}\,\cos(nx),
\frac{1}{\gamma_2^*}\displaystyle\sum_{n=2}^\infty\frac{A_n^2}{n\sigma_n}\, \cos(nx)\, , \frac{1 }{2d_2^{2}} \frac{ d ^{2}T_1^+( d ,\vartheta)A_0^2-\frac12{S_1}A_0^1}{\det\big(D_\lambda \mathcal P^1(\lambda^*)\big)},\frac{A_0^2d_1-A_0^1d_2}{\det\big(D_\lambda \mathcal P^1(\lambda^*)\big)}
\right)   $}}
    \end{split}
\end{equation*}
where $\det\big(D_\lambda \mathcal P^1(\lambda^*)\big)$ was calculated in \eqref{det-j-poly}.
Now we are going to prove that the linearization $D_{(f,\lambda)}\mathcal{F}^1(0,0,\lambda^*)$ is an isomorphism from $ \mathcal{X}^{N+\log}\times \mathbb{\ R}
\times \mathbb{\ R}$  to $\mathcal{Y}^{N-1}$. From Proposition \ref{lem2-3b}
and \eqref{det-j-poly:sqg}, it is obvious that $D_{(f,\lambda)}\mathcal{F}^1(0,0,\lambda^*)$ is an isomorphism from $ \mathcal{X}^{N+\log }\times \mathbb{\ R}
\times \mathbb{\ R}$  into $\mathcal{Y}^{N-1}$. Hence
only the invertibility needs to be considered. In fact, the restricted
linear operator $D_{(f,\lambda)}\mathcal{F}^1(0,0,\lambda^*)$  is invertible if and only if the Jacobian determinant  in \eqref{det-j-poly:sqg} is  non-vanishing. Thus the desired result is obtained.
\end{proof}

\begin{theorem}\label{thm:polygon2}
  Let  $b_1,b_2,d_1,d_2\in (0,\infty)$ such that $d=d_2/d_1>0$ satisfies \eqref{det-j-poly:sqg}, and let $\gamma_0,\gamma_1\in \R\setminus\{0\}$ such that \eqref{non-degb} holds. Then
\begin{enumerate}[label=\rm(\roman*)]
\item There exists $\varepsilon_0 > 0$ and a neighborhood $\Lambda$ of $\lambda^*$ in $\R^2$ such that  $\mathcal{F}^1a$
can be extended to a $C^1$ mapping $(-\varepsilon_0,\varepsilon_0)\times  \mathcal{B}_X\times \Lambda \to \mathcal{Y}^{N-1}$.
\item $\mathcal{F}^1(0,0,\lambda^*)=0,
$
    where $\lambda^*=(\Omega^*_1,\gamma_2^*)$ is given by \eqref{posb} and \eqref{velb}.
\item  The linear operator
$D_{(f,\lambda)}\mathcal{F}^1(0,0,\lambda^*)\colon   \mathcal{X}^{N+\log}\times\R^2 \to  \mathcal{Y}^{N-1}$ is an isomorphism.
\item There exists $\varepsilon_1>0$ and a unique $C^1$ function $(f,\lambda)\colon (-\varepsilon_1,\varepsilon_1)\to   \mathcal{B}_X\times\R^2$ such that
\begin{equation}\label{gf1f2-polyb}
\mathcal{F}^1 \big(\varepsilon, f(\varepsilon),\lambda(\varepsilon)\big)=0,
\end{equation}
with $\lambda(\varepsilon)=\lambda^*+o(\varepsilon)$ and
\begin{gather}
f(\varepsilon)=(f_0(\varepsilon),f_1(\varepsilon),f_2(\varepsilon))=\frac{8}{3\pi}\,\Big(0,\frac{\varepsilon b_1\mathcal{H}_1^1}{\gamma_1d_1^{3}}\sin(x) ,\frac{\varepsilon b_2\mathcal{H}_2^1}{\gamma_2^* d_2^{3}}\sin(x)\Big)+o(\varepsilon),
\notag \\
\label{Qjalphab}
{\resizebox{.94\hsize}{!}{$\mathcal{H}_j^1:=
\gamma_0+ \displaystyle\sum_{\ell=1}^{2} \gamma_\ell\displaystyle\sum_{k=\delta_{\ell j}}^{m-1}\fint \frac{2d_jd_\ell\cos(2k\pi  /m-(\delta_{2\ell}-\delta_{2j})\vartheta\pi /m) -d_j^2-d_\ell^{2}\cos(2(2k\pi  /m-(\delta_{2\ell}-\delta_{2j})) }{(d_\ell^{2}+2d_jd_\ell\cos(2k\pi  /m-(\delta_{2\ell}-\delta_{2j})\vartheta\pi /m)-d_j^{2})^{\frac{5}{2}}},$}}
\end{gather}
\item For all  $\varepsilon\in (-\varepsilon_1,\varepsilon_1)\setminus\{0\}$,  the domains $\mathcal{O}_j^\varepsilon$, whose boundaries are given by the conformal parametrizations $R_j^\varepsilon(x)=1+\varepsilon|\varepsilon| b_j^{2} f_j(x)\colon\mathbb{T}\to \partial \mathcal{O}_j^\varepsilon$,   are
 strictly convex.
\end{enumerate}
\end{theorem}
\begin{proof}
The case \( \alpha = 1 \) can be treated in a similar manner as the case \( \alpha \in (1,2) \) discussed in Theorem~\ref{thm:polygon}.
 The regularity of the nonlinear operator $\mathcal{F}^1 $ follows from Propositions~\ref{p1b} and \ref{lem2-3b}. The reflection symmetry property is also obtained for $f_j$ even function
\begin{equation*}
\mathcal{F}_j^1 (\varepsilon,f,\lambda)(-x)=-\mathcal{F}_j^1 (\varepsilon,f,\lambda)(x).
\end{equation*}
Thus, it remains to check the $m$-fold  symmetry  property of $\mathcal{F}_0^1 $, namely, that if
\begin{equation}\label{m-fold}
{f_0(Q_{\frac{2\pi }{m}}x)}=Q_{\frac{2\pi }{m}} f_0(x), \quad \forall x\in \mathbb{T},
\end{equation}
then
\begin{equation*}
\mathcal{F}_0^1 (\varepsilon,f,\lambda)(Q_{\frac{2\pi }{m}}{x})=\mathcal{F}_0^1 (\varepsilon,f,\lambda)({x}), \quad \forall x\in \mathbb{T}.
\end{equation*}
From \eqref{f01}  and \eqref{m-fold} one has
\begin{equation*}
\begin{split}
& \mathcal{F}^1_{01}=\Omega  |\varepsilon |^{3 }b_{0}^{3
}f_{0}^{\prime }(x) ,
\end{split}
\end{equation*}
Same computations to the derivation of  \eqref{f02}, using  \eqref{m-fold} and the fact the rotation matrix $Q_{\frac{2\pi }{m}}$ preserves length and area one has  one gets
\begin{equation}
\begin{split}
&{\resizebox{.98\hsize}{!}{$\mathcal{F}^1_{02}(\varepsilon,f_0,\lambda)(Q_{\frac{2\pi }{m}}x)=	\frac{\gamma _0}{4}\displaystyle\fint\frac{f_0(Q_{\frac{2\pi }{m}}(x-y))\sin
(y)dy}{|\sin (\frac{y}{2})|}-\frac{\gamma _0}{2}\int
\!\!\!\!\!\!\!\!\!\;{}-{}\frac{(f_0^{\prime }(Q_{\frac{2\pi }{m}}x)-f_0^{\prime }(Q_{\frac{2\pi }{m}}(x-y)))\cos
(y)dy}{|\sin (\frac{y}{2})|}+\varepsilon |\varepsilon |\mathcal{R}%
_{02}(\varepsilon ,f_0),$}}\\
&
\qquad\qquad\qquad{\resizebox{.84\hsize}{!}{$= \frac{\gamma _0}{4}\displaystyle\fint\frac{f_0(x-y)\sin
(y)dy}{|\sin (\frac{y}{2})|}-\frac{\gamma_0}{4} \int
\!\!\!\!\!\!\!\!\!\;{}-{}\frac{(f_0^{\prime }(x)-f_0^{\prime }(x-y))\cos
(y)dy}{|\sin (\frac{y}{2})|}+\varepsilon |\varepsilon |\mathcal{R}%
_{02}(\varepsilon ,f_0),$}}\\
&
\qquad\qquad\qquad =\mathcal{F}^1_{02}(\varepsilon,f_0,\lambda)(x).
\end{split}
\end{equation}%
Similarly, we  have
\begin{equation*}
\begin{split}
\mathcal{F}^1_{03}(\varepsilon,f_0,\lambda)(Q_{\frac{2\pi }{m}}x)=\mathcal{F}^1_{03}(\varepsilon,f_0,\lambda)(x) .
\end{split}
\end{equation*}%
Then, by  \eqref{m-fold}, we also deduce that
\begin{equation*}
\mathcal{F}^1_0(\varepsilon, f_0)(Q_{\frac{2\pi }{m}}x)=\mathcal{F}^1_0(\varepsilon, f_0)(x).
\end{equation*}
Then, \eqref{m-fold} is satisfied. These concludes the proof of {\rm(i)}.
 The proof of {\rm(ii)} follows immediately from Proposition~\ref{lem2-4b}, \eqref{p-gG2b}, \eqref{posb} and \eqref{velb}.
In order to show {\rm(iii)} we  use  Proposition~\ref{equivalenceb}  and  \eqref{det-j-poly:sqg} to get,
 for all $h=(h_0,h_1,h_2)\in  \mathcal{X}^{N+\log}$ and  $(\dot\Omega,
\dot\gamma_2)\in\mathbb{R}^2$,
\begin{equation}\label{eq:devb}
D_{(f,\lambda)}\mathcal{F}^1(0,0,\lambda^*)\begin{pmatrix}
\dot\Omega\\
\dot\gamma_2\\
h
\end{pmatrix}= \begin{pmatrix}
0  & 0\\
d_1 & - \frac{1 }{2}\frac{ T_1^+( d ,\vartheta)}{d_1^2}\\
 d_2 & - \frac{1 }{2} \frac{ S_1}{2d_2^2}
\end{pmatrix}
\begin{pmatrix} \dot\Omega\\ \dot \gamma_2\end{pmatrix}\,\sin(x)
+\sum_{n=2}^\infty n\sigma_n\begin{pmatrix}
{\gamma_0}\, a_n^0 \\
{\gamma_1}\, a_n^1
\\
\gamma_2^*\, a_n^2
\end{pmatrix}\,\sin(nx),
\end{equation}
where $\sigma_n$ is given by \eqref{gammab}.
Proposition~\ref{lem2-4b} and the assumption \eqref{det-j-poly:sqg} then imply {\rm(iii)}.

\vspace{0.2cm}

The existence and uniqueness in {\rm(iv)} follow from the implicit function theorem. In order to compute the asymptotic behavior of the solution, we   differentiate \eqref{gf1f2-polyb} with respect to \( \varepsilon \) at the point \( (0, 0, \lambda^*) \) to obtain the following formula
\begin{equation}\label{compb}
\partial_\varepsilon \big(f(\varepsilon),\lambda(\varepsilon)\big)\big|_{\varepsilon=0}=-D_{(f,\lambda)} \mathcal{F}^1\big(0,0,\lambda^*)^{-1}\partial_\varepsilon \mathcal{F}^1(0,0,\lambda^*).
\end{equation}
For any $g\in \mathcal{Y}^{N-1}$ with the expansion
\begin{equation*}
g(x)=\sum_{n=1}^{\infty }%
\begin{pmatrix}
A_n^0  \\
A_n^1\\
A_n^2
\end{pmatrix}%
\sin (nx).
\end{equation*}%
with $A_n^0=0$ if $n$ is not a multiple of $m$, we have
\begin{equation}\label{eq:invb}
    \begin{split}
     &D_{(f,\lambda)}\mathcal{F}^1(0,0,\lambda^*)^{-1}g(x)=\\
     &
    {\resizebox{.98\hsize}{!}{$ \left(
\frac{1}{\gamma_0}\displaystyle\sum_{n=2}^\infty\frac{A_n^0}{n\sigma_n}\, \cos(nx),
\frac{1}{\gamma_1}\displaystyle\sum_{n=2}^\infty\frac{A_n^1}{n\sigma_n}\,\cos(nx),
\frac{1}{\gamma_2^*}\displaystyle\sum_{n=2}^\infty\frac{A_n^2}{n\sigma_n}\, \cos(nx)\, , \frac{1}{2d_2^2} \frac{ d ^{2}T_1^+( d ,\vartheta)A_0^2-\frac12{S_1}A_0^1}{\det\big(D_\lambda \mathcal P^1(\lambda^*)\big)},\frac{A_0^2d_1-A_0^1d_2}{\det\big(D_\lambda \mathcal P^1(\lambda^*)\big)}
\right)   $}}
    \end{split}
\end{equation}
where $\det\big(D_\lambda \mathcal P^1(\lambda^*)\big)$ was calculated in \eqref{det-j-poly:sqg}.
On the other hand, from   \eqref{2-21} we have
\begin{equation}\label{eq:fb}
 \begin{cases}
{\mathcal{F}}_0^{1}(\varepsilon,0,\lambda)&=\varepsilon\mathcal{R}_{0}(\varepsilon ,f),
\\
{\mathcal{F}}_1^{1}(\varepsilon,0,\lambda)&=\Omega d_1 \sin(x)-\frac{1}{2d_1^{2}}\Big[\gamma_0+\frac{\gamma_1}{2}S_1 +\gamma_2T_1^+( d ,\vartheta)\Big]\sin(x)+\varepsilon\mathcal{R}_{0}(\varepsilon ,f),
\\
{\mathcal{F}}_2^{1}(\varepsilon,0,\lambda)&=\Omega d_2 \sin(x) -\frac{1}{2d_2^{2}}\Big[\gamma_0+\gamma_1T_1^-( d ,\vartheta)+\frac{\gamma_2}{2}S_1 \Big]\sin(x)+\varepsilon\mathcal{R}_{0}(\varepsilon ,f).
\end{cases}
\end{equation}
Similarly to the derivation of the equations \eqref{eq:f03}-\eqref{eq:fj4} one gets
\[{\mathcal{F}}_{031}^{1}=\varepsilon\mathcal{R}_{031}(\varepsilon,f_\ell,f_0)\]
\begin{equation*}
   {\mathcal{F}}_{j31}^{1}=  -\frac{3 \gamma_0}{d_j^{3}} \varepsilon b_j\cos(x)\sin(x)+\varepsilon\mathcal{R}_{j31}(\varepsilon,f_0,f_j) .
\end{equation*}
\begin{equation*}
    \begin{split}
       {\mathcal{F}}_{j41}^{1} &=-\frac{1}{2}\sum_{\ell=1}^{2} \gamma_\ell\displaystyle\sum_{k=\delta_{\ell j}}^{m-1}\frac{1}{ b_\ell  } \displaystyle\fint \frac{B_{k\ell j}\sin(x-y+2k\pi  /m-(\delta_{2\ell}-\delta_{2j})\vartheta\pi /m)dy}{\left( A_{k\ell j}\right)^{\frac{3}{2}}}\\
&
 {\resizebox{.94\hsize}{!}{$\qquad+\frac{3 }{4}\displaystyle\sum_{\ell=1}^{2} \gamma_\ell\displaystyle\sum_{k=\delta_{\ell j}}^{m-1}\frac{1}{ b_\ell  } \fint \frac{\varepsilon B_{k\ell j}^2\sin(x-y+2k\pi  /m-(\delta_{2\ell}-\delta_{2j})\vartheta\pi /m)dy}{\left(A_{k\ell j}\right)^{\frac{5}{2}}}+\varepsilon\mathcal{R}_{j41}(\varepsilon,f_\ell,f_j)$}},
    \end{split}
\end{equation*}
Thus, we can conclude that
\begin{equation}
\label{f0dif-epsb}
 {\resizebox{.96\hsize}{!}{$\partial_\varepsilon \mathcal{F}^1_j(0, 0, \lambda^*)(x) = 3\displaystyle\sum_{\ell=1}^{2} \gamma_\ell\displaystyle\sum_{k=\delta_{\ell j}}^{m-1}b_j\fint\frac{2d_jd_\ell\cos(2k\pi  /m-(\delta_{2\ell}-\delta_{2j})\vartheta\pi /m) -d_j^2-d_\ell^{2}\cos(2(2k\pi  /m-(\delta_{2\ell}-\delta_{2j}))  }{(d_\ell^{2}+2d_jd_\ell\cos(2k\pi  /m-(\delta_{2\ell}-\delta_{2j})\vartheta\pi /m)-d_j^{2})^{\frac{5}{2}}} \sin(x)\cos(x).$}}
\end{equation}
Hence, for $\alpha=1$ we conclude that
\[
\partial_\varepsilon \mathcal{F}^1_j(0, 0, \lambda^*) \in \mathcal{Y}_0^{N-1}.
\]
Given that the linear operator
\( D_{f} \mathcal{F}^1(0, 0, \lambda^*) \colon \mathcal{X}^{N+\log} \to \mathcal{Y}_0^{N-1} \) is an isomorphism and, according to the hypothesis, the kernel of the Jacobian operator $D_{\lambda}{\mathcal{F}}^{1}(\lambda)$ is non-trivial, then we can combine \eqref{eq:devb}, \eqref{compb}, \eqref{f0dif-epsb}, and Proposition~\ref{lem2-4b} to deduce that
\[
\partial_\varepsilon \lambda(\varepsilon) \big|_{\varepsilon=0} = 0 ,
\]
and
\begin{equation}\label{eq:deriv}
{\resizebox{.98\hsize}{!}{$    \partial_\varepsilon f_{j}(\varepsilon) \big|_{\varepsilon=0}(x) = \frac{3}{\sigma_2 \gamma_j} \displaystyle\sum_{\ell=1}^{2} \gamma_\ell\displaystyle\sum_{k=\delta_{\ell j}}^{m-1}b_\ell\fint \frac{2d_jd_\ell\cos(2k\pi  /m-(\delta_{2\ell}-\delta_{2j})\vartheta\pi /m) -d_j^2-d_\ell^{2}\cos(2(2k\pi  /m-(\delta_{2\ell}-\delta_{2j})) }{(d_\ell^{2}+2d_jd_\ell\cos(2k\pi  /m-(\delta_{2\ell}-\delta_{2j})\vartheta\pi /m)-d_j^{2})^{\frac{\alpha}{2}+2}}\sin(x).$}}
\end{equation}
Finally, simple calculations using \eqref{gammab} yields  and combining the two last identities with \eqref{compb}, \eqref{eq:invb} and \eqref{eq:deriv} yields
\begin{equation*}
\partial_\varepsilon \big( f(\varepsilon),\lambda(\varepsilon)\big)\big|_{\varepsilon=0}=\frac{8}{3\pi}\Big(0,\frac{b_1\mathcal{H}_1^0}{\gamma_1d_1^{3}}\sin(x),\frac{b_2\mathcal{H}_2^0}{\gamma_2^* d_2^{3}}\sin(x)\, ,0,0 \Big).
\end{equation*}
Thus, we conclude  the proof of {\rm (iv)}.

The convexity in {\rm (v)} is an straightforward computation. Hence, we omit the proof. Thus the
proof of Theorem \ref{thm:polygon2} is completed.
\end{proof}

\section*{Acknowledgments}
\noindent  EC  was supported by FAPESP through grant 2021/10769-6, Brazil. LCFF was supported by CNPq through grant 312484/2023-2, Brazil.

\thispagestyle{empty} \vspace{0.5cm}

\noindent\textbf{Conflict of interest statement:} The authors declare that
they have no conflict of interest.

\

\noindent\textbf{Data availability statement:} This manuscript has no
associated data.

\vspace{0.3cm}

\addcontentsline{toc}{section}{References}


\phantom{s} \thispagestyle{empty}

\end{document}